\documentclass[11pt,a4paper]{amsart}

\usepackage{a4wide}
\usepackage{amsfonts,amsthm,amsmath,amssymb}
\usepackage{amscd,color}
\usepackage{psfrag,graphicx,subfigure}
\usepackage{import} %% Figures with labels
\theoremstyle{plain}
\newtheorem{lem}{Lemma}[section]
\newtheorem{prop}[lem]{Proposition}
\newtheorem{thm}[lem]{Theorem}

\newtheorem*{thm*}{Main Theorem}
\newtheorem*{thmA}{Theorem A}
\newtheorem*{thmB}{Theorem B}
\newtheorem*{thmC}{Theorem C}

\newtheorem*{cor*}{Corollary}
\newtheorem*{cora}{Corollary A'}
\newtheorem*{corb}{Corollary B'}

\newtheorem{cor}[lem]{Corollary}

\theoremstyle{definition}

\newtheorem*{defn*}{Definition}

\newtheorem*{ex*}{Example}

\newtheorem*{rem*}{Remark}

\theoremstyle{remark}
\DeclareMathOperator{\dist}{dist}

\DeclareMathOperator{\id}{id}
\newcommand{\C}{\mathbb C}
\newcommand{\D}{\mathbb D}
\newcommand{\clC}{\widehat{\C}}
\newcommand{\chat}{\widehat{\C}}

\newcommand{\clH}{\overline \H}
\newcommand{\R}{\mathbb R}

\newcommand{\N}{\mathbb N}
\newcommand{\Q}{\mathbb Q}

\newcommand{\DD}{\mathcal D}
\newcommand{\NN}{\mathcal N}

\newcommand{\bd}{\partial}
\renewcommand{\Re}{\textup{Re}}

\renewcommand{\H}{\mathbb H}
\newcommand{\HH}{\mathbb H}

\newcommand{\Ch}{\ensuremath{\widehat{\mathbb C}}}

\newcommand{\inter}{{\rm int}}
\newcommand{\ext}{{\rm ext}}
\newcommand{\inv}{^{-1}}

\renewcommand{\ss}{\scriptsize}
\begin{document}

\title[On the connectivity of the Julia sets]{On the connectivity of the Julia sets of meromorphic functions}

\date{\today}

\author{Krzysztof Bara\'nski}
\address{Institute of Mathematics, University of Warsaw,
ul.~Banacha~2, 02-097 Warszawa, Poland}
\email{baranski@mimuw.edu.pl}

\author{N\'uria Fagella}
\address{Departament de Matem\`atica Aplicada i An\`alisi,
Universitat de Barcelona, 08007 Barcelona, Catalonia, Spain}
\email{fagella@maia.ub.es}

\author{Xavier Jarque}
\address{Departament de Matem\`atica Aplicada i An\`alisi,
Universitat de Barcelona, 08007 Barcelona, Catalonia, Spain}
\email{xavier.jarque@ub.edu}

\author{Bogus{\l}awa Karpi\'nska}
\address{Faculty of Mathematics and Information Science, Warsaw
University of Technology, Pl.~Politechniki~1, 00-661 Warszawa, Poland}
\email{bkarpin@mini.pw.edu.pl}

\thanks{The first author is partially supported by
Polish MNiSW Grant N N201 607940. The second and third authors were partially supported by the Catalan grant 2009SGR-792, and by the Spanish grants MTM-2006-05849 and MTM-2008-01486 Consolider (including a FEDER contribution) and MTM2011-26995-C02-02. The fourth author is partially supported  
by Polish MNiSW Grant N N201 607940 and Polish PW Grant 504G 1120 0011 000.}
\subjclass[2000]{Primary 30D05, 37F10, 30D20}

\bibliographystyle{plain}

\begin{abstract} We prove that every transcendental meromorphic map $f$ with a disconnected Julia set has a weakly repelling fixed point. This implies that the Julia set of Newton's method  for finding zeroes of an entire map is connected.  
Moreover, extending a result of Cowen for holomorphic self-maps of the disc, we show the existence of absorbing domains for holomorphic self-maps of hyperbolic regions whose iterates tend to a boundary point. In particular, the results imply that periodic Baker domains of Newton's method for entire maps are simply connected, which solves a well-known open question.
\end{abstract}
\maketitle

\section{Introduction}

Let $f:\C \to \chat$ be a holomorphic map. If the point at infinity is an essential singularity of~$f$, then we call it a transcendental meromorphic map; otherwise $f$ is rational. We consider the dynamical system given by the iterates of $f$, which induces a dynamical partition of the complex plane into two completely invariant sets: the {\em Fatou set} $F(f)$, which is the set of points $z\in\C$, where the family of iterates $\{f^n\}_{n\geq 0}$ is defined and normal in some neighborhood of $z$, and its complement, the {\em Julia set} $J(f)$. The Fatou set is open and consists of points with, in some sense, stable dynamics, while the Julia set is closed and its points exhibit chaotic behavior. If $f$ is transcendental meromorphic (unless $f$ has a unique omitted pole), then the Julia is the closure of the set of all prepoles of $f$ and both  Julia and Fatou sets are unbounded. In any case, $J(f)$ is the closure of the set of repelling periodic points of $f$ (see~\cite{bkl1}). For general background on the 
dynamics of rational and meromorphic maps we refer to \cite{bergweiler,carlesongamelin,milnor}. 

Connected components of the Fatou set, known as {\em Fatou components}, are mapped by $f$ among themselves. A Fatou component $U$ is {\em periodic} of period $p$, or {\em $p$-periodic}, if $f^p(U) \subset U$; a component which is not eventually periodic is called {\em wandering}. Unlike the rational case \cite{sullivan}, transcendental meromorphic maps may have wandering components. There is a complete classification of periodic Fatou components: such a component can either be a rotation domain (Siegel disc or Herman ring), the basin of attraction of an attracting or parabolic periodic point or a {\em Baker domain} (the latter possibility can occur only for transcendental maps). Recall that a $p$-periodic Fatou component $U\subset \C$ is a Baker domain, if $f^{pn}$ on $U$ tend to a point $\zeta$ in the boundary of $U$ as $n \to \infty$, and $f^j(\zeta)$ is not defined for some $j \in\{0,\ldots p-1\}$. This implies the existence of an unbounded Fatou component $U'$ in the same cycle, such that $f^{pn}\to \infty$ on $U'$. 
The first example of a Baker domain was given by Fatou \cite{fatou2}, who considered the function $f(z)=z+1+e^{-z}$ and showed that the right half-plane is contained in an invariant Baker domain. If $f$ is an entire function, then all its Baker domains (and other periodic Fatou components) must be simply connected \cite{baker75}. In the case of meromorphic maps, Baker domains are, in general, multiply connected, as shown in examples by Dominguez \cite{dominguez} and K\"onig \cite{konig}. There are a number of papers studying dynamical properties
of Baker domains, see e.g.~\cite{barfag, faghen,henriksen} for the entire case and \cite{berdrasin,rippon,ripponstallard} for the meromorphic one.

In this paper we study the relation of the connectivity of the Julia set and the existence of weakly repelling fixed points for meromorphic maps. We say that a fixed point $z_0$ of a holomorphic map $f$ is {\em weakly repelling}, if $|f'(z_0)| > 1$ or $f'(z_0) = 1$. It was proved by Fatou~\cite{fatou2} that a rational map of degree greater than one has at least one weakly repelling fixed point in $\clC$. In 1990, Shishikura \cite{shishikura} proved a remarkable result, showing that if $f$ is rational and its Julia is disconnected, then $f$ has at least two weakly repelling fixed points in $\clC$. For transcendental meromorphic maps the situation is more complicated, since they need not have fixed points at all. However, the point at infinity can be treated as an additional ``fixed point''. 

In this paper we prove the following result.

\begin{thm*}
Let $f$ be a transcendental meromorphic function with a disconnected Julia set. 
Then $f$ has at least one weakly repelling fixed point.
\end{thm*}

An important motivation for this theorem is the question of the connectivity of Julia sets of the celebrated Newton's method 
\[
N_g(z)=z - \frac{g(z)}{g'(z)}
\]
of finding zeroes of an entire map $g: \C \to \C$. The dynamical properties of  Newton's method, especially for polynomials $g$, were studied in a number of papers, see e.g.~\cite{hubbardschleicher,manning,mayer,mcmullen1,mcmullen2,przytycki,johannes,tanlei}. Notice that the map $N_g$ is meromorphic, its fixed points in $\C$ are, precisely, zeroes of $g$, and all of them are attracting. For a polynomial $g$, the map $N_g$ is rational and the point at infinity is a repelling fixed point, while for transcendental entire $g$, its Newton's method is transcendental meromorphic. Hence,  Shishikura's result shows that for polynomials $g$, the Julia set of $N_g$ is connected. Our theorem immediately implies the following corollary.

\begin{cor*}
If $g$ is an entire map and $N_g$ is its Newton's method, then $J(N_g)$ is connected.
\end{cor*}

Since the Julia set is closed, it is connected if and only if all the Fatou components are simply connected. Therefore, the proof of the Main Theorem  splits into several cases -- for each type of the Fatou component one should show that if it is multiply connected, then the map has a weakly repelling fixed point. However, Shishikura's proofs in the rational case cannot be directly extended to the transcendental case, because of the appearance of new phenomena such as lack of compactness, presence of asymptotic values and new types of  Fatou components. 

For transcendental meromorphic maps, the case of wandering domains was solved by Bergweiler and Terglane in \cite{berter}, while the case of attracting or parabolic cycles and preperiodic components were dealt with by Fagella, Jarque and Taix\'es in \cite{FJT1,FJT2}. Therefore, the remaining cases were Baker domains and Herman rings, which are the subject of the present work.

The known proofs for a $p$-periodic Fatou component $U$, such that $f^{pn} \to \zeta$ on $U$ as $n\to\infty$ (i.e.~when $U$ is the basin of attraction of an attracting or parabolic periodic point), are based on the existence of a simply connected domain $W \subset U$, which is absorbing for $F = f^p$ and tends to $\zeta$ under iterations of $F$.

\begin{defn*}[\bf Absorbing domain] Let $U$ be a domain in $\C$ and let $F: U \to U$ be a holomorphic map. An $F$-invariant domain $W \subset U$ is {\em absorbing} for $F$, if for every compact set $K \subset U$ there exists $n=n(K) \ge 0$, such that $F^n(K) \subset W$. 
\end{defn*}

The problem of existence of suitable absorbing domains has a long history. For attracting and parabolic basins it is  part of the classical problem of studying the  local behavior of an analytic map near a fixed point. In particular, if $U$ is the basin of a (super)attracting $p$-periodic point $\zeta$, then $F = f^p$ is conformally conjugate to $z\mapsto F'(\zeta) z$ (if $F'(\zeta)\neq 0$) or $z\mapsto z^k$ for some $k\in\N$ (if $F'(\zeta)= 0$) near $z=0$. In this case, if we take $W$ to be the image of a small disc centered at $z=0$ under the conjugating map, then $W$ is a simply connected absorbing domain for $F$ and $\bigcap_{n\geq 0} F^n(W) = \{\zeta\}$. Likewise, if $U$ is a basin of a parabolic $p$-periodic point, an attracting petal in $U$ would provide a similar example. 

The existence of such absorbing regions in  Baker domains was an open question, and one of the main obstacles for the completion of the proof of the Main Theorem.  In this paper we prove that we can always construct suitable absorbing regions in  Baker domains, if we drop the condition of simple connectedness. This is a corollary of the following more general theorem, which we prove in Section~\ref{sec:proof_thmA}.
(By $\DD_U(z, r)$ we denote the disc of radius $r$, centered at $z \in U$, with
respect to the hyperbolic metric in $U$.)

\begin{thmA}[\bf Existence of absorbing regions for holomorphic self-maps of hyperbolic domains]
Let $U$ be a hyperbolic domain in $\C$ and let $F: U \to U$ be a holomorphic
map, such that for some $\zeta$ in the boundary of $U$ in $\clC$ we have
$F^n(z) \to \zeta$ as $n \to
\infty$ for $z \in U$. Then for every point $z \in U$ and every sequence of positive numbers $r_n$, $n \geq 0$ with $\lim_{n\to\infty} r_n = \infty$,
there exists a domain $W\subset U$, such that:
\begin{itemize}
\item[$(a)$]$W \subset \bigcup_{n=0}^\infty \DD_U(F^n(z), r_n)$,
\item[$(b)$] $\overline{W} \setminus \{\zeta\} \subset U$,
\item[$(c)$] $F^n(\overline{W}\setminus \{\zeta\}) =
\overline{F^n(W)}\setminus \{\zeta\} \subset F^{n-1}(W)$ for every $n\geq 1$,
\item[$(d)$] $\bigcap_{n=1}^\infty F^n(\overline{W}\setminus \{\zeta\})
= \emptyset$,
\item[$(e)$] $W$ is absorbing for $F$ in $U$.
\end{itemize}
Moreover, $F|_W$ lifts under a universal covering of $W$ to a univalent map and the induced endomorphism $(F|_W)^*$ of the fundamental group of $W$ is an isomorphism.
\end{thmA}

This theorem is an extension of the well-known Cowen's result \cite{cowen} on absorbing regions for holomorphic self-maps of simply connected domains. Recall that if $G$ is a holomorphic self-map of the right-half plane $\H$ without fixed points, then Denjoy--Wolff's theorem ensures that (after a possible change of variables) $G^n\to \infty$ uniformly on compact sets in $\H$. Cowen's result (see Theorem~\ref{thm:cowen}) implies the existence of a simply connected absorbing domain $V\subset\H$, such that $\overline{V} \subset  \H$, $G^n(\overline{V}) = \overline{G^n(V)} \subset G^{n-1}(V)$ for $n\geq 1$ and $\bigcap_{n\geq 0} G^n(\overline{V}) = \emptyset$. Moreover, there exists a univalent map $\varphi:V \to \C$ conjugating $G$ to a map $T$ of the form $T(\omega) = \omega +1$, $T(\omega) = a\omega$, $a > 1$ or $T(\omega) = \omega \pm i$ and $\varphi$ extends to a holomorphic map from $\H$ to $\C$ or $\H$, which semi-conjugates $G$ to $T$ (see Theorem~\ref{thm:cowen} for details). 
Using the Riemann Mapping Theorem, one can apply this result to a holomorphic self-map $F$ of any simply connected region $U$, without fixed points.

Applied to the case of Baker domains, Theorem~A has the following form.

\begin{cora}[\bf Existence of absorbing regions in Baker domains] 
Let $f: \C \to \clC$ be a meromorphic map and let $U$ be a periodic Baker domain of period $p$. Then there exists a domain $W \subset U$ with the
properties listed in Theorem~{\rm A} for the map $F = f^p$.
\end{cora}

Note that if $U$ is a simply connected Baker domain (which is always the case for entire maps), Cowen's Theorem immediately provides the existence of a suitable simply connected absorbing region in $U$. In the case of a multiply connected $p$-periodic Baker domain $U$ of a meromorphic map $f$, one can consider a universal covering map $\pi:\H \to U$ and lift $F = f^p$ by $\pi$ to a holomorphic map $G:\H \to \H$ without fixed points. K\"onig \cite{konig} showed that if $f$ has finitely many poles, then the absorbing region $V \subset\H$ projects under $\pi$ to a suitable simply connected absorbing region $W \subset U$ (see Theorem~\ref{thm:konig} for a precise statement).
However, \cite{konig} contains examples showing that there are Baker domains which do not admit simply connected absorbing regions.

Hence, Corollary~A' can be treated as a generalization of K\"onig's result, which weakens the assumptions on the map $f$ and provides some estimates on the size of the absorbing region, but does not ensure simple connectivity of $W$. 

Using Corollary~A', we are able to prove:

\begin{thmB}
\label{thm:baker}
Let $f$ be a transcendental meromorphic map with a multiply connected periodic   Baker domain. Then $f$ has at least one weakly repelling fixed point.
\end{thmB}

In particular, Theorem~B implies:

\begin{corb}
Periodic Baker domains of a Newton's method $N_g$ for an entire map $g$ are simply connected.
\end{corb}

This solves a well-known open question, raised e.g.~by Bergweiler, Buff, R\"uckert, Mayer and Schleicher \cite{bergweiler2,buffruc,mayer,johannes}. In particular, Corollary~B' implies that so-called virtual immediate basins for Newton maps (i.e.~invariant simply connected unbounded domains in $\C$, where the iterates of the map converge locally uniformly to $\infty$), defined by Mayer and Schleicher \cite{mayer}, are equal to the entire invariant Baker domains.

Apart from Corollary~A', the proof of Theorem~B uses several general results on the existence of weakly repelling fixed points of meromorphic maps on some domains in the complex plane, under certain combinatorial assumptions. These tools, which are developed in Section~\ref{section:configurations}, have some interest in themselves, since they generalize the results used by Shishikura, Bergweiler and Terglane \cite{berter,shishikura} and can be applied in a wider setup. In particular, we use them to prove the following result, which completes the proof of the Main Theorem.

\begin{thmC}
\label{thm:herman}
Let $f$ be a transcendental meromorphic map with a cycle of Herman rings. Then $f$ has at least one weakly repelling fixed point.
\end{thmC}
The proof of Theorem~C applies also to the rational setting and is an alternative to Shi\-shi\-kura's arguments for Herman rings of rational maps.

%%%%%%%%%%%%%%%%%%%%%%%%%%%%%%%%%%%%%%%%%%%%%%%%%%

The paper is organized as follows. In Section~\ref{sec:background} we state and reference some results we use in this paper. They include some estimates of the hyperbolic metric, the theorems of Cowen and K\"onig on the existence of absorbing domains and the results of Buff and Shishikura on the existence of weakly repelling fixed points for holomorphic maps. Section~\ref{sec:proof_thmA} contains the proof of Theorem~A. The proofs of Theorems~B and~C are contained, respectively, in Sections~\ref{sec:proof_thmB} and~\ref{sec:proof_thmC}, with an initial Section~\ref{section:configurations} which contain our preliminary results on the existence of weakly repelling fixed points in various configurations of domains. 

\subsection*{Acknowledgements} We wish to thank the Institut de Matem\`atiques de la Universitat de Barcelona for its hospitality.

%%%%%%%%%%%%%%%%%%%%%%%%%%%%%%%%%%%%%%%%%%%%%
%%%%%%%%%%%%%%%%%%%%%%%%%%%%%%%%%%%%%%%%%%%%%
%%%%%%%%%%%%%%%%%%%%%%%%%%%%%%%%
\section{Background and tools} \label{sec:background}

In this section we introduce notation and review the necessary background to
prove the main results of this paper. 

First, we present basic notation. 
The symbol $\dist(\cdot, \cdot)$ denotes the Euclidean distance on the complex plane $\C$. For a set $A
\subset\C$, the symbols $\overline{A}$, $\bd A$ denote, respectively, the closure 
and boundary in $\C$.
The Euclidean disc of radius $r$ centered at $z \in\C$  and the right half-plane are denoted, respectively, 
by $\D(z,r)$ and $\HH$. The unit disc $\D(0,1)$ is simply written as $\D$.

For clarity of exposition we divide this section into three parts. The first one contains standard estimates of hyperbolic metric. In the second and third one we present, respectively, some known results on the existence of absorbing domains and weakly repelling fixed points for holomorphic maps. 

\subsection{Hyperbolic metric and Schwarz--Pick's Lemma}

Let $U$ be a domain in the Riemann sphere $\clC$. We call $U$ hyperbolic, if its boundary in $\clC$ contain at
least three points. By the Uniformization Theorem, in this case there exists a
universal holomorphic covering $\pi$ from $\D$ (or $\HH$) onto $U$. Every holomorphic map $F: U \to U$ can be lifted by $\pi$ to a holomorphic map $G: \H \to \H$, such that the diagram
\[
\begin{CD}
\HH @>G>> \HH\\
@VV\pi V @VV\pi V\\
U @>F>> U
\end{CD}
\]
commutes. By
$\varrho_U(\cdot)$ and $\varrho_U(\cdot, \cdot)$ we denote, respectively, the
density of the hyperbolic metric and the hyperbolic distance in $U$. In particular, 
we will extensively use the hyperbolic metric in $\D$ and $\HH$ of density
\[
\varrho_\D(z) = \frac{2}{1 - |z|^2} \quad \text{and} \quad \varrho_\HH(z) = \frac{1}{\Re(z)},
\]
respectively. By $\DD_U(z, r)$ we denote the hyperbolic disc of radius $r$, centered at $z \in U$ (with
respect to the hyperbolic metric in $U$). The following lemma contains well-known inequalities related to the hyperbolic metric. 

\begin{lem}[\bf{Hyperbolic estimates I {\cite[Theorem 4.3]{carlesongamelin}}}] \label{lem:CG} Let $U \subset \C$ be a
hyperbolic domain. Then 
\[
\varrho_U(z) \leq \frac{2}{\dist(z, \bd U)} \qquad \text{for } z \in U
\]
and
\[
\varrho_U(z) \geq \frac{1 + o(1)}{\dist(z, \bd U) \log (1/\dist(z, \bd
U))}\qquad \text{as } z \to \bd U. 
\]
Moreover, if $U$ is simply connected, then
\[
\varrho_U(z) \geq \frac{1}{2\dist(z, \bd U)}\qquad \text{for } z \in U.
\]
\end{lem}

Every holomorphic map between hyperbolic domains does not increase the hyperbolic metric. This very useful result is known as the Schwarz--Pick Lemma.

\begin{lem}[\bf{Schwarz--Pick's Lemma {\cite[Theorem 4.1]{carlesongamelin}}}] \label{lemma:schwarz_pick}
Let $U, V \subset \C$ be hyperbolic domains and
let $f: U \to V$ be a holomorphic map. Then
\[
\varrho_V(f(z_1), f(z_2)) \leq \varrho_U(z_1, z_2)
\]
for every $z_1, z_2 \in U$. In particular, if $U \subset V$, then
\[
\varrho_V(z_1, z_2) \leq \varrho_U(z_1, z_2),
\]
with strict inequality unless $f$ lifts to a M\"obius transformation from
$\HH$ onto $\HH$.
\end{lem}

Using the two lemmas above we can easily deduce the following estimate, which will be useful in further parts of the paper. We include its proof for completeness.

\begin{lem}[{\bf Hyperbolic estimates II}]
\label{lem:dist} 
Let $U \subset \C$ be an unbounded
hyperbolic domain. Then there exists $c > 0$ such that 
\[
\varrho_U(z) > \frac{c}{|z| \log |z|} 
\]
if $z \in U$ and $|z|$ is sufficiently large.
\end{lem}
\begin{proof} Since $U$ is hyperbolic, there exist two distinct points $z_0,
z_1 \in \C \setminus U$, so $U$ is a subset of $U' = \C \setminus \{z_0, z_1\}$. By the
Schwarz--Pick's Lemma \ref{lemma:schwarz_pick}, we have $\varrho_U(z) \ge \varrho_{U'}(z)$ for $z \in U$.  
At the same time, $\varrho_{U'}(z) = c \varrho_{U''}(z)$ for $U''=\C\setminus \{0,1\}$ and a positive constant $c$, and the standard estimates of the hyperbolic metric in $U''$ (see e.g.~\cite{ahlfors,carlesongamelin}) give
\[
\varrho_{U''}(z) = \frac{\mathcal O(1)}{|z|\log(1/|z|)}
\]
as $|z|\to 0$. Transforming the metric under $1/z$, which leaves $U''$ invariant, we obtain
\[
\varrho_{U''}(z) =  \frac{\mathcal O(1)}{|z|\log|z|}
\]
as $|z| \to \infty$, from which the estimate follows.
\end{proof}

The following result follows easily from the algebraic properties of universal coverings (see e.g. \cite[Theorem~2]{pommar} or \cite[Lemma~4]{konig}). We include its 
proof for completeness.

\begin{lem}\label{lem:cover}
Let $U$ be a hyperbolic domain in $\C$ and let $F: U \to U$ be a holomorphic
map, such that for some $\zeta$ in the boundary of $U$ in $\clC$ we have
$F^n(z) \to \zeta$ as $n \to \infty$ for $z \in U$. Let $\pi: \HH \to U$ be 
a holomorphic universal covering and let $G: \HH \to \HH$ be a lift of $F$ by
$\pi$, i.e.~$F\circ\pi = \pi\circ G$. Suppose that $G$ is univalent. Then 
the induced endomorphism $F^*$ of the fundamental group of $U$ is an isomorphism. Moreover, if additionally, for every closed curve $\gamma \subset U$ there exists $n \ge 0$ such that $F^n(\gamma)$ is contractible in $U$, then $U$ is simply connected and $\pi$ is a Riemann map.
\end{lem}
\begin{proof}  
The domain $U$ is isomorphic (as a Riemann surface) to the quotient $\HH /
\Gamma$, where $\Gamma$ is the group of cover transformations acting on $\HH$.
The group $\Gamma$ is isomorphic to the fundamental group of $U$, denoted by
$\pi_1(U)$. For $n \ge 0$ let $\theta_n: \Gamma \to \Gamma$ be an endomorphism induced by $G^n$ (i.e.~$G^n\circ \gamma = \theta_n(\gamma) \circ G^n$ for $\gamma \in
\Gamma$). The endomorphism $\theta_n$ corresponds to an endomorphism $\tilde
\theta_n = (F^n)^*: \pi_1(U) \to \pi_1(U)$ induced by $F^n$ (see \cite{pommar}). Set $N
= \bigcup_{n=0}^\infty \ker \theta_n$, $\tilde N
= \bigcup_{n=0}^\infty \ker \tilde \theta_n$. Since $G$ is univalent, we have 
$N = \{\id\} = \tilde N$, so $(F^n)^*$ is an isomorphism. 

Suppose that for every closed curve $\gamma \subset U$ there exists $n \ge 0$ such that $F^n(\gamma)$ is contractible in $U$. Then $\pi_1(U) = \tilde N = \{\id\}$, so $U$ is simply connected and $\pi$ is a Riemann map.  
\end{proof}
%%%%%%%%%%%%%%%%%%%%%%%%%%

\subsection{Lifts of maps and absorbing domains}\label{subsec:absorbing}

Let $U$ be a hyperbolic domain in $\C$ and let $F: U \to U$ be a holomorphic map. Recall that an invariant domain $W \subset U$ is absorbing for $F$ in $U$, if for every compact set $K \subset U$ there exists $n >0$, such that $F^n(K) \subset W$. The main goal of this subsection is to present some results due to Cowen and K\"onig on the existence of absorbing domains.

Recall first the classical Denjoy--Wolff Theorem, which describes the dynamics of a holomorphic map $G$ in $\H$. 

\begin{thm}[\bf{Denjoy--Wolff's Theorem {\cite[Theorem 3.1]{carlesongamelin}}}] 
\label{thm:denjoy}
Let $G: \H \to \H$ be a non-constant holomorphic map, which is not an automorphism of $\H$. Then there exists a point $z_0 \in \clH \cup \{\infty\}$ $($called
the Denjoy--Wolff point of $G)$, such that $G^n$ tends to $z_0$  
uniformly on compact subsets of $\H$ as $n \to \infty$.
\end{thm}

The following result, due to Cowen, gives the main tool for constructing absorbing domains.

\begin{thm}[{\bf Cowen's Theorem {\cite[Theorem 3.2]{cowen}}}, see also {\cite[Lemma 1]{konig}}] \label{thm:cowen}
Let $G: \HH \to \HH$ be a holomorphic
map such that $G^n \to\infty$ as $n \to \infty$. Then there exists a simply
connected domain $V \subset \HH$, a domain $\Omega$ equal to $\HH$ or $\C$, a
holomorphic map $\varphi: \HH \to \Omega$, and a M\"obius transformation $T$
mapping $\Omega$ onto itself, such that:
\begin{itemize}
\item[$(a)$] $G(V) \subset V$,
\item[$(b)$] $V$ is absorbing in $\HH$ for $G$, 
\item[$(c)$] $\varphi(V)$ is absorbing in $\Omega$ for $T$,  
\item[$(d)$] $\varphi \circ G = T \circ \varphi$ on $\HH$,
\item[$(e)$] $\varphi$ is univalent on $V$.
\end{itemize}
Moreover, $\varphi$, $T$ depend only on $G$. In fact $($up to
a conjugation of $T$ by a M\"obius transformation preserving $\Omega)$, one of
the following cases holds:
\begin{itemize}
\item $\Omega = \C$, $T(\omega) = \omega + 1$,
\item $\Omega = \HH$, $T(\omega) = a\omega$ for some $a > 1$,
\item $\Omega = \HH$, $T(\omega) = \omega \pm i$.
\end{itemize}
\end{thm}

Using Cowen's result, K\"onig  proved the following theorem which provides the existence of simply connected absorbing domains for $F$ in $U$ under certain assumptions. In particular, these assumptions are trivially satisfied if $U$ is simply connected.

\begin{thm}[\bf{K\"onig's Theorem \cite{konig}}] 
\label{thm:konig}
Let $U$ be a hyperbolic domain in $\C$ and let $F: U \to U$ be a holomorphic
map, such that $F^n \to \infty$ as $n \to \infty$. Suppose that for every
closed curve $\gamma \subset U$ there exists $n > 0$ such that $F^n(\gamma)$ is
contractible in $U$. Then there exists a simply
connected domain $W \subset U$, a domain $\Omega$ and a transformation $T$ as
in Cowen's Theorem~{\rm \ref{thm:cowen}}, and a holomorphic map $\psi: U \to \Omega$, such that:
\begin{itemize}
\item[$(a)$] $F(W) \subset W$,
\item[$(b)$] $W$ is absorbing in $U$ for $F$, 
\item[$(c)$]  $\psi(W)$ is absorbing in $\Omega$ for $T$, 
\item[$(d)$] $\psi \circ F = T \circ \psi$ on $U$,
\item[$(e)$] $\psi$ is univalent on $W$.
\end{itemize}
In fact, if we take $V$ and $\varphi$ from Cowen's Theorem~{\rm  \ref{thm:cowen}} for $G$ being a
lift of $F$ by a universal covering $\pi : \HH \to U$, then $\pi$ is univalent
in $V$ and one can take $W = \pi(V)$ and $\psi = \varphi \circ \pi^{-1}$, which
is well defined in $U$.

Moreover, if $f$ is a meromorphic map with
finitely many poles, and $U$ is a periodic
Baker domain of period $p$, then the above assumptions are satisfied for $F = f^p$, and consequently, there exists $W\subset U$ with the properties $(a)$--$(e)$ for $F = f^p$. 
\end{thm}

\subsection{Existence of weakly repelling fixed points} \label{sec:prelimswrfp}

We shall use several tools to establish the existence of weakly repelling fixed
points in certain subsets of the plane. The results in this section will not be used until Section~\ref{sec:proof_thmB}.

The first classical result in this direction is due to Fatou.

\begin{thm}[\bf Fatou {\cite[Corollary 12.7]{milnor}}] \label{thm:fatouwrfp}
Every rational map $f: \clC \to\clC$ with $\deg(f)\geq 2$ has at least one weakly
repelling fixed point.
\end{thm}

In view of this, a map which locally behaves as a rational map should also have
points of the same character. This is formalized in the following two
propositions. By a proper map $f:D' \to D$ we mean a map from $D'$ onto $D$,
such that for every compact set $X \subset D$, the set $f^{-1}(X)$ is compact. Proper maps
always have a well defined finite degree.

\begin{thm}[\bf Polynomial-like maps \cite{doua-hubb}] \label{pol-like}
Let $D$ and $D'$ be simply connected domains in $\C$ such that $\overline{D'}
\subset D$ and let $f:D' \to D$ be a proper holomorphic map. Then $f$ has a
weakly repelling fixed point in $D'$.
\end{thm}

Indeed, if $\deg f|_{D'}=1$, then $f$ is invertible and, by Schwarz--Pick's
Lemma \ref{lemma:schwarz_pick} applied to $f^{-1}$, the map $f$ has a repelling fixed point. Otherwise,
$(f,D',D)$ form a polynomial-like map. By the Straightening Theorem (see
\cite{doua-hubb}), $f|_{D'}$ is conjugate to a polynomial and therefore has a weakly
repelling fixed point. 

\begin{thm}[\bf Rational-like maps \cite{buff}] \label{rat-like}
Let $D$ and $D'$ be domains in $\C$ with finite Euler
characteristic, such that $\overline{D'}
\subset D$ and let $f:D' \to D$ be a proper holomorphic map. Then $f$ has a
weakly repelling fixed point in $D'$.
\end{thm}

Maps with this property are called rational-like (see \cite{roesch}). The proof
of the result above is due to Buff and can be found in \cite{buff}, where he actually shows the
existence of {\em virtually repelling fixed points}, which is a stronger
statement. (Note that in \cite{buff} rational-like maps are assumed to have degree
larger than one. However, the proof  is valid also in
the case of degree one.) 

In the following result, also proved in \cite{buff}, the hypothesis of compact
containment is relaxed. In return, the image  is assumed to be a
disc.

\begin{thm}[\bf Rational-like maps with boundary contact \cite{buff}]
\label{boundarycontact}
Let $D$ be an open Euclidean disc in $\C$ and $D^{\prime}\subset D$ be
a domain with finite Euler characteristic. Let $f:D^{\prime}\to D$ be a
proper map of degree greater than one, such that $|f(z)-z|$ is bounded away
from zero as $z \to \bd D'$. Then $f$ has a weakly repelling fixed point in
$D^{\prime}$.
\end{thm}

By a meromorphic map on a domain $D \subset \clC$ we mean an analytic map from
$D$ to $\clC$. The result above implies the following corollary.

\begin{cor}[\bf Rational-like maps with boundary contact]\label{cor:boundarycontact}
Let $D$ be a simply connected domain in $\clC$ with locally connected
boundary and $D^{\prime}\subset D$ a domain in $\clC$ with finite Euler
characteristic. Let $f$ be a continuous map on the closure of $D'$ in $\clC$,
meromorphic in $D'$, such that $f:D^{\prime}\to D$ is proper. 
If $\deg f > 1$ and $f$ has no fixed points in $\partial D \cap \partial D'$, or $\deg f = 1$ and $D \ne D^{\prime}$,
then $f$ has a weakly repelling fixed point in
$D^{\prime}$. 
\end{cor}
\begin{proof} Suppose $\deg f > 1$. Changing the coordinates in $\clC$ by a
M\"obius transformation, we can assume $D \subset \C$. Let $\varphi$ be a
Riemann map from the unit disc $\D$ onto $D$. Since the boundary of $D$ is
locally connected, the map $\varphi$ extends continuously to $\overline{\D}$.
Let $g = \varphi^{-1}\circ f \circ \varphi$ on $\varphi^{-1}(D')$. Then
$g: \varphi^{-1}(D') \to \D$ satisfies the assumptions of
Theorem~\ref{boundarycontact}. Indeed, one should
only check that $|g(z)-z|$ is bounded away from zero as $z \to \bd
(\varphi^{-1}(D'))$. If it was not the case, then there would exists a sequence
$z_n \in \varphi^{-1}(D')$ with $z_n \to \bd(\varphi^{-1}(D'))$ and
$|z_n - g(z_n)| \to 0$. We can assume $z_n \to z \in \bd(\varphi^{-1}(D'))$.
Then $g(z_n) \to z$, $\varphi(z_n)
\to \varphi(z)$ and $\varphi(z)$ is in the boundary of $D'$, so
$f(\varphi(z_n))=
\varphi(g(z_n)) \to \varphi(z)$ and $f(\varphi(z)) = \varphi(z)$. Since
$f:D'\to D$ is proper, $\varphi(z)$ is in the boundary of $D$, so $\varphi(z)$
is a fixed point of $f$ in the intersection of
the boundaries of $D$ and $D'$, which contradicts the
assumptions of the corollary.

If $\deg f = 1$, then by the Riemann--Hurwitz Formula, $D^{\prime}$ is simply connected 
and $f$ is invertible, so the existence of a repelling fixed
point of $f$ follows from Schwarz--Pick's Lemma \ref{lemma:schwarz_pick}  applied to $f^{-1}$. 
\end{proof}

To apply this corollary we have to ensure the local connectedness of the boundary of the domain. To do so we shall often use the following result due to Torhorst.

\begin{thm}[\bf{Torhorst's Theorem {\cite[p.~106, Theorem 2.2]{whyburn}}}] \label{theorem:torhorst}
If $X$ is a locally connected continuum in $\clC$, then
the boundary of every component of $\clC \setminus X$ is a locally connected
continuum.
\end{thm}

We conclude this section stating  a surgery result due to Shishikura, which will be generalized in Section~\ref{section:configurations} (see Proposition~\ref{mapin}).

\begin{thm}[\bf Shishikura {\cite[Theorem 2.1]{shishikura}}] \label{shishikura}
Let $V_0, V_1$ be simply connected domains in $\clC$ with $V_0\neq \clC$ and
let $f$ be a meromorphic map in a neighbourhood $N$ of $\clC \setminus
V_0$, such that $f(\partial V_0)=\partial V_1$ and $f(V_0\cap N)\subset V_1$.
Suppose that for some $k\geq 1$, the map $f^k$ is defined on $V_1$, such that 
\[
f^j(V_1)\cap V_0  = \emptyset \text{ for } j = 0, \ldots, k - 1\quad
\text{and}\quad f^k(\overline{V_1})\subset V_0.
\]
Then $f$ has a weakly repelling fixed point in $\clC \setminus \overline{V_0}$.
\end{thm}
 
See Figure \ref{fig:shishi}.

\begin{center}
\begin{figure}[hbt!]
\def\svgwidth{0.4\textwidth}
%%%%
\begingroup%
  \makeatletter%
  \providecommand\color[2][]{%
    \errmessage{(Inkscape) Color is used for the text in Inkscape, but the package 'color.sty' is not loaded}%
    \renewcommand\color[2][]{}%
  }%
  \providecommand\transparent[1]{%
    \errmessage{(Inkscape) Transparency is used (non-zero) for the text in Inkscape, but the package 'transparent.sty' is not loaded}%
    \renewcommand\transparent[1]{}%
  }%
  \providecommand\rotatebox[2]{#2}%
  \ifx\svgwidth\undefined%
    \setlength{\unitlength}{493.76936035bp}%
    \ifx\svgscale\undefined%
      \relax%
    \else%
      \setlength{\unitlength}{\unitlength * \real{\svgscale}}%
    \fi%
  \else%
    \setlength{\unitlength}{\svgwidth}%
  \fi%
  \global\let\svgwidth\undefined%
  \global\let\svgscale\undefined%
  \makeatother%
  \begin{picture}(1,0.63116512)%
    \put(0,0){\includegraphics[width=\unitlength]{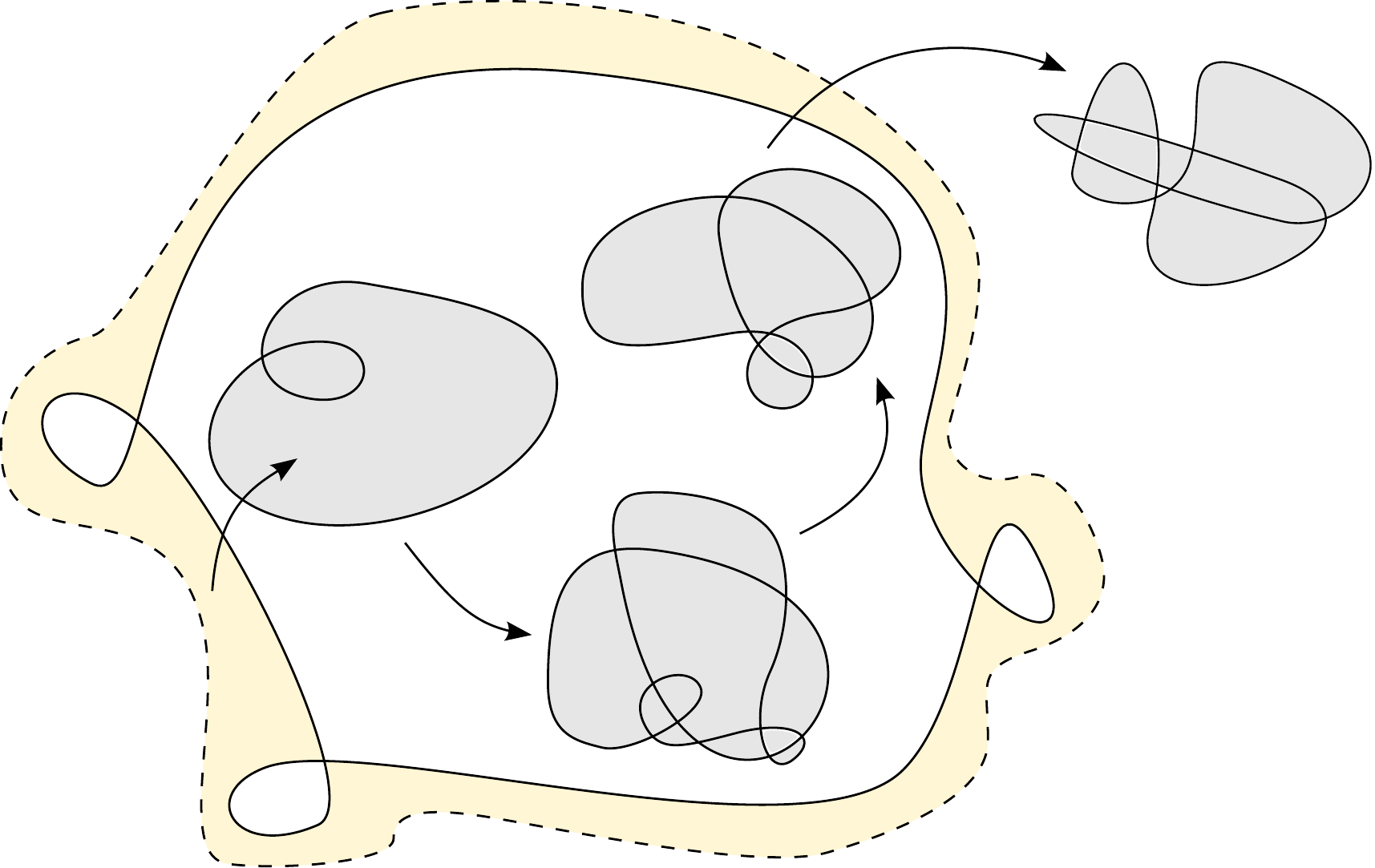}}%
    \put(0.0492187,0.49329894){\color[rgb]{0,0,0}\makebox(0,0)[lb]{\smash{$V_0$}}}%
    \put(0.28297516,0.32799975){\color[rgb]{0,0,0}\makebox(0,0)[lb]{\smash{$V_1$}}}%
    \put(0.82562393,0.59848934){\color[rgb]{0,0,0}\makebox(0,0)[lb]{\smash{$f^k³(V_1)$}}}%
    \put(0.26627826,0.58267111){\color[rgb]{0,0,0}\makebox(0,0)[lb]{\smash{\ss $N\cap V_0$}}}%
  \end{picture}%
\endgroup%
%%%%%%
\caption{ Setup of Theorem~\ref{shishikura}.} 
\label{fig:shishi}
\end{figure}
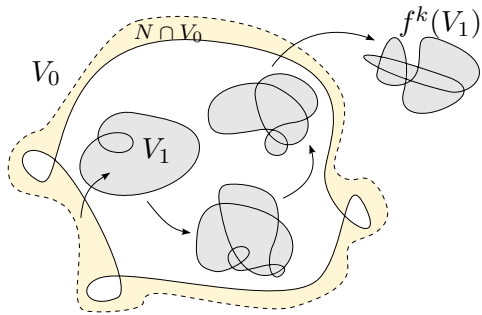
\end{center}

%%%%%%%%%%%%%%%%%%%%%%%%%%%%%%%%%%%%%%%%%%%%%
%%%%%%%%%%%%%%%%%%%%%%%%%%%%%%%%%%%%%%%%%%%%%
%%%%%%%%%%%%%%%%%%%%%%%%%%%%%%%%%%%%%%%%%%%%%

\section{Proof of Theorem A} \label{sec:proof_thmA}

The general setup for this section is the following. Let $U$ be a hyperbolic domain in $\C$. Then there exists a
holomorphic universal covering $\pi$ from $\HH$ onto $U$. Take a holomorphic
map $F: U \to U$ as in
Theorem~A. Then $F$ can be lifted to a holomorphic map $G: \HH \to \HH$, such
that
\[
F\circ\pi = \pi\circ G.
\]
Since $F$ has no fixed points, the map $G$ has no fixed points either, so by
the Denjoy--Wolff's Theorem \ref{thm:denjoy}, conjugating $G$ by a suitable M\"obius
transformation preserving $\HH$, we can assume that $G^n \to
\infty$ as $n \to \infty$.
Hence, by Cowen's Theorem \ref{thm:cowen}, $G$ is semi-conjugated to a M\"obius
transformation $T$ on $\Omega$ by a map $\varphi$, which is univalent on $V$.
In other words, we have the following commutative diagram.
\[
\begin{CD}
\varphi(V) @. \;\; \subset \;\;@. \Omega @>T>> \Omega\\
@VV{\varphi^{-1}}V @. @AA\varphi A @AA\varphi A\\
V @. \;\; \subset \;\; @. \HH @>G>> \HH\\
@. @. @VV\pi V @VV\pi V\\
@. @. U @>F>> U
\end{CD}
\]
We use the above notation throughout the proof.

Since the proof of Theorem~A  is rather technical, we first briefly discuss its geometric ideas.

We will define the absorbing set $W$ as the projection $W=\pi (\varphi^{-1} (A))$ of a suitable domain $A \subset \varphi(V)$, which is absorbing for $T$. Then one can easily show that $W$ is absorbing for $F$. However, we should be careful to define $A$ sufficiently ``thin'', so that $\overline{W} \subset U$ and $\bigcap_{n=1}^\infty F^n(\overline{W}) =\emptyset$ (a~priori, we could have e.g.~$W = U$). 

Notice that the map $T$ is an isometry in the hyperbolic metric in $\H$ (in the case $\Omega = \H$) or the Euclidean metric in $\C$ (in the case $\Omega = \C$). Hence, the idea is to define $A$ (in the case $\Omega = \H$) in the form
\[
A=\bigcup_{n\geq m} \DD_{\HH} (T^n(\omega),c_n)
\]
for a point $\omega\in\Omega$ and a suitable sequence $c_n$ which increases to $\infty$ sufficiently slowly (in the case $\Omega=\mathbb C$ we take Euclidean discs instead of hyperbolic ones). Then we show that $\overline{A} \subset \varphi^{-1} (V)$, $A$ is absorbing for $T$ and (by Schwarz--Pick's Lemma), $T(\overline{A})\subset A$. Moreover, taking a suitable sequence $c_n$, we can achieve
\[
\overline{A} \subset \bigcup_{n\geq m} \DD_{\varphi(V)} (T^n(\omega),b_n) 
\]
for any given sequence $b_n$ with $b_n \to \infty$. (Notice that since $V \subset \H$ is simply connected and $\varphi$ is univalent, the set $\varphi(V)$ is simply connected and $\varphi(V) \subsetneq \C$, so $\varphi(V)$ is hyperbolic.) The precise construction of the suitable domain $A$ will be done in Proposition~\ref{prop:star}.

Then, using Schwarz--Pick's Lemma, for any $z_0 \in U$ and any sequence $r_n$ with $r_n\to \infty$ we will choose $\omega$ and $b_n$ such that
\[
W = \pi (\varphi^{-1} (A))\subset \bigcup_{n\geq 0} \DD_U (F^n(z_0),r_n). 
\]
Taking $r_n$ converging to $\infty$ slowly enough, depending on the speed of escaping of $F^n(z_0)$ to $\infty$, we will show that $W$ is sufficiently ``thin'' to satisfy the assertions of Theorem~A. Notice that although we construct $A$ to be simply connected, the set $W$ will not be in general simply connected, unless $U$ is simply connected. 

The construction of the absorbing domain $A$ is done in the following proposition.

\begin{prop}[\bf{Absorbing domains in $\boldsymbol \Omega$}]\label{prop:star} Under the notation of Cowen's Theorem~{\rm \ref{thm:cowen}}, for every $\omega \in\Omega$ and every sequence of positive numbers $\{b_n\}_{n\geq 0}$ with $\lim_{n\to \infty} b_n = \infty$, there exist
$m\in\N$ and a simply connected domain $A \subset \Omega$ with the following properties: 
\begin{itemize}
\item[$(a)$] $\overline{A} \subset \bigcup_{n=m}^\infty \DD_{\varphi(V)}(T^n(\omega), b_n)\subset \varphi(V)$,
\item[$(b)$] $T(\overline{A}) \subset A$,
\item[$(c)$] $A$ is absorbing for $T$ in $\Omega$.
\end{itemize}
Moreover, if $\Omega = \C$, $T(\omega) = \omega +1$, then for every $\omega \in\Omega$ and $b > 0$ there exist a sequence $\{b_n\}_{n\geq 0}$ with $b_n < b$ and $\lim_{n\to \infty} b_n = 0$, a number $m\in\N$ and a simply connected domain $A \subset \Omega$, such that the conditions $(a)$--$(c)$ are satisfied.
\end{prop}
\begin{proof}
The proof splits in two cases, according to
$\Omega=\HH$ or $\Omega= \mathbb C$ in Cowen's Theorem~\ref{thm:cowen}.

\subsection*{Case 1. $\boldsymbol{\Omega= \HH}$} Then $T(\omega) = a \omega,\ a>1$ or $T(\omega) = \omega \pm i$. Notice that in this case $T$ is an isometry with respect to the hyperbolic metric in $\HH$. Take $\omega \in \HH$ and a sequence  $\{b_n\}_{n \geq 0}$ of positive numbers with $b_n \to \infty$ as $n\to \infty$. 

To define the domain $A$, first we claim that there is $m \in \N$ and a sequence of positive numbers $\{d_n\}_{n \geq 0}$ with $d_n \to \infty$ as $n\to \infty$, such that 
\begin{equation}\label{eq:omega}
\overline{\DD_\HH(T^n(\omega), d_n)} \subset \varphi(V) \qquad \text{for every }  n \ge m.
\end{equation}
To see the claim, suppose it is not true. Then there exists $d > 0$ such that $\overline{\DD_\HH(T^n(\omega), d)} \not\subset \varphi(V)$ for infinitely many $n$, 
which contradicts the assertion~(c) of Cowen's Theorem for the compact set $K = \overline{\DD_\HH(\omega, d)}$. Hence, we can take a sequence $\{d_n\}_{n \geq 0}$ satisfying \eqref{eq:omega}.

Now we define the absorbing set $A$ as
\[
A = \bigcup_{n=m}^\infty \DD_\HH(T^n(\omega), c_n),
\]
where
\[
c_n = \frac{1}{2} \min \left( \inf_{k\geq n}\ln \frac{1 + B_k D_k}{1 -
B_k D_k},\;  \varrho_\HH(T^n(\omega),\omega)\right)
\]
for
\[
B_n = \frac{e^{b_n} - 1}{e^{b_n} + 1}, \qquad D_n = \frac{e^{d_n} - 1}{e^{d_n} 
+ 1}.
\] 
Since, by definition, $b_n, d_n > 0$ and $b_n \to \infty$, $d_n \to \infty$ as $n\to \infty$, it follows that $0 < B_n < 1$, $0 < D_n < 1$ and $B_n \to 1$, $D_n \to 1$ as $n\to \infty$. Hence, the definition of $c_n$ implies (notice that $\varrho_\HH(T^n(\omega),\omega)\nearrow \infty$ as $n\to \infty$) that $\{c_n\}_{n \geq 0}$ is positive, increasing, tends to infinity and satisfies 
\[
c_n < \ln \frac{1 + B_n D_n}{1 - B_n D_n}.
\]
To ensure that $A$ is a domain we enlarge $m$ if necessary, so that 
$c_n > \varrho_\HH(\omega, T(\omega)) = \varrho_\HH (T^{n+1}(\omega), T^n(\omega))$ for all $n \geq m$. Hyperbolic discs in $\HH$ are Euclidean discs, so they are convex. Consequently, $A$ is simply connected, because it is a union of convex sets, all of them intersecting the straight line containing the trajectory of $T^n(\omega)$ under $T$. Notice also that defining
\[
C_n = \frac{e^{c_n} - 1}{e^{c_n} + 1},
\]
we have $C_n > 0$ and $c_n = \ln((1+C_n)/(1-C_n))$, so 
\begin{equation}\label{eq:C_n}
C_n < B_n D_n <D_n \quad \text{and} \quad c_n < d_n.
\end{equation}

The main ingredient to end the proof of the proposition is to show that 
the closure of $A$ equals the union of the closures of the respective discs, i.e.
\begin{equation}\label{eq:clA}
\overline{A} = \bigcup_{n=m}^\infty \overline{\DD_\HH(T^n(\omega), c_n)}.
\end{equation}

Before proving \eqref{eq:clA} we show how it implies the particular statements of the proposition. To prove the
statement~(b), it is enough to use \eqref{eq:clA} and notice that
\[
T(\overline{\DD_\HH(T^n(\omega), c_n)}) = \overline{\DD_\HH(T^{n+1}(\omega),
c_n)}
\subset \DD_\HH(T^{n+1}(\omega), c_{n+1}),
\]
because $c_{n+1} > c_n$. To show  the assertion~(c), take a compact set
$K \subset \HH$. Then $K \subset \DD_\HH(\omega, r)$ for some $r > 0$, so
\[
T^n(K) \subset T^n(\DD_\HH(\omega, r)) = \DD_\HH(T^n(\omega), r) \subset
\DD_\HH(T^n(\omega), c_n) \subset A 
\]
for sufficiently large $n$, because $c_n \to \infty$. 

Now we prove the statement~(a) of the proposition. By \eqref{eq:clA}, it suffices to show that
\begin{equation}\label{eq:d}
\overline{\DD_\HH(T^n(\omega), c_n)} \subset \DD_{\varphi(V)}(T^n(\omega),
b_n).
\end{equation}
Note that by \eqref{eq:omega} and Schwarz--Pick's Lemma~\ref{lemma:schwarz_pick} for the inclusion map, we have
\[
\overline{\DD_{\DD_\HH(T^n(\omega), d_n)}(T^n(\omega),
b_n)} \subset \DD_{\varphi(V)}(T^n(\omega), b_n), 
\]
and so, to show \eqref{eq:d} it is enough to prove
\begin{equation}\label{eq:d1}
\DD_\HH(T^n(\omega), c_n) \subset \DD_{\DD_\HH(T^n(\omega),
d_n)}(T^n(\omega), b_n).
\end{equation}
To show \eqref{eq:d1}, let $h_1$ be a M\"obius transformation of $\clC$ mapping $\HH$ onto $\D$ with $h_1(T^n(\omega)) = 0$. Then
\begin{align*}
& h_1(\DD_\HH(T^n(\omega), c_n)) = \DD_\D(0, c_n),\\
& h_1(\DD_{\DD_\HH(T^n(\omega), d_n)}(T^n(\omega), b_n)) =
\DD_{\DD_\D(0, d_n)}(0, b_n) = \DD_{\D(0, D_n)}(0, b_n),
\end{align*}
where the later equality follows from the definition of $D_n$ and the formula for the hyperbolic metric in $\D$. Hence, to prove
\eqref{eq:d1}, it suffices to check that
\begin{equation}\label{eq:d2}
\DD_\D(0, c_n) \subset \DD_{\D(0, D_n)}(0, b_n).
\end{equation}
Let $h_2(v) = v / D_n$ be the M\"{o}bius transformation which maps univalently 
$\D(0, D_n)$ onto $\D$. Similarly as before, we have
\begin{align*}
& h_2(\DD_\D(0, c_n)) = h_2(\D(0,C_n)) = \D\left(0,\frac{C_n}{D_n}\right),\\
&h_2(\DD_{\D(0, D_n)}(0, b_n)) = \DD_\D(0, b_n) = \D(0, B_n).
\end{align*}
Therefore, to prove \eqref{eq:d2} (and consequently \eqref{eq:d} and the statement~(a)), it is enough to show
\[
\D\left(0,\frac{C_n}{D_n}\right)\subset \D(0, B_n),
\]
which holds by \eqref{eq:C_n}. 

To end the proof of the proposition, it remains to prove \eqref{eq:clA}. Obviously, it suffices to show the inclusion $\overline{A} \subset \bigcup_{n=m}^\infty \overline{\DD_\HH(T^n(\omega), c_n)}$. Take $v \in \overline{A}$ and a sequence $v_k \in A$
such that $v_k \to v$ as $k \to \infty$. By the definition of $A$, there exists
a sequence $n_k \ge m$, such that
\[
v_k \in \DD_\HH(T^{n_k}(\omega), c_{n_k}).
\]
Since, by definition, $c_{n_k} \leq \varrho_\HH(T^{n_k}(\omega),\omega)/2$, we have
\[
\frac{\varrho_\HH(T^{n_k}(\omega),\omega)}{2} \geq c_{n_k} >
\varrho_\HH(T^{n_k}(\omega), v_k) \geq \varrho_\HH(T^{n_k}(\omega), \omega) - 
\varrho_\HH(v_k, \omega),
\]
so
\[
\varrho_\HH(v_k, \omega) > \frac{\varrho_\HH(T^{n_k}(\omega),\omega)}{2}.
\]
On the other hand, the sequence $\varrho_\HH(v_k, \omega)$ is bounded, because
$v_k \to v$.
Hence, the sequence $\varrho_\HH(T^{n_k}(\omega),\omega)$ must be bounded, so
$n_k$
is bounded. Therefore, taking a subsequence, we can assume that there exists $n
\geq m$ such that $n_k = n$ for every $k$, so
\[
v_k \in \DD_\HH(T^n(\omega), c_n).
\]
This implies 
\[
v \in \overline{\DD_\HH(T^n(\omega), c_n)},
\]
which finishes the proof of \eqref{eq:clA}. 

\subsection*{Case 2: $\boldsymbol{\Omega = \C}$}

In this case $T(\omega) = \omega +1$, so $T$ is an isometry with respect to the Euclidean metric in $\C$.
Since most of the arguments here are similar to the previous case (with the Euclidean metric instead of the hyperbolic one), we skip some details.

Similarly as before, we claim that the absorbing region $\varphi(V)$ must contain a union of appropriate discs of increasing radii. 
More precisely, for a given $\omega \in \C$ there exists $m \in\N$
and a sequence $\{d_n\}_{n \geq 0}$ of positive numbers with $d_n\to\infty$ as $n\to \infty$ such that 
\begin{equation}\label{eq:omega1}
\overline{\D(T^n(\omega), d_n)} \subset \varphi(V) \qquad \text{for every } n \ge
m.
\end{equation}
(If the claim was not true, then for the compact set $K =
\overline{\D(\omega, d)}$ we would have a contradiction with the assertion~(c) of Cowen's Theorem.) Hence, in what follows we will assume that the sequence $\{d_n\}_{n \geq 0}$ satisfies \eqref{eq:omega1}.

Take $b >0$ and let $b_n = 1 / \sqrt{d_n} \to 0$. Enlarging $m$ if necessary, we may assume $b_n < b$ for all $n\geq m$. We define the absorbing set $A$ as
\[
A = \bigcup_{n=m}^\infty \D(T^n(\omega), c_n)
\]
for
\[
c_n = \frac 1 2 \min \left(\inf_{k\geq n}\frac{e^{b_k} - 1}{e^{b_k} + 1}
d_k, n\right).
\]
Clearly, $\{c_n\}_{n \geq 0}$ is an increasing sequence of positive numbers. Moreover, we have 
\begin{equation}\label{eq:c_n}
c_n < \frac{e^{b_n} - 1}{e^{b_n} + 1} d_n < d_n \quad \text{and} \quad \frac{e^{b_n} - 1}{e^{b_n} + 1} d_n = \frac{e^{1/\sqrt{d_n}} -
1}{e^{1/\sqrt{d_n}} + 1} d_n \to \infty
\end{equation}
as $n\to\infty$. Hence, $c_n \to \infty$.
 
As in the previous case, enlarging $m$ if necessary, we can assume $A$ is a domain. Moreover, $A$ is simply connected, since it is a union of Euclidean discs intersecting the straight line
containing the $T$-trajectory of $\omega$. 

The main ingredient of the proof is to prove
\begin{equation}\label{eq:clA1}
\overline{A} = \bigcup_{n=m}^\infty \overline{\D(T^n(\omega), c_n)}.
\end{equation}
As in Case~1, first we show how \eqref{eq:clA1} implies the particular statements of the proposition. 
To show the statement~(b), we use \eqref{eq:clA1} and notice that
\[
T(\overline{\D(T^n(\omega), c_n)}) = \overline{\D(T^{n+1}(\omega), c_n)}
\subset \D(T^{n+1}(\omega), c_{n+1}),
\]
because $c_{n+1} > c_n$. To prove the assertion~(c), take a compact set
$K \subset \C$. Then $K \subset \D(\omega, r)$ for some $r > 0$, so
\[
T^n(K) \subset T^n(\D(\omega, r)) = \D(T^n(\omega), r) \subset
\D(T^n(\omega), c_n) \subset A 
\]
for sufficiently large $n$, because $c_n \to \infty$.

To prove the statement~(a), in view of \eqref{eq:clA1}, it suffices to show
\begin{equation}\label{eq:d'}
\overline{\D(T^n(\omega), c_n)} \subset \DD_{\varphi(V)}(T^n(\omega),
b_n).
\end{equation}
Note that by \eqref{eq:omega1} and Schwarz--Pick's Lemma \ref{lemma:schwarz_pick} we have
\[
\overline{\DD_{\D(T^n(\omega), d_n)}(T^n(\omega),
b_n)} \subset \DD_{\varphi(V)}(T^n(\omega), b_n),
\]
so, to show \eqref{eq:d'}, it is enough to prove
\begin{equation}\label{eq:d''}
\D(T^n(\omega), c_n) \subset \DD_{\D(T^n(\omega),
d_n)}(T^n(\omega), b_n).
\end{equation}
To see this is true we apply the univalent function  
$h(v) =\left (v- T^n(\omega)\right)/d_n$, which maps $\D(T^n(\omega), d_n)$ onto $\D$. We have
\begin{align*}
& h(\D(T^n(\omega), c_n)) = \D\left(0, \frac{c_n}{d_n}\right),\\ 
& h(\DD_{\D(T^n(\omega),
d_n)}(T^n(\omega), b_n)) = \DD_\D(0, b_n) = \D\left(0,
\frac{e^{b_n} - 1}{e^{b_n} + 1}\right).
\end{align*}
Therefore, to prove \eqref{eq:d''} (and consequently the statement~(a)), it is sufficient to check 
\[
\D\left(0, \frac{c_n}{d_n}\right) \subset \D\left(0,
\frac{e^{b_n} - 1}{e^{b_n} + 1}\right),
\]
which follows from \eqref{eq:c_n}. 

Finally, we prove \eqref{eq:clA1}. As in Case~1, it suffices to show $\overline{A} \subset \bigcup_{n=m}^\infty \overline{\D(T^n(\omega), c_n)}$. Take $v \in \overline{A}$ and a sequence $v_k \in A$ such that $v_k
\to v$ as $k
\to \infty$. Then there exists a sequence $n_k \ge m$, such that
\[
v_k \in \D(T^{n_k}(\omega), c_{n_k}).
\]
Since, by definition, $c_{n_k} \leq n_k/2$, we have
\[
\frac{n_k}{2} \geq c_{n_k} > |T^{n_k}(\omega) - v_k| = |n_k + \omega - v_k| \geq
n_k - |\omega| - |v_k|,
\]
so
\[
|v_k| > \frac{n_k}{2} - |\omega|.
\]
On the other hand, the sequence $v_k$ is bounded, because $v_k \to { v}$.
Hence, the sequence $n_k$ must be bounded, so taking a subsequence, we
can assume that $n_k = n$ for every $k$ and some $n \geq m$, so
\[
v_k \in \D(T^n(\omega), c_n) \text{ for every } k>0 \quad \text{and} \quad v \in \overline{\D(T^n(\omega), c_n)}.
\]
Hence, \eqref{eq:clA1} follows. 
\end{proof}

With Proposition \ref{prop:star} in hand, we are ready to prove Theorem A. 
We construct the absorbing region $W$ by projecting $A$ into the domain $U$.

\begin{proof}[Proof of Theorem~{\rm A}]
Changing coordinates in $\clC$ by a M\"obius transformation, we can assume
$\zeta = \infty$. Then the assertions~(b)--(d) of Theorem~A have the form
\begin{itemize}
\item[\textup{(b)}] $\overline{W} \subset U$,
\item[\textup{(c)}] $F^n(\overline{W}) =
\overline{F^n(W)} \subset F^{n-1}(W)$ for $n \ge 1$,
\item[\textup{(d)}] $\bigcap_{n=1}^\infty F^n(\overline{W})
= \emptyset$. 
\end{itemize}
Note that by Lemma~\ref{lem:dist}, there exist $c > 0$ and a large $r>0$ such that 
\begin{equation}\label{eq:delta}
\varrho_U(u) >\frac{c}{|u|\log|u|} \qquad \text{for }  u \in U, \; |u| \ge r.
\end{equation}
Fix some
$v_0 \in \varphi(V)$ and let 
$z_0 = \pi(\varphi^{-1}(v_0)).$
Since $F^n(z_0) \to \infty$, replacing $v_0$ by $T^j(v_0)$ for sufficiently
large $j$, we can assume
\begin{equation}\label{eq:c_1}
|F^n(z_0)| > r^{\log r} > r \quad \text{for every } n\geq 0.
\end{equation}

Take $z\in U$ and a sequence of positive numbers $\{r_n\}_{n\geq 0}$ with $r_n \to \infty$. Fix a number $n_0 \in \N$ such that 
\begin{equation}\label{eq:n_0}
r_n > 2\varrho_U(z,z_0) \quad \text{for every } n \geq n_0. 
\end{equation}
We define the sequence
\begin{equation}\label{eq:a_n}
a_n =  \frac{1}{2}\min\left(r_n, \frac{c}{2} \inf_{k\geq
n}\log\log|F^k(z_0)|\right).
\end{equation}
Clearly, $a_n\to \infty$ as $n\to \infty$. Let $A \subset \Omega$ be the domain from Proposition~\ref{prop:star} with $\omega = T^{n_0}(v_0)$ and $b_n = a_{n+n_0}$. Finally,  define 
\[
W = \pi(\varphi^{-1}(A)).
\]
By construction, we have the following commutative diagram.
\[
\begin{CD}
A @. \; \subset \;@. \varphi(V) @. \; \subset \;@. \Omega @>T>>
\Omega\\
@VV{\varphi^{-1}}V @. @VV{\varphi^{-1}}V @. @AA\varphi A @AA\varphi A\\
\varphi^{1}(A) @. \;\; \subset \;\;@. V @. \;\; \subset \;\; @. \HH @>G>> \HH\\
@VV\pi V @.@VV\pi V @. @VV\pi V @VV\pi V\\
W @. \;\; \subset \;\;@. \pi(V) @.\; \; \subset \;\; @. U @>F>> U
\end{CD}
\]
In the remaining part of the proof we show that $W$ satisfies the conditions listed in Theorem~A. 

First, we prove the statement~(a).
By Proposition~\ref{prop:star} we know that, for some $m\in\N$, 
\begin{equation}\label{eq:A}
A \subset \bigcup_{n = m}^\infty\DD_{\varphi(V)}(T^n(\omega), b_n) =
\bigcup_{n = m+n_0}^\infty\DD_{\varphi(V)}(T^n(v_0), a_n)\subset
\bigcup_{n =n_0}^\infty\DD_{\varphi(V)}(T^n(v_0), a_n).
\end{equation}
Hence, by Schwarz--Pick's Lemma~\ref{lemma:schwarz_pick} for $\varphi^{-1}$ and the inclusion map, we obtain
\[
\varphi^{-1}(A) \subset \bigcup_{n = n_0}^\infty \DD_V(\varphi^{-1}(T^n(v_0)),
a_n) = \bigcup_{n = n_0}^\infty \DD_V(G^n(\varphi^{-1}(v
_0)), a_n) \subset \bigcup_{n = n_0}^\infty \DD_\HH(G^n(\varphi^{-1}(v_0)), a_n)
\]
and
\[
W \subset \bigcup_{n =n_0}^\infty \DD_U(\pi(G^n(\varphi^{-1}(v_0))), a_n) = \bigcup_{n =
n_0}^\infty \DD_U(F^n(z_0), a_n).
\]
Using this together with \eqref{eq:n_0}, \eqref{eq:a_n} and Schwarz--Pick's
Lemma~\ref{lemma:schwarz_pick}, we get
\begin{multline*}
W \subset \bigcup_{n = n_0}^\infty \DD_U(F^n(z), a_n + \varrho_U(F^n(z),
F^n(z_0))) \\\subset \bigcup_{n = n_0}^\infty \DD_U(F^n(z), a_n +
\varrho_U(z, z_0)) \subset \bigcup_{n = n_0}^\infty \DD_U(F^n(z), r_n),
\end{multline*}
which ends the proof of the statement~(a). 

Now we prove the assertions (b)--(d).
Fix $j \geq 0$ and consider an arbitrary $u \in \overline{F^j(W)}$. Let $\{w_k\}_{k\ge 1}$ be a sequence of points in $W$, such that for $u_k = F^j(w_k)$ we have $u_k \to u$ as $k \to \infty$. Since $W = \pi(\varphi^{-1}(A))$, there exists a sequence of points $v_k \in A$ with $w_k = \pi(\varphi^{-1}(v_k))$. By \eqref{eq:A}, for every $k$ there exists $n_k \geq n_0$, such that
\begin{equation}\label{eq:v_k}
v_k \in \DD_{\varphi(V)}(T^{n_k}(v_0), a_{n_k}).
\end{equation}
Thus, by Schwarz--Pick's Lemma~\ref{lemma:schwarz_pick}, we have
\begin{equation}\label{eq:w_k,u_k}
w_k \in \DD_U(F^{n_k}(z_0), a_{n_k}), \quad 
u_k \in \DD_U(F^{n_k+j}(z_0), a_{n_k}).
\end{equation}

The key ingredient in the proof of the assertions (b)--(d) is to show
\begin{equation}\label{eq:u>}
|u_k| > e^{\sqrt{\log|F^{n_k + j}(z_0)|}}.
\end{equation}
To prove \eqref{eq:u>}, take $\gamma_k: [0,1] \to U$ to be a curve in $U$ such that $\gamma_k(0)=F^{n_k + j}(z_0)$, $\gamma_k(1)=u_k$,
\begin{equation}\label{eq:int}
\int_{\gamma_k} \varrho_U(\xi) |d\xi| < 2\varrho_U(F^{n_k + j}(z_0),u_k)
\end{equation}
and let
\[
t_k = \sup\{t \in [0,1]: |\gamma_k({ t' })| \geq r \text{ for all } 0< t' < t\}.
\]
By \eqref{eq:c_1}, $|\gamma_k(0)| > r$, so $t_k$ is well defined. Moreover, we have $|\gamma_k(t)| \ge r$ for $t \in [0, t_k]$ and $|\gamma_k(t_k)|\in \{r, |\gamma_k(1)|\}$. Notice that if $|\gamma_k(0)| < |\gamma_k(1)|$, then \eqref{eq:u>} follows from \eqref{eq:c_1}. Hence, we may assume $|\gamma_k(0)| \geq |\gamma_k(1)|$, which implies $|\gamma_k(0)| \geq |\gamma_k(t_k)|$. 
Using this together with
\eqref{eq:delta}, \eqref{eq:a_n}, \eqref{eq:w_k,u_k} and \eqref{eq:int}, we obtain
\begin{multline*}
\frac{c}{4} \log\log|F^{n_k + j}(z_0)| \geq a_{n_k} > \varrho_U(
F^{n_k + j}(z_0),u_k) \\>
\frac 1 2 \int_{\gamma_k} \varrho_U(\xi) |d\xi| \geq \frac 1  2
\int_{\gamma_k([0, t_k])} \varrho_U(\xi) |d\xi| \geq  \frac c 2 
\int_{\gamma_k([0,
t_k])} \frac{|d\xi|}{|\xi|\log|\xi|} {\geq} \frac c 2
\int_{|\gamma_k(t_k)|}^{|\gamma_k(0)|} \frac{ds}{s\log s}\\ = \frac c 2
(\log\log|F^{n_k
+ j}(z_0)| - \log\log|\gamma_k(t_k)|),
\end{multline*}
where the later inequality follows from the definition of the Riemann integral.
We conclude that $\log\log|\gamma_k(t_k)| > (\log\log|F^{n_k + j}(z_0)|)/2$, which means
\begin{equation}\label{eq:gamma_k>}
|\gamma_k(t_k)| > e^{\sqrt{\log|F^{n_k + j}(z_0)|}}.
\end{equation}
In particular, this implies that $|\gamma_k(t_k)| \neq r$, because otherwise we have a contradiction with \eqref{eq:c_1}. Hence, $|\gamma_k(t_k)| = |\gamma_k(1)|=|u_k|$, so \eqref{eq:gamma_k>} shows \eqref{eq:u>}.

Having \eqref{eq:u>}, we now prove the assertions (b)--(d) of Theorem~{\rm A}.
First, notice that since $u_k \to u$ as $k \to \infty$ and $F^n(z_0) \to \infty$ as $n\to\infty$, \eqref{eq:u>} implies that the sequence
$n_k$ is bounded. Hence, \eqref{eq:v_k} shows that the sequence $v_k$ is bounded, so taking a subsequence, we can assume that 
\[
v_k \to v \in\overline{A},
\]
and, by Proposition~\ref{prop:star}, $v \in \varphi(V)$. Therefore, by continuity, 
\begin{equation}\label{eq:w_k_to_w}
w_k \to w  = \pi(\varphi^{-1}(v)) \in \overline{W} \cap U \quad \text{and} \quad F^j(w) = u.
\end{equation}
Recall that $u$ was taken as an arbitrary point in $\overline{F^j(W)}$.
Hence, for $j=0$, \eqref{eq:w_k_to_w} implies $u = w \in U$, which proves the statement~(b) and shows that $F^j(\overline{W})$ is well defined for $j\ge 1$. To prove the assertion~(c), notice that \eqref{eq:w_k_to_w} gives $u = F^j(w) \in F^j(\overline{W})$, which shows $\overline{F^j(W)} \subset F^j(\overline{W})$. On the other hand, the inclusion $F^j(\overline{W}) \subset
\overline{F^j(W)}$ is obvious by the continuity of $F^j$, so $F^j(\overline{W}) = \overline{F^j(W)}$ for $j\ge 1$. To end the proof of the assertion~(c), it is sufficient to show $F^j(\overline{W}) \subset F^{j-1}(W)$ for $j \geq 1$. To do it, notice that Proposition~\ref{prop:star} implies $T(v) \in T(\overline{A}) \subset A$, so for $j=1$ \eqref{eq:w_k_to_w} gives  $u = F(w) = F(\pi(\varphi^{-1}(v))) =  \pi(\varphi^{-1}(T(v))) \in W$. Hence,
\[
F(\overline{W}) = \overline{F(W)}\subset W.
\]
This and induction on $j$ proves $F^j(\overline{W}) \subset F^{j-1}(W)$ for $j \geq 1$, which ends the proof of the assertion~(c). 

To show the statement~(d), notice that \eqref{eq:u>} implies $|u| \geq \inf_{n\ge
j+n_0}e^{\sqrt{\log|F^n(z_0)|}}$, so
\[
F^j(\overline{W}) =  \overline{F^j(W)}\subset \C \setminus \D\left(0, \inf_{n\ge
j+n_0}e^{\sqrt{\log|F^n(z_0)|}}\right).
\]
This proves (d), because $F^n(z_0) \to \infty$ as $n \to \infty$.

Now we show the statement~(e).
Take a compact set $K \subset U$ and a point $u \in K$. Let $w \in \HH$ be such
that $\pi(w) = u$ and take $N(w)$ to be an open neighbourhood of $w$, such that
$\overline{N(w)} \subset \HH$. Then $\pi (N(w))$ is an open neighbourhood of
$u$, so by the compactness of $K$, we can choose a finite number of points $u_1,
\ldots, u_k \in K$, such that $K \subset \bigcup_{j=1}^k \pi (N(w_j))$. Since
$L = \bigcup_{j=1}^k \varphi (\overline{N(w_j)})$ is a compact set in $\Omega$,
by Proposition~\ref{prop:star}, there exists $n$ such that $T^n(L) \subset A$.
This implies 
\[
\bigcup_{j=1}^k
G^n(N(w_j)) \subset 
\varphi^{-1}\left(\bigcup_{j=1}^k T^n (\varphi(N(w_j)))\right) =
\varphi^{-1}\left(\bigcup_{j=1}^k \varphi(G^n(N(w_j)))\right) \subset 
\varphi^{-1}(A),
\]
so 
\[
F^n(K) \subset \bigcup_{j=1}^k F^n(\pi(N(w_j))) = \bigcup_{j=1}^k
\pi(G^n(N(w_j))) \subset W,
\]
which ends the proof of the statement~(e).

To end the proof of Theorem~{\rm A}, notice that by Proposition~\ref{prop:star}, the domain $A$ is simply connected, $T|_A: A \to A$ is univalent and the map $\pi \circ\varphi^{-1}|_A$ is a holomorphic covering conjugating $T|_A$ to $F|_W: W \to W$. Hence, the lift of $F|_W$ under a universal covering is univalent. By Lemma~\ref{lem:cover}, the induced endomorphism $(F|_W)^*$ of the fundamental group of $W$ is an isomorphism. This ends the proof of Theorem~{\rm A}.
\end{proof}

%%%%%%%%%%%%%%%%%%%%%%%%%%%%%%%%%%%%%%%%%%%%%
%%%%%%%%%%%%%%%%%%%%%%%%%%%%%%%%%%%%%%%%%%%%%
%%%%%%%%%%%%%%%%%%%%%%%%%%%%%%%%%%%%%%%%%%%%%

\section{Configurations of domains and their images} \label{section:configurations}

In this section we present preliminary lemmas which we use repeatedly throughout the proofs of Theorems~B and~C. They provide the existence of weakly repelling fixed points for meromorphic maps in some domains under certain combinatorial conditions related to the configuration of the domain and its subsequent images.  These lemmas are formulated in a general setup and may have further applications apart from the ones used in this paper. 

The first lemma shows that a meromorphic map is proper on 
bounded components of the preimage of a domain with finite Euler
characteristic. 

\begin{lem}[\bf Proper restrictions of meromorphic maps] \label{lem:proper} Let $D\subset
\clC$ be a domain with finite Euler characteristic and 
let $f$ be a map, which is non-constant and meromorphic on a neighbourhood of $\overline{D^{\prime}}$,
where $D'$ is a bounded component of $f^{-1}(D)$. Then $D'$ has finite Euler
characteristic and the restriction $f:D' \to D$ is proper.
\end{lem}
\begin{proof} Clearly, we have $f(D') = D$. Since $D$ has finite Euler characteristic,
its boundary has a finite number of connected components, and each 
component of $\bd D^{\prime}$ is mapped by $f$ onto a component of $\bd D$.
Hence, the boundary of $D'$ has finitely many components, because otherwise
we could find $w_0$ in the boundary of $D$, such that $f $ takes the value $w_0$ on
a set with an accumulation point in $\overline{D'}$, so $f \equiv w_0$.
This implies that $D'$ has finite Euler characteristic and $f:D' \to
D$ is proper.
\end{proof}

\begin{defn*}[\bf Exterior of a compact set]
For a compact set $X \subset \C$ we denote by $\ext(X)$
the connected component of $\clC \setminus X$ containing infinity. 
We set $K(X) = \clC \setminus \ext(X)$. For a Jordan curve $\gamma \subset \C$
we denote by $\inter(\gamma)$ the bounded component of $\C \setminus \gamma$.
\end{defn*}

The following facts are immediate consequences of some standard topological
facts and the
maximum principle. We will use them repeatedly without explicit quotation. 

\begin{lem}[{\bf Properties of $\boldsymbol{K(X)}$ and $\boldsymbol{\ext(X)}$}]\label{lem:K} Let $X \subset \C$ 
be a compact set. Then:
\begin{itemize}
\item[$(a)$] if $X$ is connected, then $\ext(X)$ is a simply connected subset of $\clC$ and $K(X)$ is a connected subset of $\C$,
\item[$(b)$] if $X$ has a finite number of components, then $\ext(X)$ has finite Euler
characteristic,
\item[$(c)$]
$K(X)$ is a full compact set in $\C$,
\item[$(d)$]
if $Y \subset X$ is a compact set, then $\ext(Y) \supset
\ext(X)$ and $K(Y) \subset K(X)$,
\item[$(e)$]
if $f$ is meromorphic map in a neighbourhood of $K(X)$ and $K(X)$ does not
contain poles of $f$, then $f(K(X)) \subset K(f(X))$. 
\end{itemize}
\end{lem}

The next lemma shows that the multiple connectivity of a Fatou component $U$
implies the existence of a pole of $f$ in a bounded component of the complement of some image of $U$. This will be an important
property used in the proofs of the main theorems.

\begin{lem}[\bf Poles in loops] \label{poles-in-holes} 
Let $f:\C \to \clC$ be
a transcendental non-entire meromorphic map and let $\gamma \subset \C$ be a
closed curve in a Fatou component $U$ of $f$, such that $K(\gamma) \cap J(f) \ne
\emptyset$. Then there exists $n \geq 0$, such that $K(f^n(\gamma))$ contains
a pole of $f$. Consequently, if $U$ is multiply connected then there exists a bounded component of $\clC \setminus f^n(U)$, which contains a pole.
\end{lem}
\begin{proof} If $f$ has exactly one pole which is an omitted value, then $f$
is a self-map of a punctured plane and the claim follows easily from \cite[Theorem 1]{bakdom}. Hence,
we can assume that $f$ has at least two poles or exactly one pole, which is
not an omitted value. Then prepoles are dense in $J(f)$, so there is a prepole
in $K(\gamma)$. Suppose $K(f^n(\gamma))$ does not contain poles of $f$ for
every $n \ge 0$. Then $f^n$ is holomorphic
in a neighbourhood of $K(\gamma)$, so by Lemma~\ref{lem:K},
$f^n(K(\gamma)) \subset
K(f^n(\gamma))$ for every $n \geq 0$. Hence, $K(\gamma)$ cannot contain any
prepoles of $f$, which gives a contradiction. 
\end{proof}

The next lemma is a consequence of Buff's results on the existence of
weakly repelling fixed points for rational-like maps (Theorem~\ref{rat-like} and Corollary~\ref{cor:boundarycontact}).
\begin{lem}[\bf Boundary maps out] \label{mapout}
Let $\Omega \subset \C$ be a bounded domain with finite Euler
characteristic and let $f$ be a meromorphic map 
in a neighbourhood of $\overline{\Omega}$. Assume that there exists a component
$D$ of \ $\clC \setminus
f(\bd\Omega)$, such that:
\begin{itemize}
\item[$(a)$] $\overline{\Omega} \subset D$,
\item[$(b)$] there exists $z_0 \in \Omega$ such that $f(z_0) \in D$. 
\end{itemize}
Then $f$ has a weakly repelling fixed point in $\Omega$. Moreover, if
additionally $\Omega$ is simply connected with locally
connected boundary, then the assumption~$(a)$ can be replaced by: 
\begin{itemize}
\item[$(a')$] $\Omega \subsetneq D$ and $f$ has no fixed points in
$\bd \Omega \cap f(\bd\Omega)$.
\end{itemize}
\end{lem}
 
\begin{center}
\begin{figure}[hbt!]
\def\svgwidth{0.25\textwidth}
%%%
\begingroup%
  \makeatletter%
  \providecommand\color[2][]{%
    \errmessage{(Inkscape) Color is used for the text in Inkscape, but the package 'color.sty' is not loaded}%
    \renewcommand\color[2][]{}%
  }%
  \providecommand\transparent[1]{%
    \errmessage{(Inkscape) Transparency is used (non-zero) for the text in Inkscape, but the package 'transparent.sty' is not loaded}%
    \renewcommand\transparent[1]{}%
  }%
  \providecommand\rotatebox[2]{#2}%
  \ifx\svgwidth\undefined%
    \setlength{\unitlength}{189.20207715bp}%
    \ifx\svgscale\undefined%
      \relax%
    \else%
      \setlength{\unitlength}{\unitlength * \real{\svgscale}}%
    \fi%
  \else%
    \setlength{\unitlength}{\svgwidth}%
  \fi%
  \global\let\svgwidth\undefined%
  \global\let\svgscale\undefined%
  \makeatother%
  \begin{picture}(1,1.05598474)%
    \put(0,0){\includegraphics[width=\unitlength]{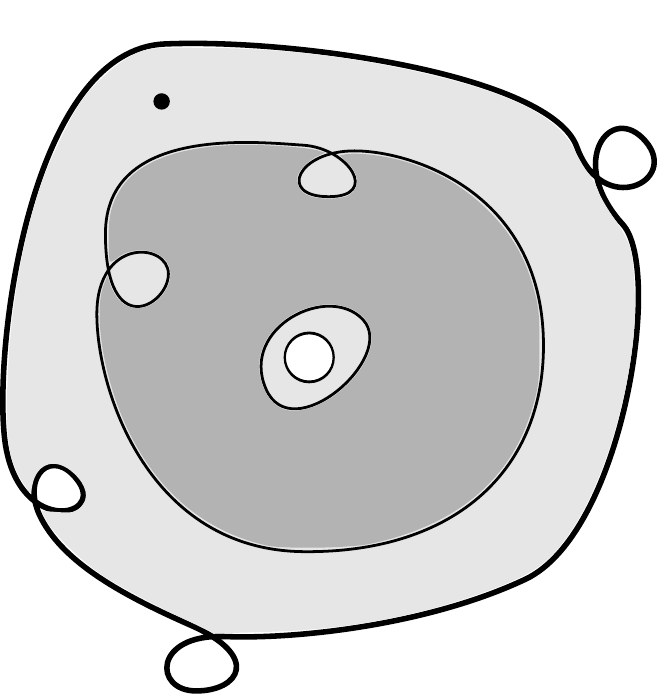}}%
    \put(0.71228117,0.76817204){\color[rgb]{0,0,0}\makebox(0,0)[lb]{\smash{$D$}}}%
    \put(0.27335038,0.88756295){\color[rgb]{0,0,0}\makebox(0,0)[lb]{\smash{\tiny $f(z_0)$}}}%
    \put(0.59349658,0.98489286){\color[rgb]{0,0,0}\makebox(0,0)[lb]{\smash{\small $f(\partial \Omega)$}}}%
    \put(0.3176375,0.66157669){\color[rgb]{0,0,0}\makebox(0,0)[lb]{\smash{\ss $z_0$}}}%
    \put(0.53499959,0.39370273){\color[rgb]{0,0,0}\makebox(0,0)[lb]{\smash{$\Omega$}}}%
  \end{picture}%
\endgroup%
%%%
\hfil
\def\svgwidth{0.25\textwidth}
%%%%%
\begingroup%
  \makeatletter%
  \providecommand\color[2][]{%
    \errmessage{(Inkscape) Color is used for the text in Inkscape, but the package 'color.sty' is not loaded}%
    \renewcommand\color[2][]{}%
  }%
  \providecommand\transparent[1]{%
    \errmessage{(Inkscape) Transparency is used (non-zero) for the text in Inkscape, but the package 'transparent.sty' is not loaded}%
    \renewcommand\transparent[1]{}%
  }%
  \providecommand\rotatebox[2]{#2}%
  \ifx\svgwidth\undefined%
    \setlength{\unitlength}{187.98725348bp}%
    \ifx\svgscale\undefined%
      \relax%
    \else%
      \setlength{\unitlength}{\unitlength * \real{\svgscale}}%
    \fi%
  \else%
    \setlength{\unitlength}{\svgwidth}%
  \fi%
  \global\let\svgwidth\undefined%
  \global\let\svgscale\undefined%
  \makeatother%
  \begin{picture}(1,0.98726729)%
    \put(0,0){\includegraphics[width=\unitlength]{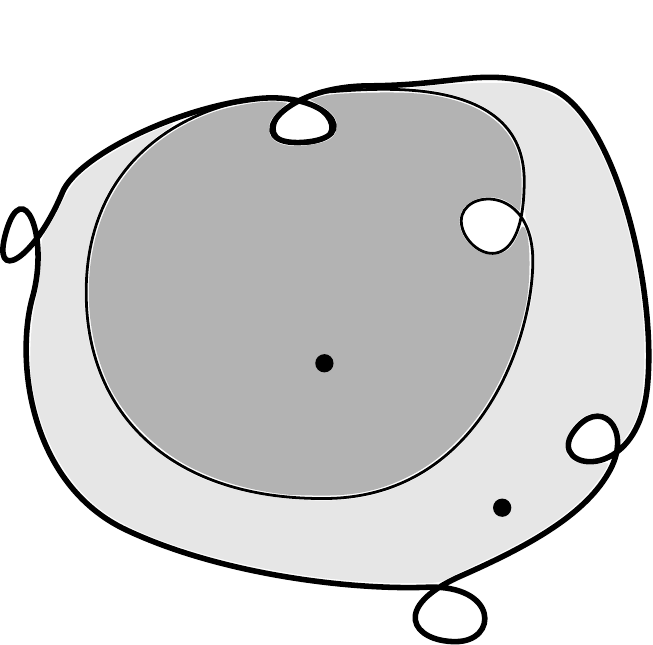}}%
    \put(0.66032452,0.915716){\color[rgb]{0,0,0}\makebox(0,0)[lb]{\smash{\small $f(\partial \Omega)$}}}%
    \put(0.71575316,0.24721122){\color[rgb]{0,0,0}\makebox(0,0)[lb]{\smash{\tiny $f(z_0)$}}}%
    \put(0.83322506,0.47768625){\color[rgb]{0,0,0}\makebox(0,0)[lb]{\smash{$D$}}}%
    \put(0.29778749,0.63744239){\color[rgb]{0,0,0}\makebox(0,0)[lb]{\smash{$\Omega$}}}%
    \put(0.43093459,0.47274){\color[rgb]{0,0,0}\makebox(0,0)[lb]{\smash{\ss $z_0$}}}%
  \end{picture}%
\endgroup%
%%%%%
\caption{Setup of Lemma~\ref{mapout} with the assumption~(a) (left) and~(a') (right).} 
\label{fig:mapout}
\end{figure}
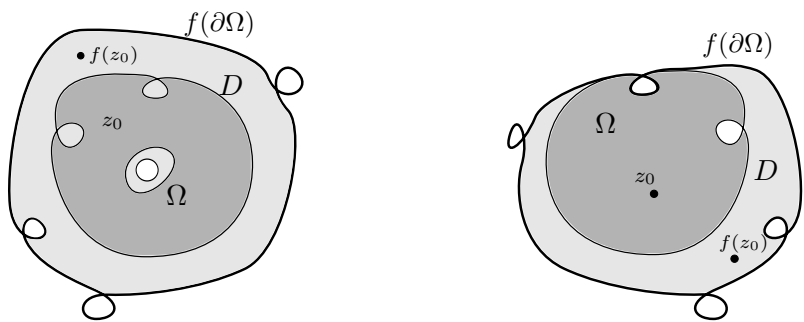
\end{center}

\begin{rem*}
Observe that if $\Omega$ is simply connected with locally connected boundary,
then  $f(\bd\Omega)$ is allowed to have common points with $\bd\Omega$ (see Figure~\ref{fig:mapout}). A version of this lemma requiring $f(\bd\Omega)$ to be
disjoint from $\bd\Omega$ and $f(z_0) = \infty$ appeared in
\cite[Lemma 1]{berter}. 
\end{rem*}

\begin{proof}[Proof of Lemma~\rm\ref{mapout}]
By the assumption (b), there exists a component $D^{\prime}$ of $f\inv (D)$
containing $z_0$. Observe that 
\[
D^{\prime} \subset \Omega.
\]
To see this, suppose that $D^{\prime}$ is not contained in $\Omega$. 
Then there exists $z\in D^{\prime}\cap \bd\Omega$. Consequently, $f(z)\in 
D \cap f(\bd\Omega)$. This is a contradiction since, by definition, $D \cap 
f(\bd\Omega) =\emptyset$.

As a consequence, $D^{\prime}$ is bounded. Moreover, since $\Omega$ has finite
Euler characteristic, $\bd\Omega$ (and hence $f(\bd\Omega)$ and $\bd D$) has a
finite number of components, so $D$ has finite Euler characteristic.
Therefore, by Lemma~\ref{lem:proper}, 
$D'$ has finite Euler characteristic and the restriction $f:D^{\prime} \to D$ is
proper. Moreover, the assumption~(a) implies $\overline{D'}
\subset D$. Hence (possibly after a change of coordinates in
$\clC$ by a M\"obius transformation), $f:D^{\prime} \to D$ is a 
rational-like map, so by Theorem~\ref{rat-like}, the map $f$ has a weakly
repelling fixed point in $D^{\prime} \subset \Omega$.

Finally, assume that $\Omega$ is simply connected with locally
connected boundary, and the assumption~(a) is replaced by~(a'). Then $\bd\Omega$ (and hence $f(\bd\Omega)$) is a 
locally connected continuum in $\clC$, so $D$ is
simply connected and, by the Torhorst Theorem~\ref{theorem:torhorst}, has locally connected boundary. Moreover, since $D'\subset
\Omega \subset D$ and the boundary of $D$ is contained in $f(\bd \Omega)$, the
intersection of the boundaries of $D$ and $D'$ is either empty or is contained
in $\bd \Omega \cap f(\bd \Omega)$. 
This together with the condition~(a') implies that the restriction $f:D^{\prime} \to D$
satisfies the assumptions of Corollary~\ref{cor:boundarycontact}, which ends
the proof.
\end{proof}

In particular, Lemma~\ref{mapout} implies the following two corollaries (see Figures~\ref{fig:boundaryoutpole}--\ref{fig:boundaryoutout}).

\begin{cor}[\bf Continuum surrounds a pole and maps out] \label{cor:mapout1} 
Let $X \subset \C$ 
be a continuum and let  $f$ be a meromorphic map in a neighbourhood of $K(X)$. Suppose that:
\begin{itemize}
\item[$(a)$] $f$ has no poles in $X$, 
\item[$(b)$] $K(X)$ contains a pole of $f$,
\item[$(c)$] $K(X) \subset \ext(f(X))$.
\end{itemize}
Then $f$ has a weakly repelling fixed point in the interior of $K(X)$.
\end{cor}

\begin{center}
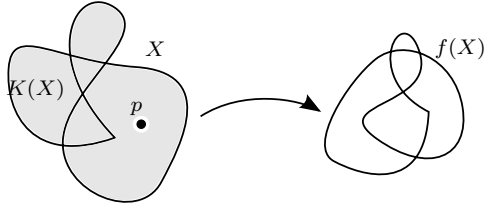
\begin{figure}[hbt!]
\def\svgwidth{0.45\textwidth}
%%%%
\begingroup%
  \makeatletter%
  \providecommand\color[2][]{%
    \errmessage{(Inkscape) Color is used for the text in Inkscape, but the package 'color.sty' is not loaded}%
    \renewcommand\color[2][]{}%
  }%
  \providecommand\transparent[1]{%
    \errmessage{(Inkscape) Transparency is used (non-zero) for the text in Inkscape, but the package 'transparent.sty' is not loaded}%
    \renewcommand\transparent[1]{}%
  }%
  \providecommand\rotatebox[2]{#2}%
  \ifx\svgwidth\undefined%
    \setlength{\unitlength}{237.08172759bp}%
    \ifx\svgscale\undefined%
      \relax%
    \else%
      \setlength{\unitlength}{\unitlength * \real{\svgscale}}%
    \fi%
  \else%
    \setlength{\unitlength}{\svgwidth}%
  \fi%
  \global\let\svgwidth\undefined%
  \global\let\svgscale\undefined%
  \makeatother%
  \begin{picture}(1,0.38161262)%
    \put(0,0){\includegraphics[width=\unitlength]{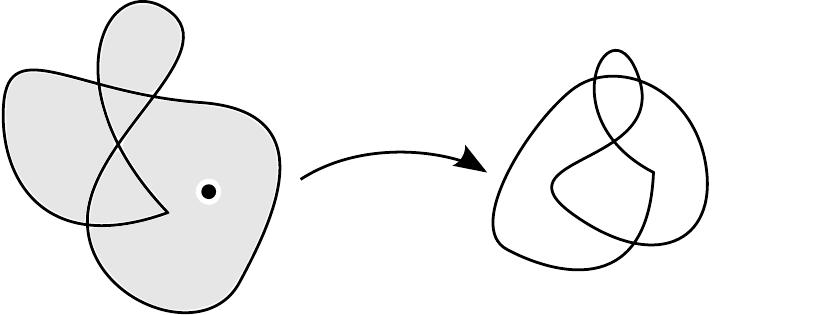}}%
    \put(0.80524992,0.28149383){\color[rgb]{0,0,0}\makebox(0,0)[lb]{\smash{\ss $f(X)$}}}%
    \put(0,0.20329486){\color[rgb]{0,0,0}\makebox(0,0)[lb]{\smash{\ss $K(X)$}}}%
    \put(0.26076438,0.27728506){\color[rgb]{0,0,0}\makebox(0,0)[lb]{\smash{\ss $X$}}}%
    \put(0.2343199,0.173851){\color[rgb]{0,0,0}\makebox(0,0)[lb]{\smash{\ss $p$}}}%
  \end{picture}%
\endgroup%
%%%%
\caption{Setup of Corollary~\ref{cor:mapout1}.} 
\label{fig:boundaryoutpole}
\end{figure}
\end{center}

\begin{proof} Let $p \in K(X)$ be a pole of $f$. Observe that by the assumption~(a), the set $f(X)$ (and hence $K(f(X))$) is a continuum in
$\C$. Moreover, (a) implies
\[
p \in \Omega \subset \overline\Omega \subset K(X)
\]
for a bounded simply connected component $\Omega$ of $\clC \setminus X$. 
We have $\bd\Omega \subset X$, which gives $f(\bd\Omega) \subset f(X)$, so by the assumption~(c),
\[
K(X) \subset \ext(f(\bd\Omega)),
\]
which implies $\overline\Omega \subset\ext(f(\bd\Omega))$.

Let $D = \ext(f(\bd\Omega))$. We have $\overline\Omega \subset D$, 
$p \in \Omega$ and $f(p) = \infty \in D$. Hence, the assumptions of
Lemma~\ref{mapout} are satisfied for
$\Omega, D, p$, so $f$ has a weakly
repelling fixed point in $\Omega$, which is a subset of the interior of
$K(X)$.
\end{proof}

\begin{cor}[\bf Continuum maps out twice] \label{cor:mapout2} 
Let $X \subset \C$ be a continuum and let
$f$ be a meromorphic map in a neighbourhood of $X \cup K(f(X))$.
Suppose that:
\begin{itemize}
\item[$(a)$] $f$ has no poles in $X$, 
\item[$(b)$] $X \subset K(f(X))$,
\item[$(c)$] $f^2(X) \subset \ext(f(X))$.
\end{itemize}
Then $f$ has a weakly repelling fixed point in the interior of $K(f(X))$.
\end{cor}

\begin{center}
\begin{figure}[hbt!]
\def\svgwidth{0.65\textwidth}
%%%%%
\begingroup%
  \makeatletter%
  \providecommand\color[2][]{%
    \errmessage{(Inkscape) Color is used for the text in Inkscape, but the package 'color.sty' is not loaded}%
    \renewcommand\color[2][]{}%
  }%
  \providecommand\transparent[1]{%
    \errmessage{(Inkscape) Transparency is used (non-zero) for the text in Inkscape, but the package 'transparent.sty' is not loaded}%
    \renewcommand\transparent[1]{}%
  }%
  \providecommand\rotatebox[2]{#2}%
  \ifx\svgwidth\undefined%
    \setlength{\unitlength}{360.73447633bp}%
    \ifx\svgscale\undefined%
      \relax%
    \else%
      \setlength{\unitlength}{\unitlength * \real{\svgscale}}%
    \fi%
  \else%
    \setlength{\unitlength}{\svgwidth}%
  \fi%
  \global\let\svgwidth\undefined%
  \global\let\svgscale\undefined%
  \makeatother%
  \begin{picture}(1,0.35797389)%
    \put(0,0){\includegraphics[width=\unitlength]{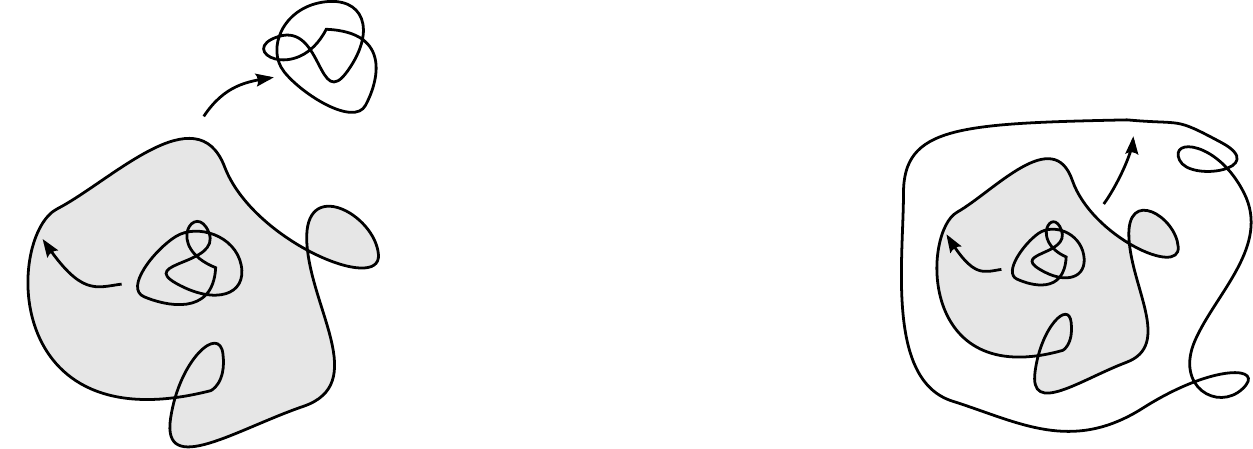}}%
    \put(-0,0.21468287){\color[rgb]{0,0,0}\makebox(0,0)[lb]{\smash{\ss $f(X)$}}}%
    \put(0.11104908,0.08932744){\color[rgb]{0,0,0}\makebox(0,0)[lb]{\smash{\ss $X$}}}%
    \put(0.12692485,0.31371206){\color[rgb]{0,0,0}\makebox(0,0)[lb]{\smash{\ss$f^2(X)$}}}%
    \put(0.84964011,0.1118699){\color[rgb]{0,0,0}\makebox(0,0)[lb]{\smash{\ss $X$}}}%
    \put(0.7246421,0.20212971){\color[rgb]{0,0,0}\makebox(0,0)[lb]{\smash{\ss$f(X)$}}}%
    \put(0.82785575,0.27326345){\color[rgb]{0,0,0}\makebox(0,0)[lb]{\smash{\ss $f^2(X)$}}}%
  \end{picture}%
\endgroup%
%%%%%
\caption{Two possible setups of Corollary~\ref{cor:mapout2}.} 
\label{fig:boundaryoutout}
\end{figure}
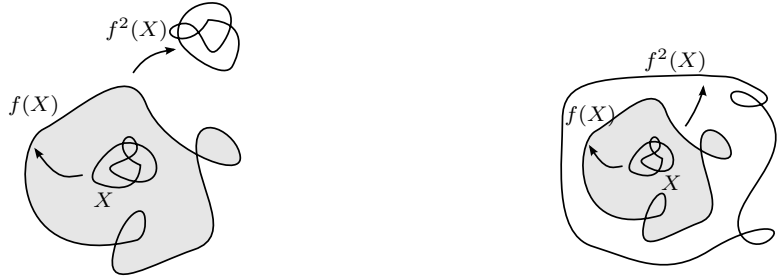
\end{center}

\begin{proof} By the assumption (a), the set $f(X)$ (and hence $K(f(X))$) is a continuum in
$\C$ and $f^2(X)$ is a continuum in $\clC$. Moreover, $X \cap f(X) = \emptyset$
(otherwise $f(X) \cap f^2(X) \ne \emptyset$, which contradicts the assumption~(c)). Hence, by (b),
\[
X \subset \Omega \subset \overline\Omega \subset K(f(X))
\]
for some bounded simply connected component $\Omega$ of $\clC \setminus f(X)$. 
We have $\bd\Omega \subset f(X)$, so $f(\bd\Omega) \subset f^2(X)$ and by the assumption~(c),
\[
K(f(X)) \subset \clC \setminus f^2(X) \subset \clC \setminus f(\bd\Omega),
\]
which gives $K(f(X)) \subset D$ for some component $D$ of $\clC
\setminus f(\bd\Omega)$. Consequently, $\overline\Omega \subset K(f(X))\subset D$. Moreover, for any $z_0 \in X$ we have $z_0 \in \Omega$ and $f(z_0) \in f(X)
\subset D$. Hence, the assumptions of Lemma~\ref{mapout} are satisfied for
$\Omega, D, z_0$, so $f$ has a weakly
repelling fixed point in $\Omega$, which is contained in the interior of
$K(f(X))$.
\end{proof}

The previous results gave some conditions for the existence of a weakly repelling fixed point in the case when
a closed curve is mapped by $f$ into its exterior. The following proposition, which is a considerable generalization of Shishikura's Theorem~\ref{shishikura}, gives conditions for the existence of a weakly repelling fixed point in the case when a closed curve before mapping out is mapped by $f$ several times into its interior (see Figure \ref{fig:mapin}).

\begin{prop}[\bf Boundary maps in] \label{mapin} Let $\Omega \subset \C$ be a
bounded simply connected domain and let $f$ be a meromorphic map in a
neighbourhood of $\overline{\Omega}$. Suppose that:
\begin{itemize}
\item[$(a)$] there exists $m \ge 2$, such that $f^m$ is defined on
$\bd \Omega$, 
\item[$(b)$] $f^j(\bd \Omega) \subset \overline{\Omega}$ for $j = 1,
\ldots, m - 1$,
\item[$(c)$] $f^m(\bd \Omega) \cap \overline{\Omega} = \emptyset$.
\end{itemize}
Then $f$ has a weakly repelling fixed point in $\Omega$.
\end{prop}

\begin{figure}[hbt!]
\def\svgwidth{0.4\textwidth}
%%%%
\begingroup%
  \makeatletter%
  \providecommand\color[2][]{%
    \errmessage{(Inkscape) Color is used for the text in Inkscape, but the package 'color.sty' is not loaded}%
    \renewcommand\color[2][]{}%
  }%
  \providecommand\transparent[1]{%
    \errmessage{(Inkscape) Transparency is used (non-zero) for the text in Inkscape, but the package 'transparent.sty' is not loaded}%
    \renewcommand\transparent[1]{}%
  }%
  \providecommand\rotatebox[2]{#2}%
  \ifx\svgwidth\undefined%
    \setlength{\unitlength}{486.00760295bp}%
    \ifx\svgscale\undefined%
      \relax%
    \else%
      \setlength{\unitlength}{\unitlength * \real{\svgscale}}%
    \fi%
  \else%
    \setlength{\unitlength}{\svgwidth}%
  \fi%
  \global\let\svgwidth\undefined%
  \global\let\svgscale\undefined%
  \makeatother%
  \begin{picture}(1,0.57315371)%
    \put(0,0){\includegraphics[width=\unitlength]{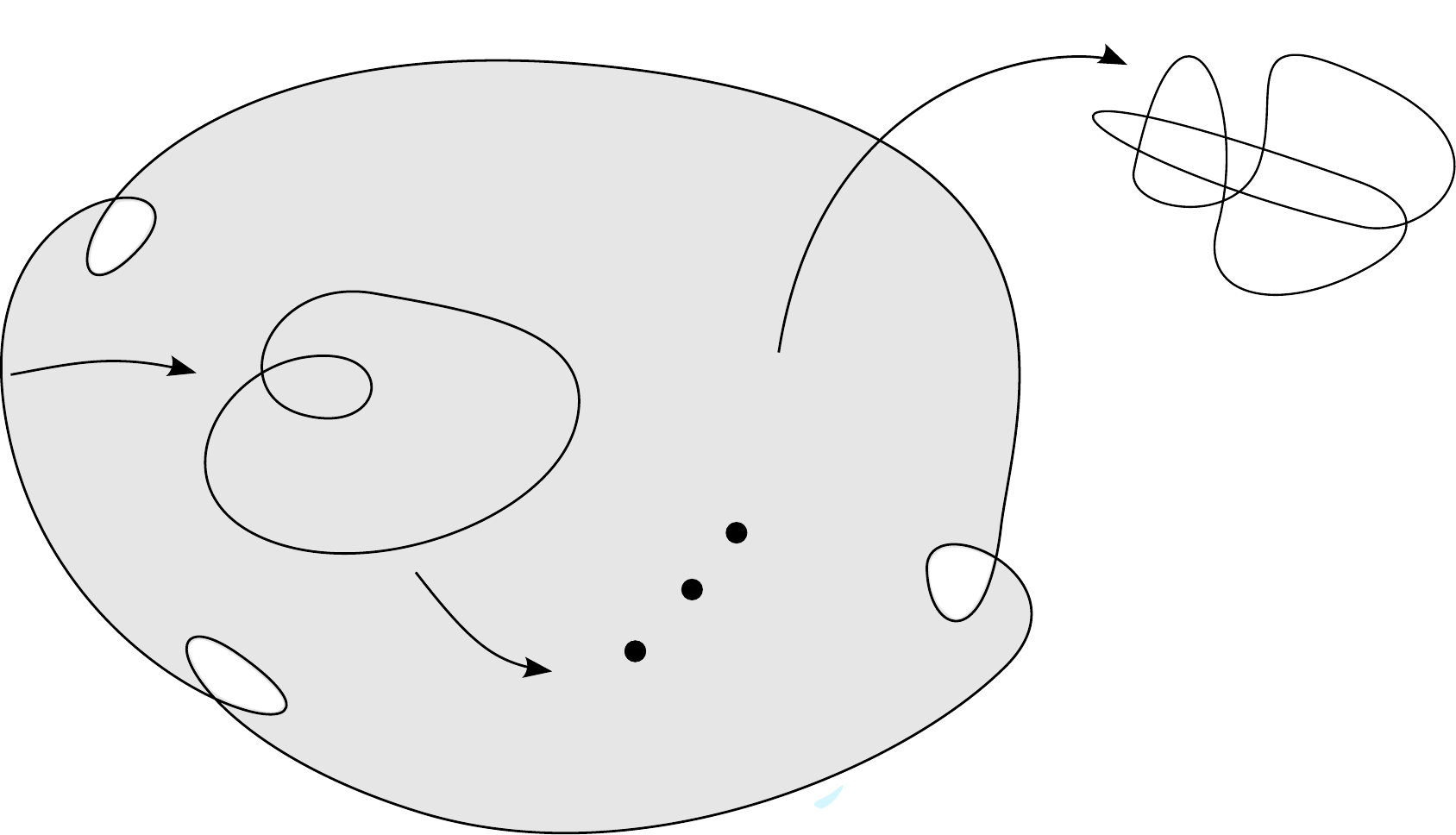}}%
    \put(0.41961323,0.44296727){\color[rgb]{0,0,0}\makebox(0,0)[lb]{\smash{$\Omega$}}}%
    \put(0.82504193,0.54947299){\color[rgb]{0,0,0}\makebox(0,0)[lb]{\smash{\small $f^m(\partial \Omega)$}}}%
    \put(0.1498929,0.38529126){\color[rgb]{0,0,0}\makebox(0,0)[lb]{\smash{\ss $f(\partial \Omega)$}}}%
  \end{picture}%
\endgroup%
%%%%
\caption{Possible setup for Proposition~\ref{mapin}.}
\label{fig:mapin}
\end{figure}
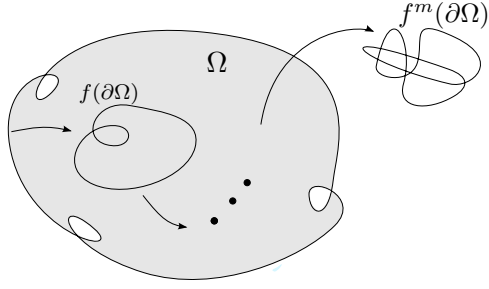

\begin{proof} We proceed by contradiction, i.e.~we assume that $f$ has no weakly repelling fixed points in $\Omega$. The proof is split into a number of steps.

\subsection*{Step 1.} First, note that the simple connectedness of $\Omega$ implies that 
$\bd\Omega$ (and hence $f^j(\bd \Omega)$ for $j = 1, \ldots, m$) is connected. Moreover,
the following conditions are satisfied:
\begin{align}
\label{eq:disjoint1} &\bd \Omega, f(\bd \Omega),\ldots, f^m(\bd \Omega) \text{ are pairwise
disjoint}, \\
& K(f^j(\bd \Omega)) \subset \Omega \quad \text{for } j = 1, \ldots, m - 1. \label{eq:inOmega}
\end{align}
To see \eqref{eq:disjoint1}, notice that if $z \in f^{j'}(\bd \Omega) \cap f^{j''}(\bd \Omega)$ for some $0 \leq j' < j'' \leq m$, then $f^{m - j''}(z) \in f^{m + j' - j''}(\bd \Omega) \cap
f^m(\bd \Omega)$ and $0 \leq m + j' - j'' < m$, which contradicts the assumptions~(b)--(c). Hence, \eqref{eq:disjoint1} follows. Now \eqref{eq:disjoint1} together with (b) implies \eqref{eq:inOmega}.

\subsection*{Step 2.}  We show that we can reduce the proof to the case
\begin{equation}\label{eq:not_in}
f^{j+1}(\bd \Omega) \subset \ext(f^j(\bd \Omega))
\quad \text{for } j = 1, \ldots, m - 1.
\end{equation}
To see this, suppose that there exists
$j_0 \in \{ 1, \ldots, m - 1\}$ such that $f^{j_0+1}(\bd \Omega) \subset K(
f^{j_0}(\bd \Omega))$ and take the maximal number $j_0$ with this
property. Then by the assumption~(c) and \eqref{eq:inOmega}, $j_0 \neq m-1$ and
$f^{j_0+1}(\bd \Omega) \subset \Omega_0 \subset \Omega$ for some
bounded simply connected component $\Omega_0$ of $\clC \setminus f^{j_0}(\bd
\Omega)$. We have
$f^k(\bd\Omega_0) \subset f^{k+j_0}(\bd\Omega)$ for $k\ge 0$. Hence, it follows from
\eqref{eq:inOmega} and (c) that there exists $m_0 \ge 2$ such that $f^j(\bd
\Omega_0) \subset
\Omega_0$ for $j = 1, \ldots, m_0 - 1$ and $f^{m_0}(\bd \Omega_0) \cap
\overline{\Omega_0} = \emptyset$. Thus, the assumptions~(a)--(c)
are satisfied for $\Omega_0$, $m_0$. Since $j_0$ was maximal, this implies that replacing, respectively, 
$\Omega$  and $m$ by $\Omega_0$ and $m_0$, we can assume 
$f^{j+1}(\bd \Omega) \not\subset K(f^j(\bd \Omega))$ for $j = 1, \ldots, m - 1$.
Since by \eqref{eq:disjoint1}, there is no intersection between the images of $\bd\Omega$, we have proven that we can reduce the proof to the case \eqref{eq:not_in}.

\subsection*{Step 3.} We claim that there exists a Jordan curve $\sigma_1 \subset \C$ close to $f(\bd \Omega)$ such that:
\begin{align}
& \inter(\sigma_1)\supset K(f(\bd \Omega)),\label{eq:Kinsigma} \\
&\text{$\sigma_1$ contains no values of critical points of $f$ in $\overline{\Omega}$}\notag,\\
&\sigma_1, f(\sigma_1), \ldots, f^{m-2}(\sigma_1) \text{ are pairwise disjoint subsets of $\Omega$} \quad
\text{and} \quad f^{m-1}(\sigma_1) \cap \overline{\Omega} = \emptyset, \label{eq:disjoint'} \\
&f^{j+1}(\sigma_1) \subset \ext(f^j(\sigma_1)) \quad \text{for } \quad j = 0, \ldots, m - 2. \label{eq:not_in'}
\end{align}
The existence of a curve satisfying these four conditions follows easily from \eqref{eq:disjoint1}, \eqref{eq:inOmega}, \eqref{eq:not_in}, the assumption~(c) and the fact that the set of critical points in $\overline{\Omega}$ is finite.

We then consider the set 
\[
D=\ext(\sigma_1).
\]
By the assumption~(c) and \eqref{eq:disjoint'}, we have $f^m(\bd \Omega) \subset D$. 
Hence, there exists a component $D^{\prime}$ of $f\inv(D)$ containing $f^{m-1}(\bd \Omega)$. By definition, $D^{\prime}$ intersects $\Omega$ and contains a pole of
$f$. Consequently,
\[
D' \subset \Omega,
\]
because otherwise $D' \cap \bd \Omega\neq \emptyset$, so $D \cap f(\bd
\Omega) \neq \emptyset$, which is impossible by \eqref{eq:Kinsigma}. 
Therefore, $D'$ is bounded and by Lemma~\ref{lem:proper}, it has finite Euler
characteristic and the restriction $f:D^{\prime} \to D$ is proper. In fact,
since $\bd D$ contains no values of critical points of $f$ in $\bd D'$,
the boundary of $D'$ consists of finitely many disjoint
Jordan curves, $f$ is a finite degree
covering in a neighbourhood of every component of $\bd D'$ and maps this
component onto $\sigma_1$. 

We now define $\sigma_0$ to be the Jordan curve, which is the boundary of the unbounded
component of $\clC\setminus D'$. Notice that $D' \subset \inter(\sigma_0) \subset \Omega$, moreover $\inter(\sigma_0)$ 
contains a pole of $f$ and $f(\sigma_0) = \sigma_1$. We will use the notation
\[
\sigma_j = f^j(\sigma_0).
\]
By \eqref{eq:disjoint'}, we have $\sigma_0 \cap
\sigma_j = \emptyset$ for $j = 1, \ldots, m$, which means
that
\begin{equation}\label{eq:disjoint''}
K(\sigma_j) \subset \inter(\sigma_0) \text{ or } \sigma_j
\subset \ext(\sigma_0).
\end{equation}
(see Figure \ref{fig:proof_Lemma_map_in_1}). Finally, we note that $\sigma_0$ and $\sigma_1$ are, by construction, Jordan curves, while $\sigma_j$ for $j=2,\ldots m$ 
are closed curves, which are not necessarily Jordan.

\begin{center}
\begin{figure}[hbt!]
\def\svgwidth{0.9\textwidth}
%%%
\begingroup%
  \makeatletter%
  \providecommand\color[2][]{%
    \errmessage{(Inkscape) Color is used for the text in Inkscape, but the package 'color.sty' is not loaded}%
    \renewcommand\color[2][]{}%
  }%
  \providecommand\transparent[1]{%
    \errmessage{(Inkscape) Transparency is used (non-zero) for the text in Inkscape, but the package 'transparent.sty' is not loaded}%
    \renewcommand\transparent[1]{}%
  }%
  \providecommand\rotatebox[2]{#2}%
  \ifx\svgwidth\undefined%
    \setlength{\unitlength}{801.52098874bp}%
    \ifx\svgscale\undefined%
      \relax%
    \else%
      \setlength{\unitlength}{\unitlength * \real{\svgscale}}%
    \fi%
  \else%
    \setlength{\unitlength}{\svgwidth}%
  \fi%
  \global\let\svgwidth\undefined%
  \global\let\svgscale\undefined%
  \makeatother%
  \begin{picture}(1,0.30903312)%
    \put(0,0){\includegraphics[width=\unitlength]{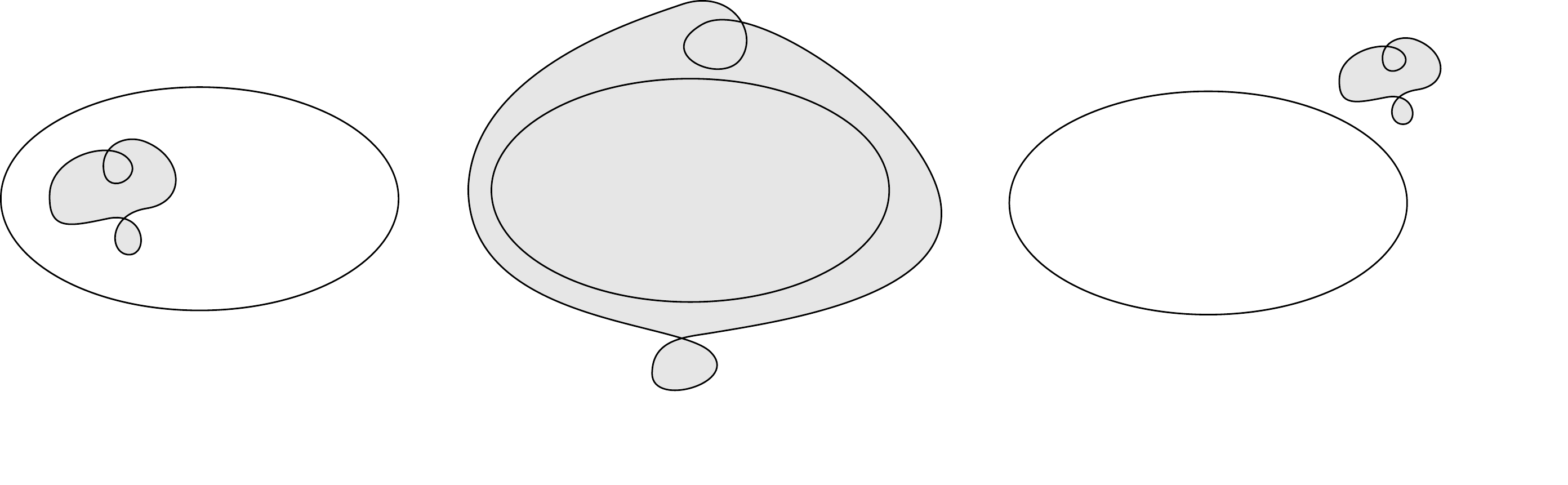}}%
    \put(0.20559516,0.23950374){\color[rgb]{0,0,0}\makebox(0,0)[lb]{\smash{\small $\sigma_0$}}}%
    \put(0.11067809,0.20453048){\color[rgb]{0,0,0}\makebox(0,0)[lb]{\smash{\small $\sigma_j$}}}%
    \put(0.57362292,0.23417315){\color[rgb]{0,0,0}\makebox(0,0)[lb]{\smash{\small $\sigma_j$}}}%
    \put(0.70256694,0.25012282){\color[rgb]{0,0,0}\makebox(0,0)[lb]{\smash{\small $\sigma_0$}}}%
    \put(0.81595196,0.26374779){\color[rgb]{0,0,0}\makebox(0,0)[lb]{\smash{\small $\sigma_j$}}}%
    \put(0.03730696,0.02141664){\color[rgb]{0,0,0}\makebox(0,0)[lb]{\smash{\small $K(\sigma_j)\subset \inter(\sigma_0)$}}}%
    \put(0.36670016,0.03861031){\color[rgb]{0,0,0}\makebox(0,0)[lb]{\smash{\small $\sigma_j \subset \ext(\sigma_0)$}}}%
    \put(0.36739599,0.00683848){\color[rgb]{0,0,0}\makebox(0,0)[lb]{\smash{\small $\sigma_0 \subset K(\sigma_j)$}}}%
    \put(0.67661948,0.0037434){\color[rgb]{0,0,0}\makebox(0,0)[lb]{\smash{\small $\sigma_0\cup K(\sigma_j) = \emptyset$}}}%
    \put(0.67755056,0.03217835){\color[rgb]{0,0,0}\makebox(0,0)[lb]{\smash{\small $\sigma_j \subset \ext(\sigma_0)$}}}%
    \put(0.39184766,0.26731356){\color[rgb]{0,0,0}\makebox(0,0)[lb]{\smash{\small $\sigma_0$}}}%
  \end{picture}%
\endgroup%
%%%
\caption{Possible relative distribution of the curve $\sigma_j$ for some 
$j=1,\ldots m$ and the curve $\sigma_0$.}
\label{fig:proof_Lemma_map_in_1}
\end{figure}
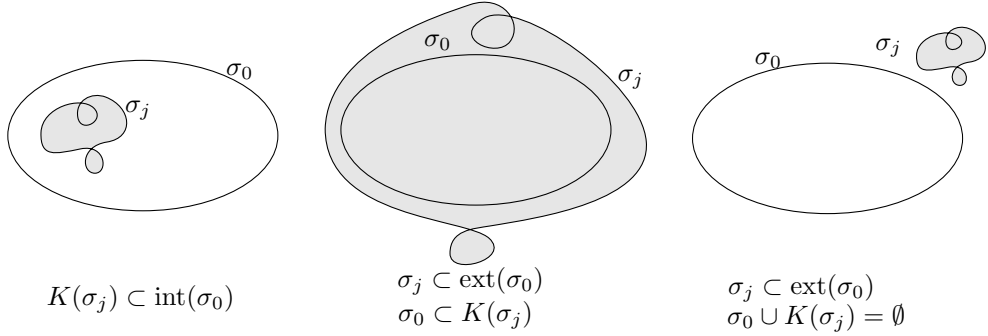
\end{center}

\subsection*{Step 4.} We show that the following conditions hold:
\begin{align}
&K(\sigma_1) \subset \inter(\sigma_0), \label{eq:sigma_in}\\
&K(\sigma_j) \subset \ext(\sigma_{j+1}) \text{ for } j = 1, \ldots, m - 2,\label{eq:in_ext}\\
&f \text{ has no poles in } K(\sigma_j)\quad \text{for } j = 1, \ldots, m -2.\label{eq:no_poles}
\end{align}
To prove it, note first that if $\sigma_j \subset K(\sigma_{j+1})$ for
some $j \in \{ 0, \ldots, m - 2\}$, then for $X_1 = \sigma_j$ we have
$X_1 \subset K(f(X_1))$, $f$ has no poles in $X_1$ and, by \eqref{eq:not_in'},
$f^2(X_1) \subset \ext(f(X_1))$, so the assumptions of Corollary~\ref{cor:mapout2} are satisfied for $X_1$. Hence,
$f$ has a weakly repelling fixed point in $K(f(X_1)) = K(f^{j+1}(\sigma_0))$,
which is contained in $\Omega$ by \eqref{eq:disjoint'}. This makes a contradiction. Hence, we have $\sigma_j \not\subset
K(\sigma_{j+1})$ for $j = 0, \ldots, m - 2$, which together with  \eqref{eq:not_in'} and \eqref{eq:disjoint''} shows $K(\sigma_1) \cap \sigma_0 = \emptyset$ and \eqref{eq:in_ext}. 

To end the proof of \eqref{eq:sigma_in}, it remains to exclude the case $K(\sigma_0) \subset \ext(\sigma_1)$. If it holds, then (since $\inter(\sigma_0)$ contains a pole of $f$), the assumptions of Corollary~\ref{cor:mapout1} are satisfied for $X = \sigma_0$. Hence, $f$ has a
weakly repelling fixed point in $K(\sigma_0) \subset \Omega$, which is a contradiction. In this way we have proved \eqref{eq:sigma_in}. 

Finally, to show \eqref{eq:no_poles}, suppose that
$f$ has a pole in $K(f^j(\sigma_0))$ for some 
$j \in \{ 1, \ldots, m - 2\}$ and take $X_2 = f^j(\sigma_0)$. Then by \eqref{eq:in_ext}, we have $K(X_2) \subset \ext(f(X_2))$, moreover $f$ has no poles in $X_2$ and $K(X_2)$ contains a pole of $f$, so by
Corollary~\ref{cor:mapout1} for $X_2$, the map $f$
has a weakly repelling fixed point in $K(X_2) = K(f^j(\sigma_0))$, which is
contained in $\Omega$ by \eqref{eq:disjoint'}. This is a contradiction. Hence, the assertion \eqref{eq:no_poles} is proved.

Notice that by \eqref{eq:disjoint'}, \eqref{eq:disjoint''} and \eqref{eq:sigma_in}, there exists $1 \le k \le m-1$, such that 
\begin{equation}\label{eq:f^k}
K(\sigma_j) \subset \inter(\sigma_0) \text{ for } j = 1, \ldots, k
\quad \text{and} \quad \sigma_{k+1} \subset \ext(\sigma_0)
\end{equation}
(see Figure \ref{fig:proof_Lemma_map_in_2}). 
\begin{figure}[hbt!]
\def\svgwidth{0.45\textwidth}
%%%
\begingroup%
  \makeatletter%
  \providecommand\color[2][]{%
    \errmessage{(Inkscape) Color is used for the text in Inkscape, but the package 'color.sty' is not loaded}%
    \renewcommand\color[2][]{}%
  }%
  \providecommand\transparent[1]{%
    \errmessage{(Inkscape) Transparency is used (non-zero) for the text in Inkscape, but the package 'transparent.sty' is not loaded}%
    \renewcommand\transparent[1]{}%
  }%
  \providecommand\rotatebox[2]{#2}%
  \ifx\svgwidth\undefined%
    \setlength{\unitlength}{576.07617244bp}%
    \ifx\svgscale\undefined%
      \relax%
    \else%
      \setlength{\unitlength}{\unitlength * \real{\svgscale}}%
    \fi%
  \else%
    \setlength{\unitlength}{\svgwidth}%
  \fi%
  \global\let\svgwidth\undefined%
  \global\let\svgscale\undefined%
  \makeatother%
  \begin{picture}(1,0.58425626)%
    \put(0,0){\includegraphics[width=\unitlength]{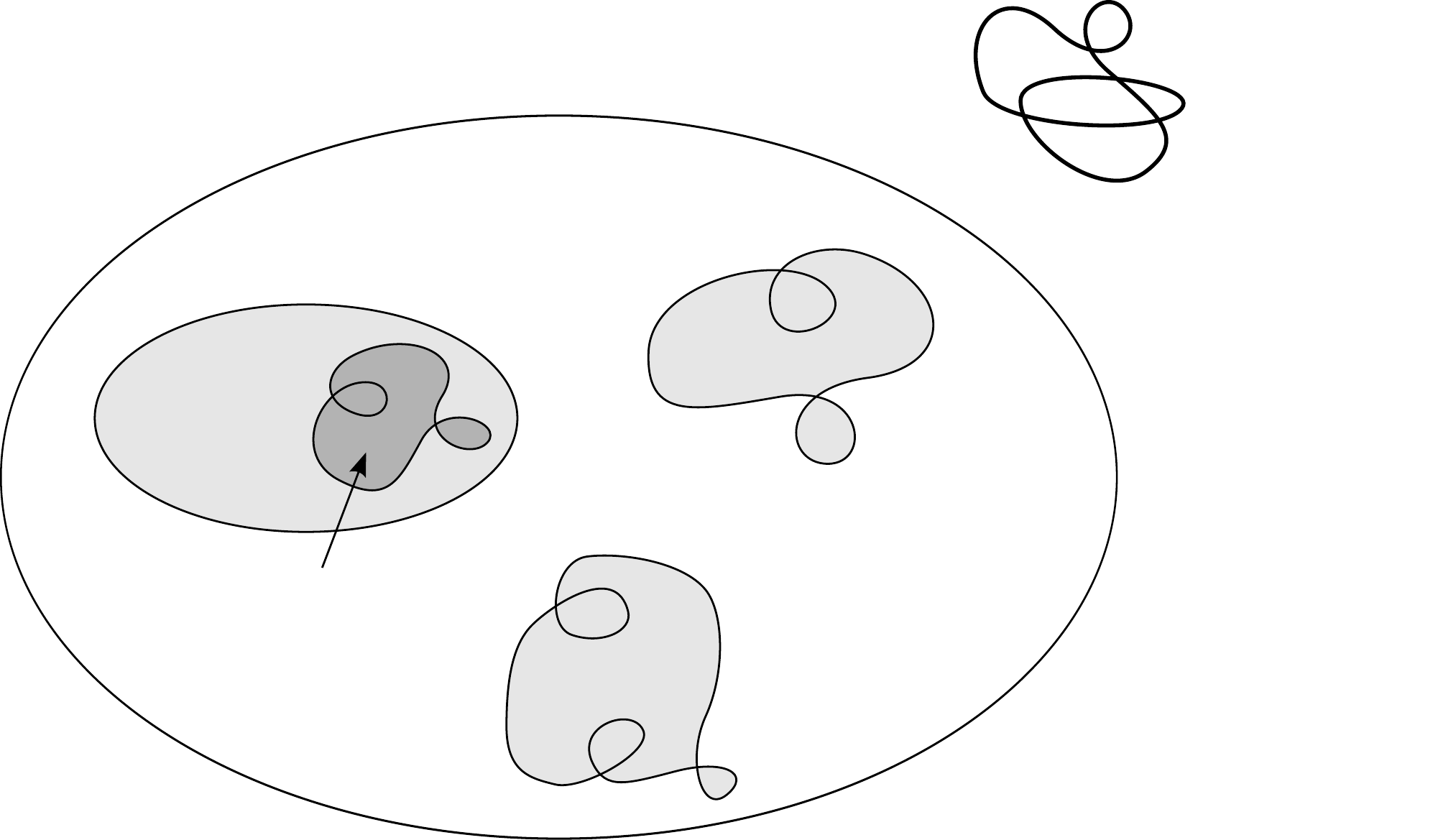}}%
    \put(0.09705251,0.45660485){\color[rgb]{0,0,0}\makebox(0,0)[lb]{\smash{\small $\sigma_0$}}}%
    \put(0.18426357,0.38096739){\color[rgb]{0,0,0}\makebox(0,0)[lb]{\smash{\small $\sigma_1$}}}%
    \put(0.08085618,0.29561944){\color[rgb]{0,0,0}\makebox(0,0)[lb]{\smash{\tiny $K(\sigma_1)$}}}%
    \put(0.50694476,0.1320606){\color[rgb]{0,0,0}\makebox(0,0)[lb]{\smash{\tiny $K(\sigma_2)$}}}%
    \put(0.17678845,0.15590691){\color[rgb]{0,0,0}\makebox(0,0)[lb]{\smash{\tiny $K(\sigma_3)$}}}%
    \put(0.45337226,0.42163515){\color[rgb]{0,0,0}\makebox(0,0)[lb]{\smash{\tiny $K(\sigma_k)$}}}%
    \put(0.76011864,0.42356616){\color[rgb]{0,0,0}\makebox(0,0)[lb]{\smash{\small  $\sigma_{k+1}$}}}%
    \put(0.36445252,0.3071855){\color[rgb]{0,0,0}\rotatebox{17.33948562}{\makebox(0,0)[lb]{\smash{$\dots$}}}}%
  \end{picture}%
\endgroup%
%%%
\hfil
\def\svgwidth{0.45\textwidth}
%%%%
\begingroup%
  \makeatletter%
  \providecommand\color[2][]{%
    \errmessage{(Inkscape) Color is used for the text in Inkscape, but the package 'color.sty' is not loaded}%
    \renewcommand\color[2][]{}%
  }%
  \providecommand\transparent[1]{%
    \errmessage{(Inkscape) Transparency is used (non-zero) for the text in Inkscape, but the package 'transparent.sty' is not loaded}%
    \renewcommand\transparent[1]{}%
  }%
  \providecommand\rotatebox[2]{#2}%
  \ifx\svgwidth\undefined%
    \setlength{\unitlength}{599.40039597bp}%
    \ifx\svgscale\undefined%
      \relax%
    \else%
      \setlength{\unitlength}{\unitlength * \real{\svgscale}}%
    \fi%
  \else%
    \setlength{\unitlength}{\svgwidth}%
  \fi%
  \global\let\svgwidth\undefined%
  \global\let\svgscale\undefined%
  \makeatother%
  \begin{picture}(1,0.57897496)%
    \put(0,0){\includegraphics[width=\unitlength]{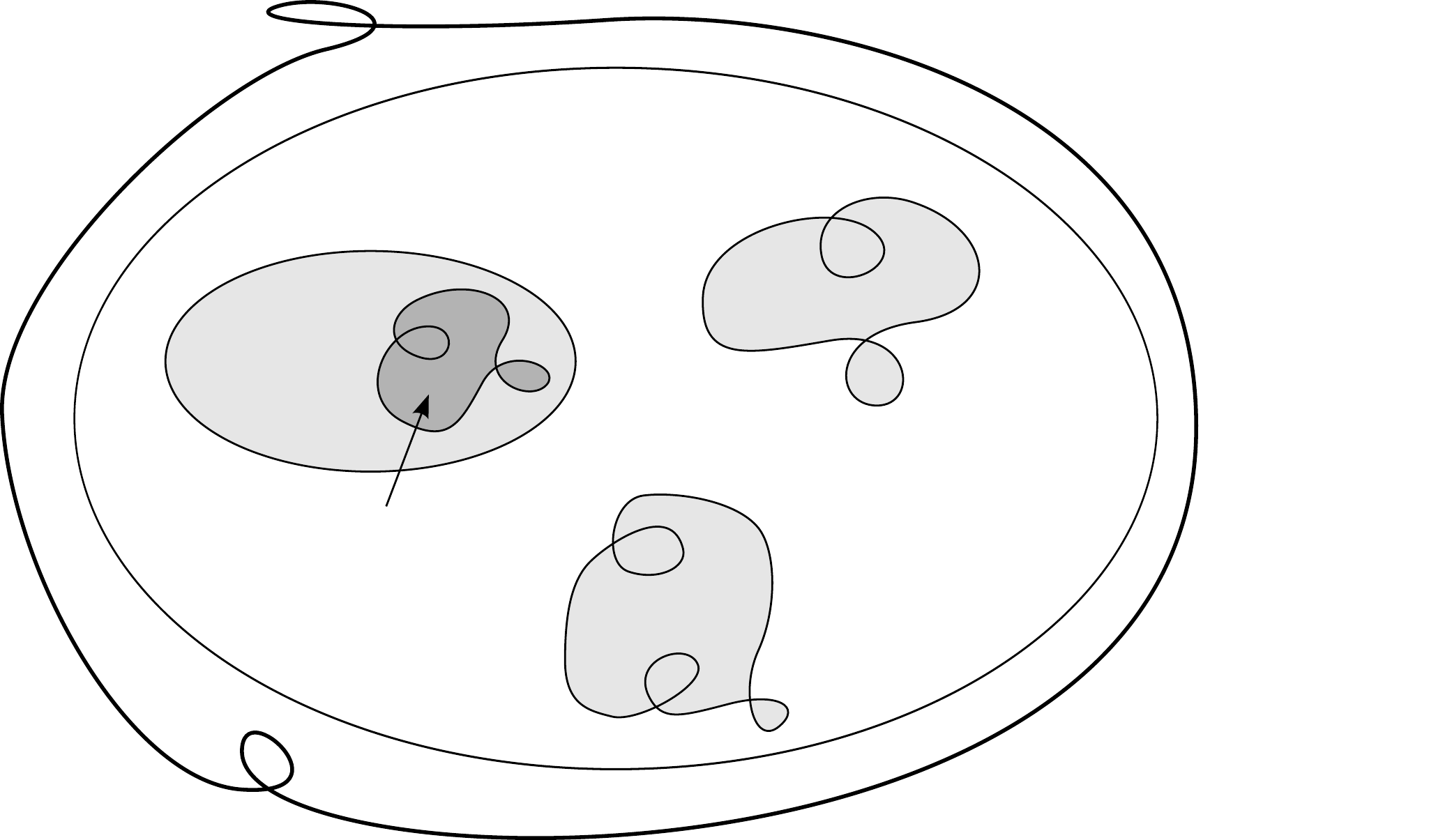}}%
    \put(0.34655516,0.49828593){\color[rgb]{0,0,0}\makebox(0,0)[lb]{\smash{\small $\sigma_0$}}}%
    \put(0.22774589,0.41418645){\color[rgb]{0,0,0}\makebox(0,0)[lb]{\smash{\small $\sigma_1$}}}%
    \put(0.12836235,0.33215962){\color[rgb]{0,0,0}\makebox(0,0)[lb]{\smash{\tiny $K(\sigma_1)$}}}%
    \put(0.53787072,0.17496528){\color[rgb]{0,0,0}\makebox(0,0)[lb]{\smash{\tiny $K(\sigma_2)$}}}%
    \put(0.22056163,0.20211854){\color[rgb]{0,0,0}\makebox(0,0)[lb]{\smash{\tiny $K(\sigma_3)$}}}%
    \put(0.48638286,0.45327173){\color[rgb]{0,0,0}\makebox(0,0)[lb]{\smash{\tiny $K(\sigma_k)$}}}%
    \put(0.67607674,0.53373398){\color[rgb]{0,0,0}\makebox(0,0)[lb]{\smash{\small $\sigma_{k+1}$}}}%
    \put(0.41700982,0.33293158){\color[rgb]{0,0,0}\rotatebox{15.16635796}{\makebox(0,0)[lb]{\smash{$\dots$}}}}%
  \end{picture}%
\endgroup%
%%%%
\caption{ Two possible relative positions of $\sigma_k=f^{k}(\sigma_0)$ and $\sigma_0$ under the condition 
\eqref{eq:f^k}. In both cases, $\sigma_{k+1} \subset \ext(\sigma_0)$.}
\label{fig:proof_Lemma_map_in_2}
\end{figure}
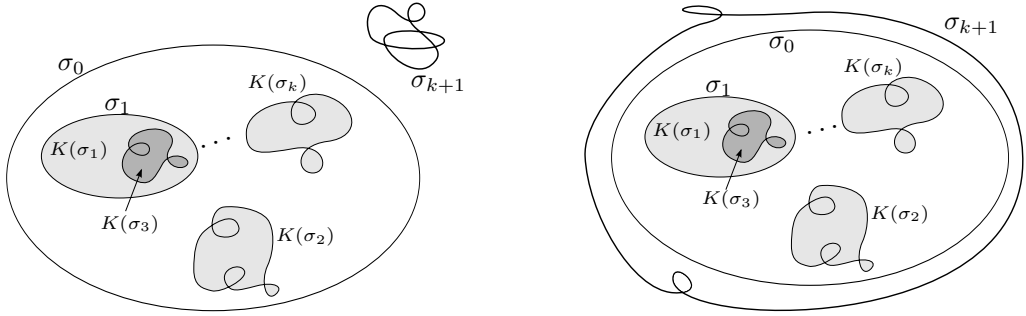

\subsection*{Step 5.} We show
\begin{equation}\label{eq:f^k-1}
f(K(\sigma_k)) \subset \ext(\sigma_0).
\end{equation}
To see it, suppose otherwise, i.e.~$f(K(\sigma_k)) \not\subset \ext(\sigma_0)$
(see Figure~\ref{fig:proof_Lemma_map_in_3}). 
%%%
\begin{figure}[hbt!]
\def\svgwidth{0.45\textwidth}
%%%
\begingroup%
  \makeatletter%
  \providecommand\color[2][]{%
    \errmessage{(Inkscape) Color is used for the text in Inkscape, but the package 'color.sty' is not loaded}%
    \renewcommand\color[2][]{}%
  }%
  \providecommand\transparent[1]{%
    \errmessage{(Inkscape) Transparency is used (non-zero) for the text in Inkscape, but the package 'transparent.sty' is not loaded}%
    \renewcommand\transparent[1]{}%
  }%
  \providecommand\rotatebox[2]{#2}%
  \ifx\svgwidth\undefined%
    \setlength{\unitlength}{602.93344594bp}%
    \ifx\svgscale\undefined%
      \relax%
    \else%
      \setlength{\unitlength}{\unitlength * \real{\svgscale}}%
    \fi%
  \else%
    \setlength{\unitlength}{\svgwidth}%
  \fi%
  \global\let\svgwidth\undefined%
  \global\let\svgscale\undefined%
  \makeatother%
  \begin{picture}(1,0.58253966)%
    \put(0,0){\includegraphics[width=\unitlength]{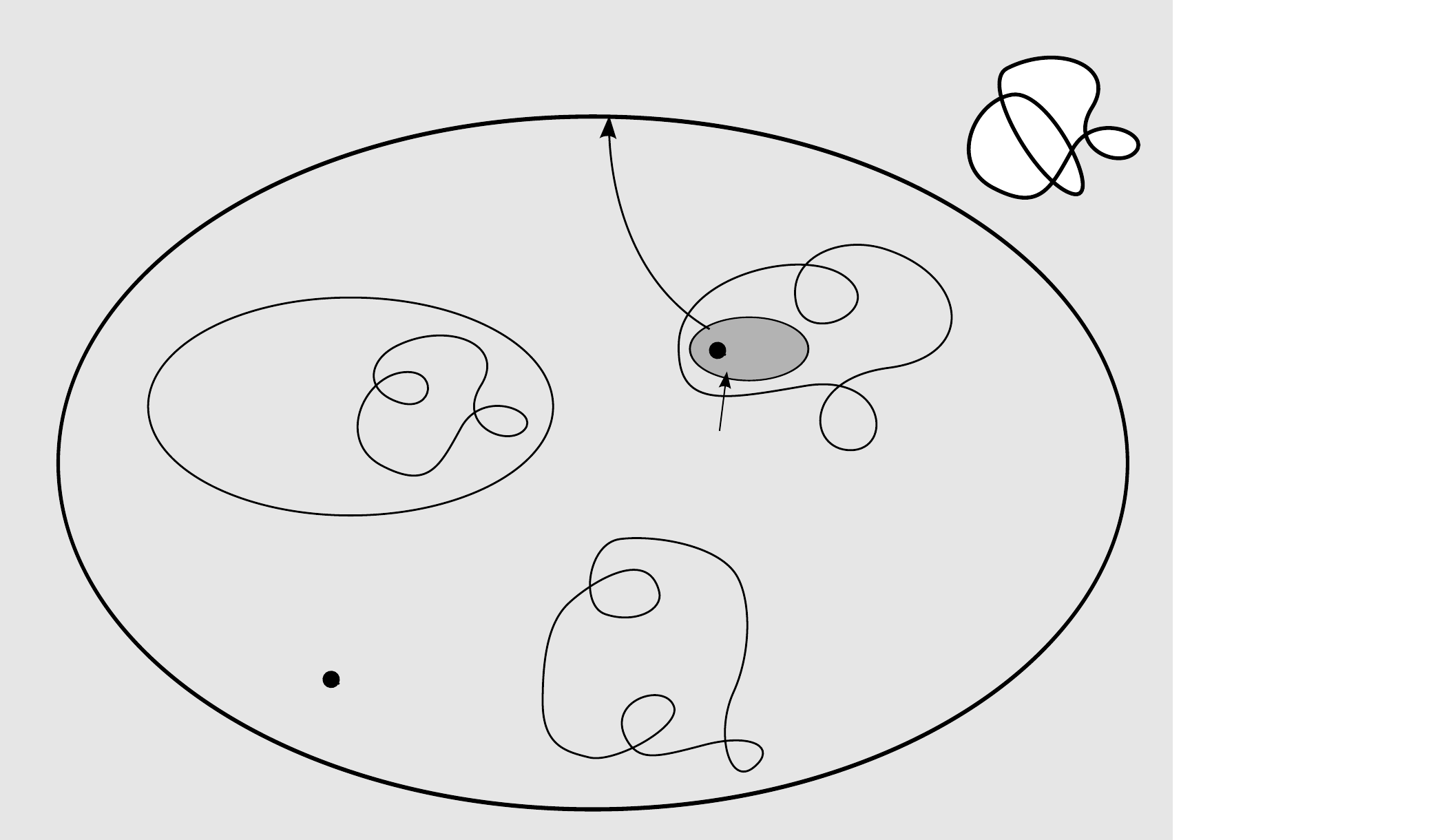}}%
    \put(0.11657145,0.45124969){\color[rgb]{0,0,0}\makebox(0,0)[lb]{\smash{\small $\sigma_0$}}}%
    \put(0.21581978,0.38428888){\color[rgb]{0,0,0}\makebox(0,0)[lb]{\smash{\small $\sigma_1$}}}%
    \put(0.4729412,0.42314513){\color[rgb]{0,0,0}\makebox(0,0)[lb]{\smash{\tiny $K(\sigma_k)$}}}%
    \put(0.75061919,0.41647502){\color[rgb]{0,0,0}\makebox(0,0)[lb]{\smash{\small $\sigma_{k+1}$}}}%
    \put(0.51005108,0.33268698){\color[rgb]{0,0,0}\makebox(0,0)[lb]{\smash{\tiny $z_0$}}}%
    \put(0.24067661,0.10570131){\color[rgb]{0,0,0}\makebox(0,0)[lb]{\smash{\tiny $f(z_0)$}}}%
    \put(0.64575132,0.15787319){\color[rgb]{0,0,0}\makebox(0,0)[lb]{\smash{$D_1$}}}%
    \put(0.47165889,0.24614621){\color[rgb]{0,0,0}\makebox(0,0)[lb]{\smash{$\Omega_1$}}}%
    \put(0.38737081,0.30212469){\color[rgb]{0,0,0}\rotatebox{18.68561467}{\makebox(0,0)[lb]{\smash{$\dots$}}}}%
  \end{picture}%
\endgroup%
%%%
\hfil
\def\svgwidth{0.45\textwidth}
%%%
\begingroup%
  \makeatletter%
  \providecommand\color[2][]{%
    \errmessage{(Inkscape) Color is used for the text in Inkscape, but the package 'color.sty' is not loaded}%
    \renewcommand\color[2][]{}%
  }%
  \providecommand\transparent[1]{%
    \errmessage{(Inkscape) Transparency is used (non-zero) for the text in Inkscape, but the package 'transparent.sty' is not loaded}%
    \renewcommand\transparent[1]{}%
  }%
  \providecommand\rotatebox[2]{#2}%
  \ifx\svgwidth\undefined%
    \setlength{\unitlength}{560.89567264bp}%
    \ifx\svgscale\undefined%
      \relax%
    \else%
      \setlength{\unitlength}{\unitlength * \real{\svgscale}}%
    \fi%
  \else%
    \setlength{\unitlength}{\svgwidth}%
  \fi%
  \global\let\svgwidth\undefined%
  \global\let\svgscale\undefined%
  \makeatother%
  \begin{picture}(1,0.6187208)%
    \put(0,0){\includegraphics[width=\unitlength]{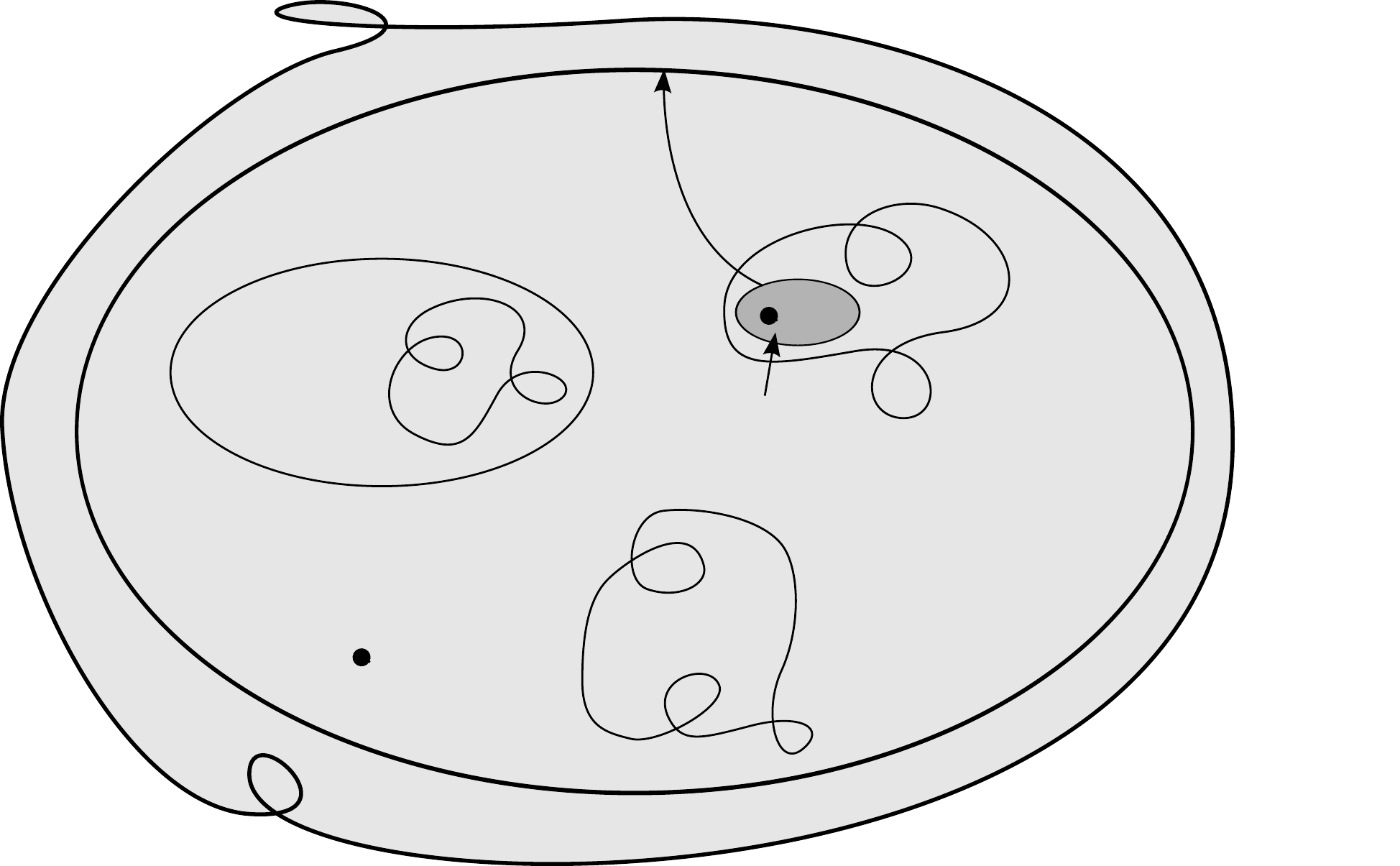}}%
    \put(0.23093253,0.49038924){\color[rgb]{0,0,0}\makebox(0,0)[lb]{\smash{\small $\sigma_0$}}}%
    \put(0.12947473,0.34413889){\color[rgb]{0,0,0}\makebox(0,0)[lb]{\smash{\small $\sigma_1$}}}%
    \put(0.51977239,0.48438822){\color[rgb]{0,0,0}\makebox(0,0)[lb]{\smash{\tiny $K(\sigma_k)$}}}%
    \put(0.76742635,0.53138091){\color[rgb]{0,0,0}\makebox(0,0)[lb]{\smash{\small $\sigma_{k+1}$}}}%
    \put(0.55966357,0.38715046){\color[rgb]{0,0,0}\makebox(0,0)[lb]{\smash{\tiny $z_0$}}}%
    \put(0.27010012,0.14315276){\color[rgb]{0,0,0}\makebox(0,0)[lb]{\smash{\tiny $f(z_0)$}}}%
    \put(0.69663582,0.18712473){\color[rgb]{0,0,0}\makebox(0,0)[lb]{\smash{$D_1$}}}%
    \put(0.51839397,0.29412366){\color[rgb]{0,0,0}\makebox(0,0)[lb]{\smash{$\Omega_1$}}}%
    \put(0.43171249,0.35838661){\color[rgb]{0,0,0}\rotatebox{19.82529035}{\makebox(0,0)[lb]{\smash{$\dots$}}}}%
  \end{picture}%
\endgroup%
%%%
\caption{Sketch of Step 5.}
\label{fig:proof_Lemma_map_in_3}
\end{figure}
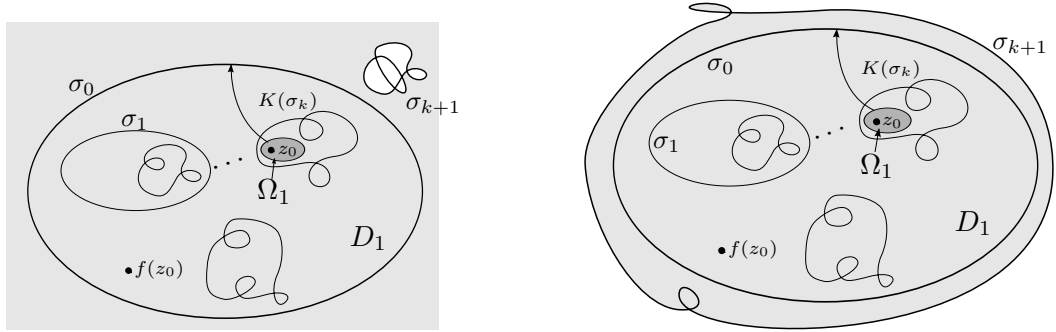
Then there
exists $z_0 \in K(\sigma_k)$ such that $f(z_0) \in K(\sigma_0)$. 
By \eqref{eq:f^k}, we have 
\[
z_0 \in \Omega_1 \subset \overline{\Omega_1} \subset K(\sigma_k)
\]
for some bounded simply connected component $\Omega_1$ of $\chat
\setminus \sigma_k$. We have $\bd \Omega_1 \subset \sigma_k$, so
$f(\bd \Omega_1) \subset \sigma_{k+1}$, which together with \eqref{eq:f^k}
implies $\overline{\Omega_1} \subset K(\sigma_0) \subset D_1$ for some component
$D_1$ of $\chat \setminus f(\bd \Omega_1)$. Moreover, $z_0 \in \Omega_1$ and
$f(z_0) \in K(\sigma_0) \subset D_1$. Hence, the assumptions of Lemma~\ref{mapout} are satisfied for
$\Omega_1, D_1, z_0$, so $f$ has a weakly repelling fixed point in
$\Omega_1$, which is contained in $\Omega$ by \eqref{eq:disjoint'}. This makes a contradiction. Therefore, \eqref{eq:f^k-1} is satisfied.

\subsection*{\bf Step 6.} We check that we are under the assumptions of Shishikura's Theorem~\ref{shishikura}. Let 
\[
V_0 = \ext(\sigma_0), \qquad V_1 = \inter(\sigma_1),
\]
and let us check that $V_0, V_1$ satisfy the required assumptions. By definition, $V_0, V_1$ are simply connected and $f(\bd V_0) = \bd
V_1$. Since $f$ is a covering in some neighbourhood $N$ of $\sigma_0 = \bd V_0$,
we have
\[
f(V_0 \cap N) = f(N \setminus \overline{D'}) \subset \C \setminus
\overline{D} = V_1.
\]
By \eqref{eq:f^k}, $K(f^j(\bd V_1)) \subset \C \setminus
\overline{V_0}$ for $j = 0, \ldots, k-1$ and $f^k(\bd V_1) \subset V_0$. Moreover, by \eqref{eq:no_poles} and \eqref{eq:f^k-1}, the
map $f^k$ is defined on $\overline{V_1}$ and
\[
f^j(\overline{V_1}) \subset K(f^j(\bd V_1)) \subset \C \setminus \overline{V_0}
\text{ for } j = 0, \ldots, k-1 \quad \text{and} \quad f^k(\overline{V_1}) \subset V_0.
\]
See Figure \ref{fig:proof_Lemma_map_in_3-2}. 
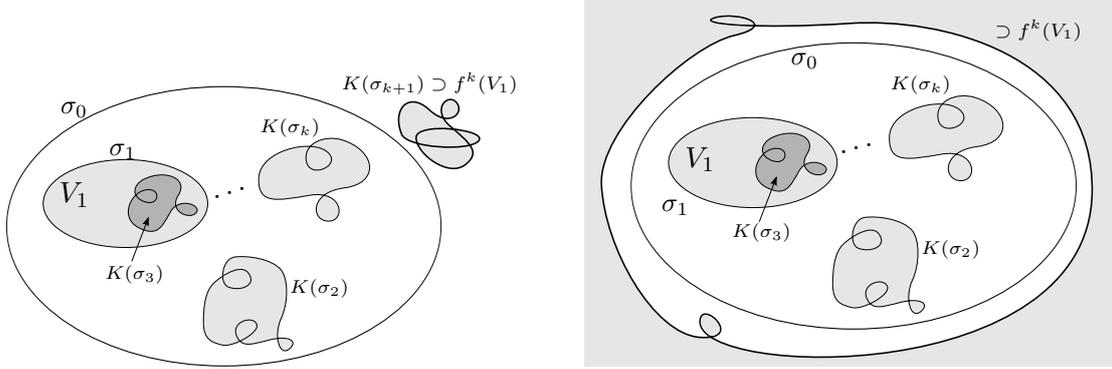
\begin{figure}[hbt!]
\def\svgwidth{0.45\textwidth}
%%%
\begingroup%
  \makeatletter%
  \providecommand\color[2][]{%
    \errmessage{(Inkscape) Color is used for the text in Inkscape, but the package 'color.sty' is not loaded}%
    \renewcommand\color[2][]{}%
  }%
  \providecommand\transparent[1]{%
    \errmessage{(Inkscape) Transparency is used (non-zero) for the text in Inkscape, but the package 'transparent.sty' is not loaded}%
    \renewcommand\transparent[1]{}%
  }%
  \providecommand\rotatebox[2]{#2}%
  \ifx\svgwidth\undefined%
    \setlength{\unitlength}{547.81668389bp}%
    \ifx\svgscale\undefined%
      \relax%
    \else%
      \setlength{\unitlength}{\unitlength * \real{\svgscale}}%
    \fi%
  \else%
    \setlength{\unitlength}{\svgwidth}%
  \fi%
  \global\let\svgwidth\undefined%
  \global\let\svgscale\undefined%
  \makeatother%
  \begin{picture}(1,0.55089575)%
    \put(0,0){\includegraphics[width=\unitlength]{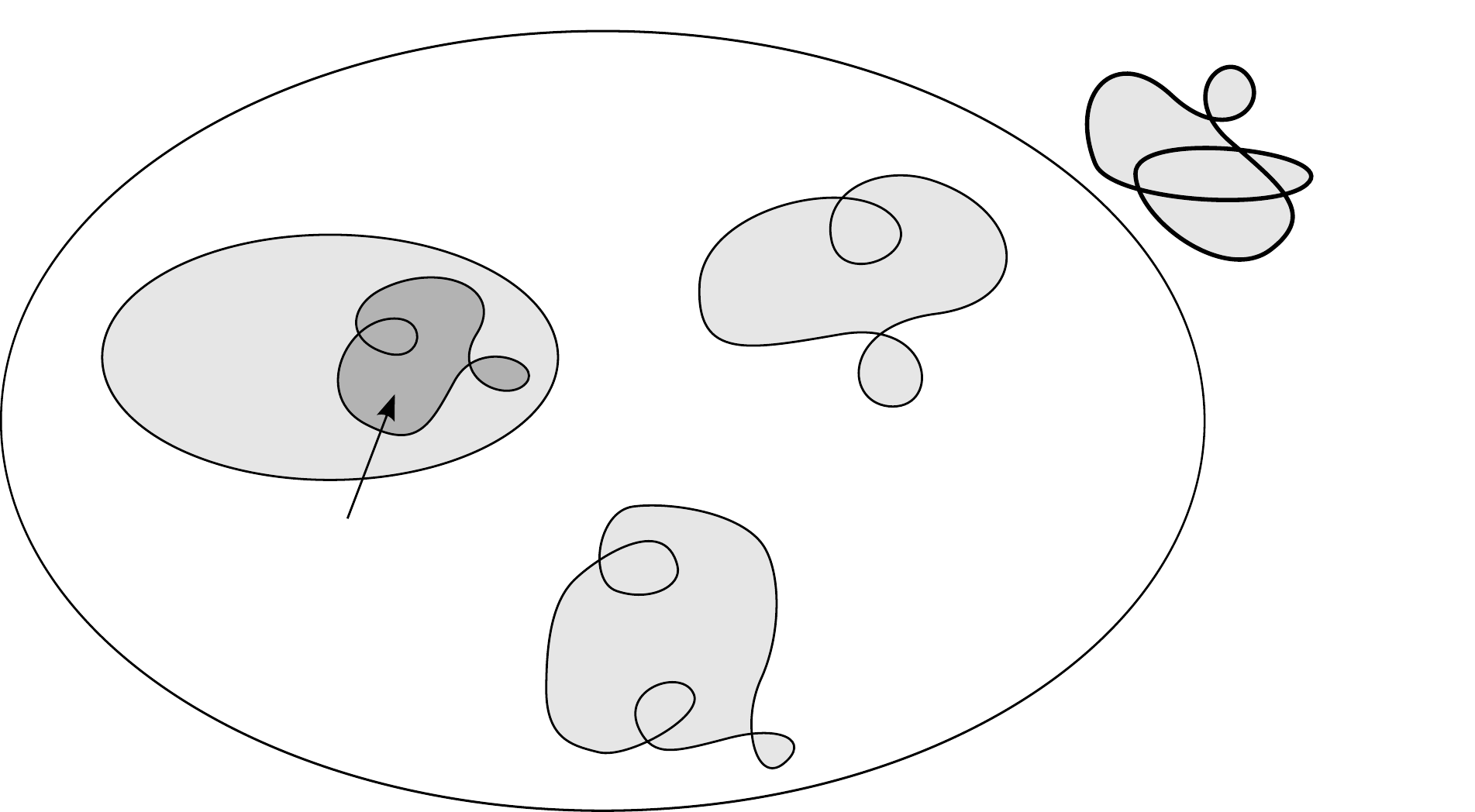}}%
    \put(0.10205903,0.48015911){\color[rgb]{0,0,0}\makebox(0,0)[lb]{\smash{\small $\sigma_0$}}}%
    \put(0.19376893,0.40061984){\color[rgb]{0,0,0}\makebox(0,0)[lb]{\smash{\small $\sigma_1$}}}%
    \put(0.09963066,0.31086917){\color[rgb]{0,0,0}\makebox(0,0)[lb]{\smash{$V_1$}}}%
    \put(0.53309584,0.13887304){\color[rgb]{0,0,0}\makebox(0,0)[lb]{\smash{\tiny $K(\sigma_2)$}}}%
    \put(0.18590814,0.16600658){\color[rgb]{0,0,0}\makebox(0,0)[lb]{\smash{\tiny $K(\sigma_3)$}}}%
    \put(0.47675977,0.44338548){\color[rgb]{0,0,0}\makebox(0,0)[lb]{\smash{\tiny $K(\sigma_k)$}}}%
    \put(0.63075566,0.52467089){\color[rgb]{0,0,0}\makebox(0,0)[lb]{\smash{\tiny  $K(\sigma_{k+1})\supset f^k(V_1)$}}}%
    \put(0.3898976,0.31708409){\color[rgb]{0,0,0}\rotatebox{16.68976894}{\makebox(0,0)[lb]{\smash{$\dots$}}}}%
  \end{picture}%
\endgroup%
%%%
\hfil
\def\svgwidth{0.45\textwidth}
\begingroup%
  \makeatletter%
  \providecommand\color[2][]{%
    \errmessage{(Inkscape) Color is used for the text in Inkscape, but the package 'color.sty' is not loaded}%
    \renewcommand\color[2][]{}%
  }%
  \providecommand\transparent[1]{%
    \errmessage{(Inkscape) Transparency is used (non-zero) for the text in Inkscape, but the package 'transparent.sty' is not loaded}%
    \renewcommand\transparent[1]{}%
  }%
  \providecommand\rotatebox[2]{#2}%
  \ifx\svgwidth\undefined%
    \setlength{\unitlength}{534.6996582bp}%
    \ifx\svgscale\undefined%
      \relax%
    \else%
      \setlength{\unitlength}{\unitlength * \real{\svgscale}}%
    \fi%
  \else%
    \setlength{\unitlength}{\svgwidth}%
  \fi%
  \global\let\svgwidth\undefined%
  \global\let\svgscale\undefined%
  \makeatother%
  \begin{picture}(1,0.69664433)%
    \put(0,0){\includegraphics[width=\unitlength]{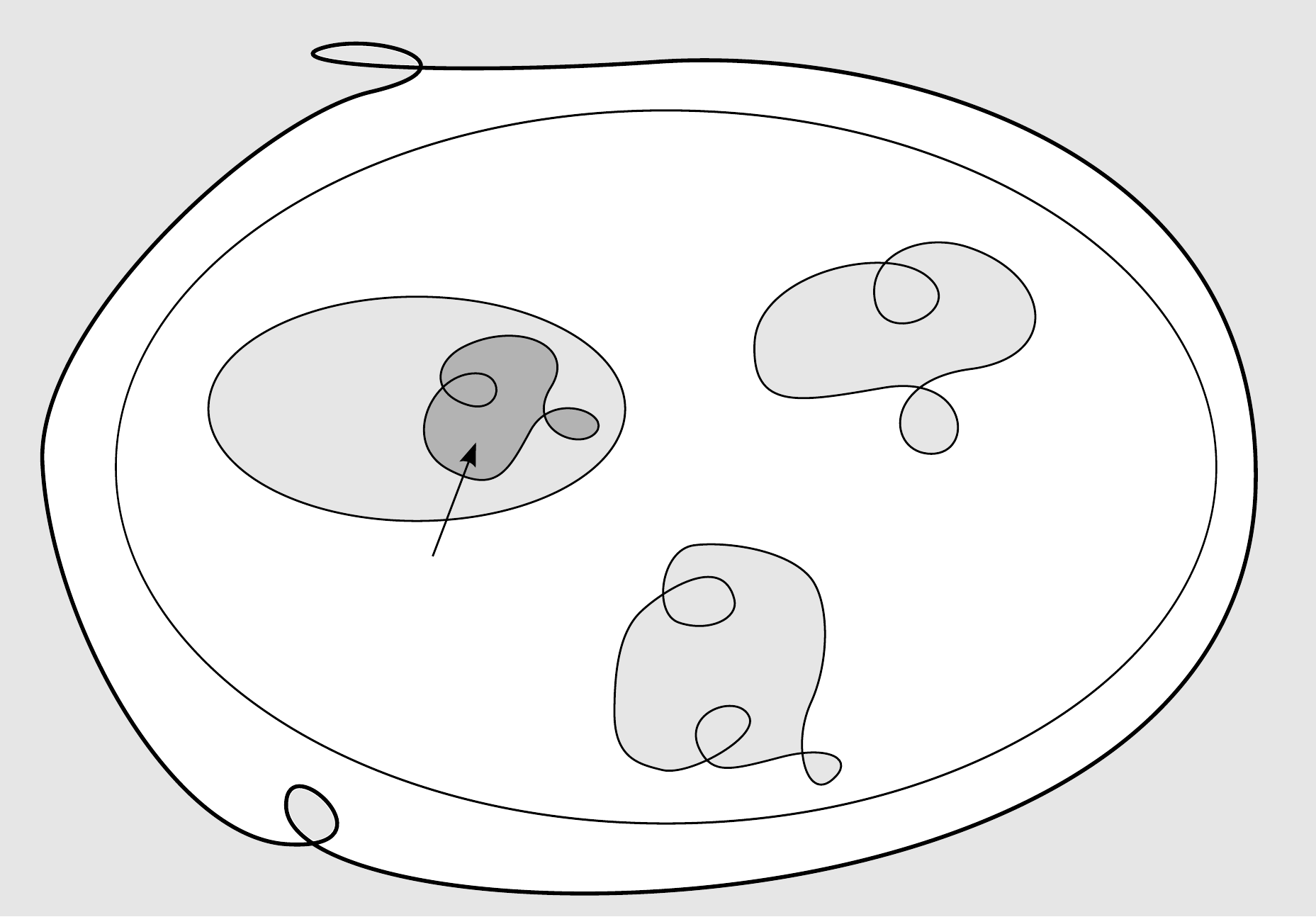}}%
    \put(0.38870249,0.57130799){\color[rgb]{0,0,0}\makebox(0,0)[lb]{\smash{\small $\sigma_0$}}}%
    \put(0.14551229,0.2948649){\color[rgb]{0,0,0}\makebox(0,0)[lb]{\smash{\small $\sigma_1$}}}%
    \put(0.18971535,0.37779652){\color[rgb]{0,0,0}\makebox(0,0)[lb]{\smash{$V_1$}}}%
    \put(0.63355714,0.21218078){\color[rgb]{0,0,0}\makebox(0,0)[lb]{\smash{\tiny $K(\sigma_2)$}}}%
    \put(0.27785238,0.24419501){\color[rgb]{0,0,0}\makebox(0,0)[lb]{\smash{\tiny $K(\sigma_3)$}}}%
    \put(0.57583905,0.52416339){\color[rgb]{0,0,0}\makebox(0,0)[lb]{\smash{\tiny $K(\sigma_k)$}}}%
    \put(0.77209022,0.6254906){\color[rgb]{0,0,0}\makebox(0,0)[lb]{\smash{\tiny $\supset f^k(V_1)$}}}%
    \put(0.4809968,0.40254202){\color[rgb]{0,0,0}\rotatebox{15.36194887}{\makebox(0,0)[lb]{\smash{$\dots$}}}}%
  \end{picture}%
\endgroup%
\caption{ Sketch of Step 6.}
\label{fig:proof_Lemma_map_in_3-2}
\end{figure}
Hence, Shishikura's Theorem~\ref{shishikura} concludes that $f$ has a weakly repelling point in $\clC \setminus\overline{V_0} = \inter(\sigma_0) \subset \Omega$, which finishes the proof.
\end{proof}

%%%%%%%%%%%%%%%%%%%%%%%%%%%%%%%%%%%%%%%%%%%%%
%%%%%%%%%%%%%%%%%%%%%%%%%%%%%%%%%%%%%%%%%%%%%
%%%%%%%%%%%%%%%%%%%%%%%%%%%%%%%%%%%%%%%%%%%%%

\section{Proof of Theorem B}\label{sec:proof_thmB}

Let $f: \C \to \clC$ be a transcendental meromorphic map and let $U_0, \ldots,
U_{p-1}$ be a periodic cycle of Baker domains of $f$ of
(minimal) period $p \ge 1$. Recall that for $j = 0, \ldots, p-1$ we have $f^{pn}
\to \zeta_j$
locally uniformly on $U_j$ as $n \to \infty$ for some $\zeta_j \in \clC$ such that
$\zeta_j = \infty$ for at least one $j$. Renumbering the Baker domains, we may
assume $\zeta_0 = \infty$, i.e.~the domain $U_0$ is unbounded and
\[
f^{pn}(z) \to \infty \quad \text{for } z\in U_0 \quad \text{as } n \to \infty. 
\]

As the first step in the proof of Theorem~B we show a technical lemma which allows us to discard some of the possible configurations of the $U_j$'s. 
More precisely, we show that under certain relative positions of the $U_j$'s the existence of a weakly repelling fixed point follows directly from the results in Section~\ref{section:configurations}.

\begin{lem}[\bf{Configurations of Baker domains}]\label{lem:U_0}
Suppose there exist a simply connected bounded domain
$\Omega\subset \C$ and a pole $p_0$ of $f$, such that
\[
p_0 \in \Omega \quad \text{and} \quad \bd\Omega \subset U_j
\]
for some $j = 0, \ldots, p-1$. Then either $f$ has a weakly
repelling fixed point or there exist $n \ge 0$ and a bounded simply connected
domain $\Omega_0\subset \C$, such that
\[
p_0 \in \Omega_0 \quad \text{and} \quad \bd\Omega_0 \subset f^n(\bd\Omega)\subset U_0.
\]
\end{lem}
\begin{proof} If $p = 1$ then we can take $n = 0$ and $\Omega_0 = \Omega$. Hence, in what follows we assume $p > 1$. 

Since $p>1$, it is clear that $\bd\Omega, f(\bd\Omega),
\ldots, f^{p-1}(\bd\Omega)$ are pairwise disjoint and we cannot have
\[
K(\bd\Omega) \subset K(f(\bd\Omega)) \subset
\cdots \subset K(f^p(\bd\Omega)),
\]
because it would contradict the connectedness of $U_j$. Thus, there is a minimal $n \ge 0$ such that 
\begin{equation}\label{eq:not_in_K}
K(f^n(\bd\Omega)) \not\subset K(f^{n+1}(\bd\Omega)).
\end{equation}
Note that we have $p_0 \in K(f^n(\bd\Omega)) \setminus
f^n(\bd\Omega)$. Hence, there exists a bounded component $\Omega_0$ of
$\clC \setminus f^n(\bd\Omega)$, such that $p_0 \in \Omega_0$. Since $\Omega$
is simply connected, $\Omega_0$ is also simply connected. 

As $\bd\Omega_0 \cap f(\bd\Omega_0) = \emptyset$, one of the
three possibilities holds: $\overline{\Omega_0} \subset K(
f(\bd\Omega_0))$, $\overline{\Omega_0}
\subset \ext(f(\bd \Omega_0))$ or $f(\bd\Omega_0) \subset \Omega_0$. Since
$\bd\Omega_0 \subset f^n(\bd\Omega)$ and $f(\bd\Omega_0) \cap f^n(\bd\Omega) = \emptyset$, the first possibility does not occur by \eqref{eq:not_in_K}.
If the second possibility holds, then the assumptions of
Corollary~\ref{cor:mapout1} are satisfied for $X = \bd \Omega_0$, so $f$ has a weakly repelling fixed point. Hence, we are left with the third possibility, i.e.~$f(\bd\Omega_0) \subset \Omega_0$.

Note that $\bd\Omega_0, f(\bd\Omega_0),
\ldots, f^{p-1}(\bd\Omega_0)$ are pairwise disjoint. Therefore, if there exists a (minimal)
number $2 \leq m \leq p-1$ such that $f^m(\bd\Omega_0) \not\subset \Omega_0$,
then $f(\bd\Omega_0), \ldots, f^{m-1}(\bd\Omega_0) \subset \Omega_0$ and
$f^m(\bd\Omega_0) \cap \overline{\Omega_0} = \emptyset$, so the assumptions of
Proposition~\ref{mapin} are fulfilled for $\Omega_0$ and we conclude that $f$ has a weakly
repelling fixed point in that case. Thus, we can assume $f(\bd\Omega_0), \ldots,
f^{p-1}(\bd\Omega_0) \subset \Omega_0$. This implies $\bd\Omega_0 \subset U_0$,
because otherwise $\bd\Omega_0 \cap U_0 = \emptyset$ and one of the sets
$f(\bd\Omega_0), \ldots, f^{p-1}(\bd\Omega_0)$ is contained in $U_0$, which
contradicts the fact that $U_0$ is connected and unbounded. Hence, $\Omega_0$ satisfies the assertion of the lemma.
\end{proof}

Let $W \subset U_0$ be an absorbing domain which
exists according to Corollary~A' (for the map $F = f^p$ and $\zeta =\infty$). Note that $W$ is unbounded and does not contain poles of $f$. 
The proof of Theorem~B splits into two cases depending on the connectivity of $W$.

\subsection*{\bf Case 1. $\boldsymbol{W}$ is not simply connected.}

\

Under this assumption we can take a closed curve
\[
\gamma \subset W,
\]
such that $K(\gamma) \cap J(f) \neq \emptyset$. Notice that, because of Corollary~A', $f^{\ell p}(\gamma)\subset W$ for all $\ell\geq 0$.

By Lemma~\ref{poles-in-holes}, there exists $n_0 \ge 0$ and a
pole $p_0$ of $f$, such that $p_0 \in K(f^{n_0}(\gamma))$. Then $p_0$ is in a
bounded simply connected component $\Omega$ of $\clC\setminus f^{n_0}(\gamma)$, such that $\bd\Omega \subset f^{n_0}(W)$.
By Lemma~\ref{lem:U_0}, we may reduce the proof to the case when there exists a bounded
simply connected domain $\Omega_0$ with 
\[
\bd\Omega_0 \subset f^{n_1}(\bd\Omega) \subset f^{n_1}(\gamma) \subset U_0 \cap f^{n_0+n_1}(W) 
\]
for some $n_1 \ge 0$, such that $p_0 \in \Omega_0$. In particular, this implies that $n_0+n_1 = \ell p$ for some $\ell \ge 0$, so by Corollary~A' we have $f^{n_0+n_1}(W) \subset W$, which implies $\bd \Omega_0 \subset W$. We conclude that
there exists a bounded component $\Omega_1$ of $\C \setminus
\overline{W}$, such that $p_0 \in \Omega_1$. Since $W$ is connected we know that 
$\Omega_1$ is simply connected. We claim that
\begin{equation}\label{eq:Omega_disjoint}
\bd\Omega_1, f(\bd\Omega_1), \ldots, f^p(\bd\Omega_1) \quad \text{are pairwise
disjoint}.
\end{equation}
To see the claim it is enough to notice that $\bd\Omega_1, f(\bd\Omega_1), \ldots, f^{p-1}(\bd\Omega_1)$ are in different Fatou
components. Moreover, $\bd\Omega_1 \subset \overline{W}\subset 
U_0$, so by Corollary~A' we get
\begin{equation}\label{eq:bdOmega}
f^p(\bd \Omega_1) \subset f^p(\overline{W}) \subset
f^p(W) \subset \C \setminus \overline{\Omega_1}. 
\end{equation}

Now we proceed like in the proof of Lemma~\ref{lem:U_0}.
By \eqref{eq:Omega_disjoint}, we have
$f(\bd\Omega_1)
\subset \Omega_1$, $\overline{\Omega_1}
\subset \ext(f(\bd \Omega_1))$ or $\overline{\Omega_1} \subset K(
f(\bd\Omega_1))$. 
In the first case, by \eqref{eq:bdOmega}
we have $p > 1$ and there exists $m \in \{2, \ldots, p\}$ such that
$f(\bd\Omega_1), \ldots, f^{m-1}(\bd\Omega_1) \subset \Omega_1$ and
$f^m(\bd\Omega_1)
\cap \overline{\Omega_1} = \emptyset$. Hence, $f$ has a weakly repelling fixed
point by Proposition~\ref{mapin} applied to $\Omega_1$. In the second case we use
Corollary~\ref{cor:mapout1} for $X = \bd \Omega_1$. Thus, we can assume that the third possibility takes place, i.e.
\[
\overline{\Omega_1} \subset K(f(\bd \Omega_1))
\]
Note that this implies 
\[
p = 1,
\]
because if $p > 1$, then $\Omega_1 \subset U_0$ and $f(\bd\Omega_1) \cap U_0 =
\emptyset$, which contradicts the fact that $U_0$ is connected and unbounded.

Let
\[
\NN = \{n \ge 0: p_0 \text{ is contained in a bounded component of } \C \setminus \overline{f^n(W)}\}.
\]
Note that $0 \in \NN$, so $\sup \NN$ is well defined.
We consider two further subcases.

\subsubsection*{Case $($\rm{i}$)$: $\sup\NN = N < \infty$}

\

Then $p_0$ is contained in a bounded
component $\Omega_2$ of $\C \setminus \overline{f^N(W)}$ but is not contained in any bounded component of
$\C \setminus \overline{f^{N+1}(W)}$. Moreover, by Corollary~A' we have
\[
f(\bd \Omega_2) \subset f(\overline{f^N(W)}) = 
f^{N+1}(\overline{W}) \subset
f^N(W) \subset \C \setminus \overline{\Omega_2}.
\]
This implies $\overline{\Omega_2}\subset \ext(f(\bd\Omega_2))$. 
Consequently, the assumptions of Corollary~\ref{cor:mapout1} 
are satisfied for $X = \bd \Omega_2$, and so $f$ has a weakly repelling fixed
point. 

\subsubsection*{Case $($\rm{ii}$)$: $\sup\NN = \infty$}

\

Fix some point $z_0\in\C$, which is not a pole of
$f$. By assumption and Corollary~A', for sufficiently large $n$ there
exists a bounded component $\Omega_3$ of $\C \setminus \overline{f^n(W)}$
containing $p_0, z_0, f(z_0)$, such that
\[
f(\bd \Omega_3) \subset f(\overline{f^n(W)}) = f^{n+1}(\overline{W})
\subset f^n(W) \subset \C \setminus \overline{\Omega_3}. 
\]
Hence,
\[
\overline{\Omega_3} \subset D,
\]
where $D$ is a component of $\clC \setminus f(\bd \Omega_3)$. We have $z_0,
f(z_0) \in \Omega_3 \subset D$. Hence, $\Omega_3, D, z_0$ satisfy the
assumptions of Lemma~\ref{mapout}, so $f$ has a weakly repelling fixed point.
This ends the proof of Theorem~B in Case~1 ($W$ is multiply connected).

\subsection*{\bf Case 2. $\boldsymbol{W}$ is simply connected.}

\

By assumption, one of the domain $U_j$ is multiply connected, so like in the proof in Case~1, using Lemmas~\ref{poles-in-holes} and~\ref{lem:U_0} we can assume that there exists a curve 
\[
\gamma \subset U_0
\]
and a pole $p_0$ of $f$, such that $p_0 \in K(\gamma)$ (the difference with respect to the previous case is that  the curve $\gamma$ was taken in $W$). Let 
\[
\Gamma = \bigcup_{n=0}^\infty f^n(\gamma).
\]
Note that $p_0 \notin \Gamma$ and $f(\Gamma) \subset \bigcup_{n=1}^\infty
f^n(\gamma) \subset
\Gamma$. Moreover, $\Gamma$ is the union of $p$ disjoint sets
\[
\Gamma_j = \bigcup_{n=0}^\infty f^{pn+j}(\gamma) \subset U_j,
\]
for $j \in\{0,\ldots p-1\}$, such that $f(\Gamma_j) \subset \Gamma_{j+1 \text{ mod } p}$ and $f^{pn} \to
\zeta_j$ uniformly on $\Gamma_j$ as $n\to\infty$. In particular, this
implies that $\Gamma_0$ is a closed subset of $\C$. 

Define
\begin{multline*}
\NN = \{n \ge 0: p_0 \text{ is contained in a bounded simply connected domain}\\\text{with boundary in } f^n(\Gamma_0)\}.
\end{multline*}
Since $p_0 \in K(\gamma) \setminus \gamma$ and $\gamma \subset \Gamma_0$, we have $0 \in \NN$, so $\sup \NN$ is well defined. By Lemma~\ref{lem:U_0}, we can reduce the proof to the case, when the following holds:
\begin{equation}\label{eq:N}
\text{for every } n \in \NN \text{ there exists } N \in \NN \text{ such that } N \geq n \text{ and } f^N(\Gamma_0) \subset U_0.
\end{equation}

Suppose $\sup\NN = \infty$. Then \eqref{eq:N} implies that there are arbitrarily large $N$ such that $p_0$ is contained in a bounded simply connected domain with boundary in $f^N(\Gamma_0) \cap U_0$. By Corollary~A', this boundary is contained in $W$ for large enough values of $N$. This is a contradiction since $W$ is simply connected by assumption. 

Hence, $\sup \NN = N_0 < \infty$ and, again by \eqref{eq:N}, there exists a bounded simply connected domain $V$ with
\begin{equation}\label{eq:V}
\bd V \subset f^{N_0}(\Gamma_0) \subset \Gamma_0 \subset U_0,
\end{equation}
such that $p_0 \in V$ and $p_0$ is not contained in any bounded simply connected domain with boundary in $f^{N_0+1}(\Gamma_0)$.
Thus, we can define $E$ to be the bounded
component of $\C \setminus f^{N_0}(\Gamma_0)$, such that $p_0 \in E$ and 
$p_0$ is not in any bounded component of $f(E)$. Note that by \eqref{eq:V},
the set $f^{N_0}(\Gamma_0)$ is closed in $\C$ and so 
\[
\bd E \subset f^{N_0}(\Gamma_0).
\]

Let 
\[
\Omega = E \cup \bigcup\{K(\sigma): \sigma \text{ is a closed curve in } E\}.
\]
By definition, $\Omega$ is a bounded simply connected
domain in $\C$, such that $p_0\in \Omega$ and 
\[
\bd \Omega \subset \bd E \subset f^{N_0}(\Gamma_0) \subset U_0.
\]

We claim that for any given $n > 0$, one of the following must be satisfied:
\begin{equation}\label{eq:or}
f^n(\bd\Omega) \cap \Omega = \emptyset \quad \text{or}  \quad f^n(\bd\Omega)
\subset \Omega.
\end{equation}
To see this observe that if $n \neq\ell p$ for all $\ell > 0$, then
$f^n(\bd\Omega) \cap \bd\Omega = \emptyset$, so \eqref{eq:or} holds due to the
connectedness of $\bd\Omega$ and $f^n(\bd\Omega)$. If $n =\ell p$ for some $\ell >
0$, then $f^n(\bd\Omega) \subset f^{N_0}(\Gamma_0)$, so $f^n(\bd\Omega)$
is disjoint from $E$. Hence, if $f^n(\bd\Omega)$
intersects $\Omega$, then $f^n(\bd\Omega) \cap K(\sigma) \neq \emptyset$ for a closed 
curve $\sigma \subset E$, so in fact $f^n(\bd\Omega)\subset K(\sigma) \subset \Omega$. This shows \eqref{eq:or}.

Using \eqref{eq:or}, we conclude that one of the following three cases holds: $\Omega \subset K(f(\bd\Omega))$, $\Omega \subset \ext(f(\bd\Omega))$ or $f(\bd\Omega) 
\subset \Omega$. The first case is not possible since it would imply that
$p_0$ is in a bounded simply connected domain with boundary in $f^{N_0+1}(\Gamma_0)$, 
which contradicts the definition of $N_0$. The second case implies that the assumptions of Corollary~\ref{cor:mapout1} are satisfied for $X = \bd\Omega$ (by Torhorst's Theorem~\ref{theorem:torhorst}, $\bd \Omega$ is locally connected; moreover, $f$ has no fixed points in $\bd\Omega$), so $f$ has a
weakly repelling fixed point. Hence, the remaining case is 
\[
f(\bd\Omega) \subset \Omega.
\]

By \eqref{eq:or} and the fact that $f^{pn} \to \infty$ as $n\to\infty$ uniformly on
$\bd\Omega$, there exists a (minimal) number $m \geq 2$ such that 
\begin{equation}\label{eq:f^j}
f(\bd\Omega), \ldots, f^{m-1}(\bd\Omega)\subset \Omega  \quad \text{and} \quad f^m(\bd\Omega) \cap\Omega = \emptyset.
\end{equation}
If $f^m(\bd\Omega) \cap\overline{\Omega} = \emptyset$, the domain $\Omega$ satisfies the assumptions of 
Proposition~\ref{mapin}, so $f$ has a weakly repelling fixed point.
Hence, we are left with the case $f^m(\bd\Omega) \cap\bd\Omega \neq \emptyset$, which implies $m=\ell p$ for a certain $\ell > 0$ and, consequently, $f^m(\bd\Omega) \subset U_0$.

In this case we will see that we can slightly modify the domain $\Omega$ to a new domain $\Omega^{\prime}$, so that $\Omega^{\prime}$ satisfies the condition \eqref{eq:f^j} and $f^m(\bd\Omega^{\prime}) \cap \overline{\Omega^{\prime}} = \emptyset$. Then Proposition~\ref{mapin} applies to $\Omega^{\prime}$ and $f$ has a weakly repelling fixed point.

To define the set $\Omega^{\prime}$ with the desired conditions, let
\[
D_\varepsilon=\{z \in U_0: \varrho_{U_0}(z, \bd \Omega) \leq \varepsilon\}
\]
for a small $\varepsilon >0$. Then $D_\varepsilon$ is a compact subset of $U_0$. It is immediate by \eqref{eq:f^j}, that if $\varepsilon$ is small enough, then all sets $f(\bd\Omega), \ldots, f^{m-1}(\bd\Omega)$ are contained in the same bounded component $\Omega^{\prime}$ of $\Omega \setminus D_\varepsilon$, such that
$\overline{\Omega^{\prime}}\subset \Omega$. Since $\bd\Omega$ is connected, the set $D_\varepsilon$ is also connected and, consequently, $\Omega^{\prime}$ is simply connected. Moreover, 
\begin{equation}\label{eq:bdOmega'}
\bd\Omega' \subset \{z \in U_0:\varrho_{U_0}(z, \bd \Omega) = \varepsilon\}
\end{equation}
and, since $\overline{\Omega^{\prime}}\subset \Omega$ and $f^m(\bd\Omega) \cap\Omega = \emptyset$, we have
\begin{equation}\label{eq:dist_Omega}
\varrho_{U_0}(z, w) \geq \varepsilon \quad \text{for every } z
\in \overline{\Omega'} \cap U_0 \text{ and } w \in f^m(\bd \Omega)
\end{equation}
(otherwise, connecting $z$ to $w$ in $U_0$ by a curve $\kappa$ of hyperbolic length smaller than $\varepsilon$, we would find $z' \in \bd\Omega' \cap \kappa$ and $w' \in \bd\Omega \cap \kappa$ such that $\varrho_{U_0}(z', w')  < \varepsilon$, which contradicts \eqref{eq:bdOmega'}).

As $f^m$ maps $U_0$ into itself, Schwarz--Pick's Lemma \ref{lemma:schwarz_pick} implies that for $z \in \bd\Omega'$ and $w \in \bd\Omega$ we have
\begin{equation}\label{eq:rho}
\varrho_{U_0}(f^m(z), f^m(w)) \leq \varrho_{U_0}(z, w),
\end{equation}
with strict inequality unless a lift of $f^m$ to a universal cover of $U_0$
is a M\"obius transformation. Suppose the inequality in \eqref{eq:rho} is not
strict. Then the first assumption of Lemma~\ref{lem:cover} is satisfied for $U =
U_0$ and $F = f^m$, while the additional assumption of this lemma is also fulfilled since $W$ is simply connected. Hence, by Lemma~\ref{lem:cover} we conclude that $U_0$ is simply connected,
a contradiction with $p_0 \in \Omega$ and $\partial \Omega \subset U_0$. 

Therefore, the inequality in \eqref{eq:rho} is strict, and by the compactness of $\bd \Omega$ we have
\begin{equation}\label{eq:epsilon}
\varrho_{U_0}(f^m(z), f^m(\bd\Omega)) < \varrho_{U_0}(z, \bd\Omega) =
\varepsilon \quad \text{for every } z \in \bd\Omega^{\prime}.
\end{equation}
This together with \eqref{eq:dist_Omega} implies
\[
f^m(\bd\Omega') \cap \overline{\Omega'} = \emptyset.
\]
Note also that if $\varepsilon$ is sufficiently small, then by \eqref{eq:f^j},
\[
f(\bd\Omega'), \ldots, f^{m-1}(\bd\Omega')\subset \Omega'.
\]
Hence, the assumptions of Proposition~\ref{mapin} are satisfied for $\Omega'$, and $f$
has a weakly repelling fixed point. This concludes the proof in Case~2 ($W$ is simply connected) and, in fact, the proof of Theorem~B.

\section{Proof of Theorem C} \label{sec:proof_thmC}

In what follows we assume that $f:\C\to \clC$ is a meromorphic map with a cycle of Herman rings $U_0, \ldots , U_{p-1}$ for some $p > 0$. Then there exists a biholomorphic map 
\[
\psi: U_0 \to \{z: 1/r<|z|<r\}
\]
for some $r > 1$, such that
$\psi \circ f^p \circ \psi^{-1} = R_\alpha$, where $R_\alpha(z)= e^{2\pi i \alpha} z$ and $\alpha \in \R\setminus \Q$. 

Herman rings are multiply connected by definition. The goal is to show that in this setup, $f$ must have a weakly repelling fixed point. Let
\[
\gamma = \psi^{-1}(\{z: |z| = 1\}).
\]
Then $\gamma$ is a Jordan curve in $U_0$. If $p=1$, then Lemma~\ref{mapout} 
applies to $\Omega=\inter(\gamma)$, and $f$ has a weakly repelling fixed point.
Hence, in what follows we assume $p>1$ and, consequently, 
$\gamma$ is a Jordan curve in $U_0$ such that
$\gamma, f(\gamma), \ldots, f^{p-1}(\gamma)$ are pairwise disjoint, $f^p(\gamma)
= \gamma$ and $\inter(\gamma) \cap J(f) \neq \emptyset$. By
Lemma~\ref{poles-in-holes}, the map $f$ has a pole $p_0$ in $\inter(f^j(\gamma))$ for some $0\leq j \leq p-1$. Without loss of
generality we assume that $j=0$, i.e. $p_0\in  \inter
(\gamma)$.

Next we discuss different relative positions of the above curves to see that the results in Section~\ref{section:configurations} imply that $f$ has a weakly repelling fixed point unless one situation occurs.  
In this case, to show the existence of a weakly repelling fixed point we will use a surgery argument, like in Shishikura's Theorem~\ref{shishikura}.

Observe that for all $j \geq
0$, we have
$f^j(\gamma) \subset \inter(f^{j+1}(\gamma))$ or $f^j(\gamma) \subset
\ext(f^{j+1}(\gamma))$. Since $f^p(\gamma) = \gamma$, we cannot have 
$f^j(\gamma) \subset \inter(f^{j+1}(\gamma))$ for all $j = 0, \ldots,
p-1$. Hence, there exists a minimal number $j_0 \in \{0, \ldots,
p-1\}$ such that $f^{j_0}(\gamma) \subset
\ext(f^{j_0+1}(\gamma))$. Set 
\[
\sigma_0 = f^{j_0}(\gamma) \quad \text{and} \quad \sigma_j = f^j(\sigma_0),\ j\geq 1
\]

By definition, $\sigma_0, \sigma_1, \ldots, \sigma_{p-1}$ are
pairwise disjoint and $\sigma_p = \sigma_0$. Moreover, $\sigma_0
\subset \ext(\sigma_1)$ and
$p_0 \in \inter(\sigma_0)$, by the minimality of $j_0$. 

Suppose first that $\sigma_1 \subset \ext(\sigma_0)$. Then $\inter(\sigma_0) \subset
\ext(\sigma_1)$, so by Corollary~\ref{cor:mapout1} for $X = \sigma_0$, the
map $f$ has a weakly repelling fixed point. Hence, we can assume 
\begin{equation}\label{eq:deleted_kr}
\sigma_1 \subset \inter(\sigma_0). 
\end{equation}
If there exists $j \in \{2, \ldots, p-1\}$ such that $\sigma_j \subset \ext
(\sigma_0)$, then the assumptions of Proposition~\ref{mapin} are satisfied for $\Omega
= \inter(\sigma_0)$, so $f$ has a weakly repelling fixed point. Therefore, from now on we suppose that 
\begin{equation} \label{eq:laprimera}
\sigma_j \subset \inter(\sigma_0) \quad \text{for }j = 1, \ldots, p-1.
\end{equation}
Suppose now that there exists $j \in \{1, \ldots, p-1\}$ such that
$\sigma_{j+1} \subset \inter(\sigma_j)$. Then the assumptions of
Proposition~\ref{mapin} are satisfied for $\Omega
= \inter(\sigma_j)$, so $f$ has a weakly repelling fixed point. Thus, we
can assume that 
\begin{equation}\label{eq:subset_ext}
\sigma_{j+1} \subset \ext(\sigma_j) \quad \text{for }j = 1, \ldots, p-1.
\end{equation}
If there exists $j \in \{1, \ldots, p-2\}$ such that
$\sigma_j \subset \inter(\sigma_{j+1})$ then, by \eqref{eq:subset_ext}, 
the assumptions of Corollary~\ref{cor:mapout2} are satisfied for $X =
\sigma_j$, so $f$ has a weakly repelling fixed point. Hence, we may also suppose
that  $\sigma_j \not\subset \inter(\sigma_{j+1})$, so
\begin{equation} \label{collons}
\inter(\sigma_j) \subset \ext(\sigma_{j+1}) \quad \text{\  for $j = 1,
\ldots,p-2$}.
\end{equation}
By Corollary~\ref{cor:mapout1} for $X = \sigma_j$, and using 
\eqref{collons} we may assume that $f$ has no poles in $\inter(\sigma_j)$ for $j =
1, \ldots, p-2$. Consequently, 
\begin{equation}\label{eq:f(K)}
f(\inter(\sigma_j)) =\inter(\sigma_{j+1}) \quad \text{for } j = 1, \ldots,
p-2.
\end{equation}
We claim that we can also reduce the proof to the case where 
\begin{equation} \label{eq:disjoint}
\inter(\sigma_1), \ldots, \inter(\sigma_{p-1}) \quad \text{are pairwise disjoint
subsets of } \inter(\sigma_0).
\end{equation}
To see this suppose otherwise, i.e.~there exist $k>0$ and $m>1$ with $k+m\leq
p-1$, such that $\sigma_{k+m} \subset \inter(\sigma_k)$ or $\sigma_k \subset 
\inter(\sigma_{k+m})$. Observe that $m=1$ is not possible by \eqref{collons}.

In the first case, observe that  by \eqref{eq:f(K)}, $f^{p-k-m}
(\inter(\sigma_k)) = \inter(\sigma_{p-m}) \subsetneq \inter(\sigma_0)$. Since $\sigma_{k+m}
\subset \inter(\sigma_k)$, we have $\sigma_0=f^{p-k-m} (\sigma_{k+m}) \subset
\inter(\sigma_{p-m})$, which again is not possible.

In the second case, again by \eqref{eq:f(K)}, $f^{p-k-m-1}
(\inter(\sigma_{k+m})) = \inter(\sigma_{p-1})$. Since $\sigma_k \subset \inter(\sigma_{k+m})$
we have $\sigma_{p-m-1}= f^{p-k-m-1}(\sigma_k) \subset \inter(\sigma_{p-1})$.
Hence, there exists $z_0\in \inter(\sigma_{p-1})$ such that $f(z_0 )\in
\inter(\sigma_0)$. Then, Lemma~\ref{mapout} with $\Omega=\inter(\sigma_{p-1})$
and $D=\inter(\sigma_0)$ provides the existence of a weakly repelling fixed point of $f$.

Hence, we may assume \eqref{eq:disjoint}. By Lemma~\ref{mapout} applied exactly
as above we may also suppose that 
\begin{equation}\label{eq:f(K)'}
f(\inter(\sigma_{p-1})) = \overline{\ext(\sigma_0)}.
\end{equation}
Finally, suppose that $f(\inter(\sigma_0)) \supset \inter(\sigma_1)$, which
together with \eqref{eq:f(K)} implies that $f(\inter(\sigma_0)) =\Ch$. By
considering a preimage of $D=\inter(\sigma_0)$ compactly contained inside $D$,
and applying Lemma~\ref{mapout}, it follows again that $f$ has a weakly
repelling fixed point. Hence, from now on we also suppose that 
\begin{equation} \label{eq:laultima}
\inter(\sigma_1) \not \subset  f(\inter(\sigma_0)),
\end{equation}
which implies that there exists a neighborhood $N$ of $\inter(\sigma_0)$ such that $f(N\cap \ext(\sigma_0)) \subset \inter(\sigma_1)$.
 
At this point we work under the assumptions
\eqref{eq:deleted_kr}--\eqref{eq:laultima}, as shown in
Figure~\ref{fig:finalsetup}. 

\begin{center}
\begin{figure}[hbt!]
\def\svgwidth{0.4\textwidth}
\begingroup%
  \makeatletter%
  \providecommand\color[2][]{%
    \errmessage{(Inkscape) Color is used for the text in Inkscape, but the package 'color.sty' is not loaded}%
    \renewcommand\color[2][]{}%
  }%
  \providecommand\transparent[1]{%
    \errmessage{(Inkscape) Transparency is used (non-zero) for the text in Inkscape, but the package 'transparent.sty' is not loaded}%
    \renewcommand\transparent[1]{}%
  }%
  \providecommand\rotatebox[2]{#2}%
  \ifx\svgwidth\undefined%
    \setlength{\unitlength}{440.075bp}%
    \ifx\svgscale\undefined%
      \relax%
    \else%
      \setlength{\unitlength}{\unitlength * \real{\svgscale}}%
    \fi%
  \else%
    \setlength{\unitlength}{\svgwidth}%
  \fi%
  \global\let\svgwidth\undefined%
  \global\let\svgscale\undefined%
  \makeatother%
  \begin{picture}(1,0.59157424)%
    \put(0,0){\includegraphics[width=\unitlength]{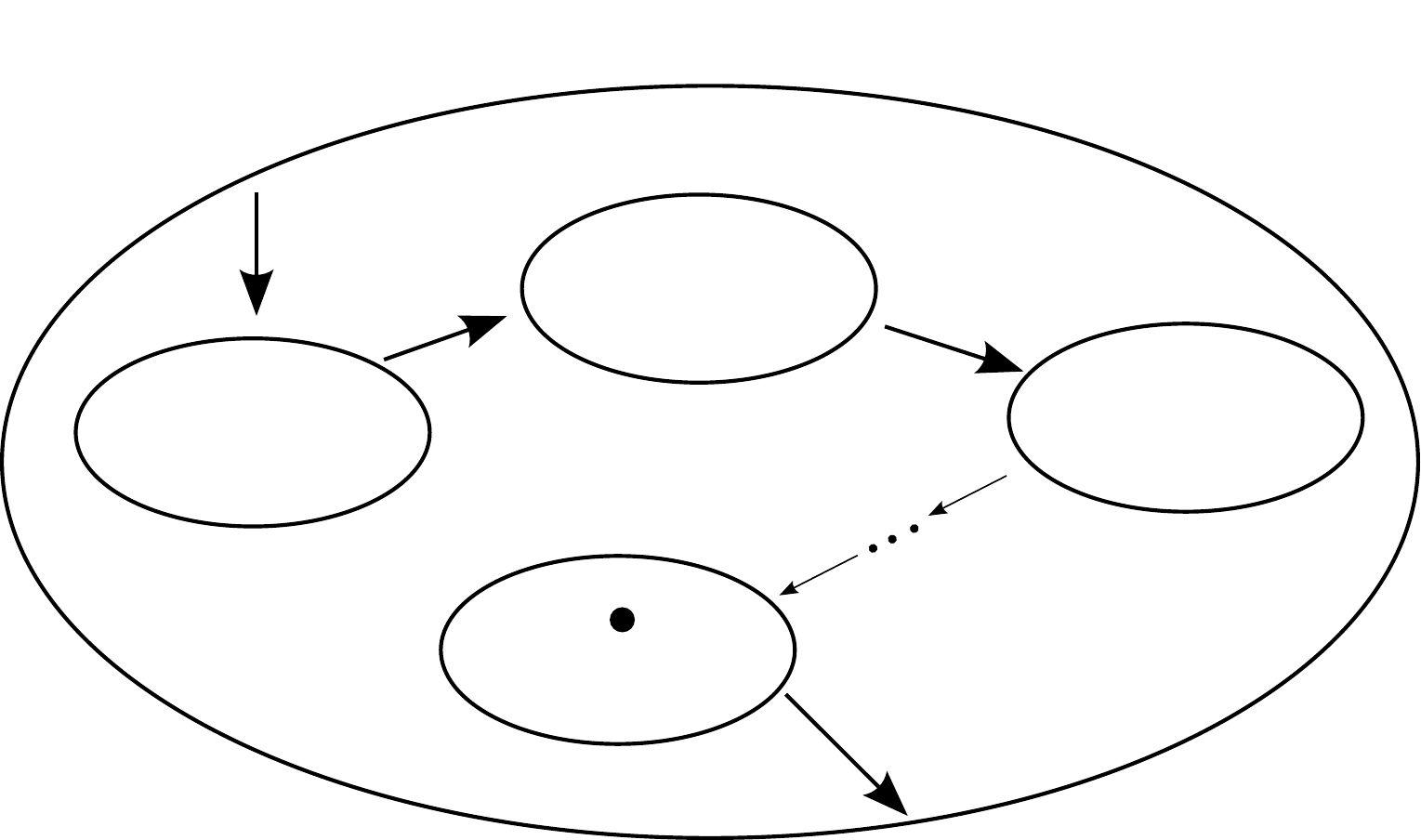}}%
    \put(0.38226793,0.10733605){\color[rgb]{0,0,0}\makebox(0,0)[lb]{\smash{$p_0$}}}%
    \put(0.52405361,0.56283216){\color[rgb]{0,0,0}\makebox(0,0)[lb]{\smash{$\sigma$}}}%
    \put(0.07458129,0.17236573){\color[rgb]{0,0,0}\makebox(0,0)[lb]{\smash{$\sigma_1$}}}%
    \put(0.40179823,0.2762441){\color[rgb]{0,0,0}\makebox(0,0)[lb]{\smash{$\sigma_2$}}}%
    \put(0.72641826,0.18015656){\color[rgb]{0,0,0}\makebox(0,0)[lb]{\smash{$\sigma_3$}}}%
    \put(0.43400059,0.02953285){\color[rgb]{0,0,0}\makebox(0,0)[lb]{\smash{$\sigma_{p-1}$}}}%
  \end{picture}%
\endgroup%
\caption{The final setup in the proof of Theorem~C.} 
\label{fig:finalsetup}
\end{figure}
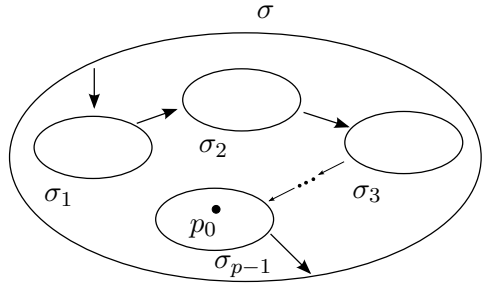
\end{center}

Observe that the situation is reminiscent of the setup of Shishikura's Theorem~\ref{shishikura} for $V_0=\ext(\sigma_0)$, $V_1=\inter(\sigma_1)$ and $k=p-1$,
except for one hypothesis, namely $f^{k}(\overline{V}_1) \subset V_0$, which
instead reads as $f^k(\overline{V_1})=\overline{V_0}$. 

We shall conclude the proof with an alternative surgery argument, which is a
particular case of Shishikura's surgery in \cite[Theorem~6]{shishikura2}. The idea is to
convert the $p$-cycle of Herman rings into a $p$-cycle of Siegel discs, by
gluing a rigid rotation in $\ext(\sigma_0)$ (for the $p$-th iterate). This will
provide the existence of a weakly repelling fixed point in $\inter(\sigma_0) \setminus
\bigcup_{j=1}^{p-1} \inter(\sigma_j)$. 

We sketch the details for completeness. Redefine the cycle of Herman rings so
that $\sigma_0 \subset U_0$. Then  
$ \sigma_0 = \psi^{-1}(\{z: |z| = 1\})$. Since $\psi|_{\sigma_0}$ is real analytic,
there exists a quasiconformal homeomorphism
\[
\Psi:\overline{\ext(\sigma_0)}  \to \Ch\setminus \D
\]
such that $\Psi=\psi$ on $\sigma_0$. We now define $h:\overline{\ext(\sigma_0)} \to
\overline{\ext(\sigma_0)}$ as 
\[ 
h = \Psi^{-1}  \circ R_\alpha \circ \Psi.
\]
Note that $h^n= \Psi^{-1}  \circ R_\alpha^n \circ \Psi$ and therefore $h^n$ is uniformly quasiregular for all $n>0$. 

Since $f^p$ is conjugate to $R_\alpha$ on $\sigma_0$, it follows that $f$ has
degree one on $\sigma_j$ for all $j=1,\ldots, p$. Together with \eqref{eq:f(K)}
and \eqref{eq:f(K)'}, this implies that for all $j=1,\ldots,p-1$, the map
$f|_{\inter(\sigma_j)}$ is univalent and  hence it has a univalent inverse. We now
define a new map on the Riemann sphere as follows:
\[
F=\begin{cases}
f & \text{\ on $\overline{\inter(\sigma_0)}$}\\
\left(f\mid_{\inter(\sigma_1)} \right)^{-1} \circ \cdots \circ 
\left(f\mid_{\inter(\sigma_{p-1})} \right)^{-1} \circ h & \text{\ on $\ext(\sigma_0)$}.
\end{cases}
\]
Note that $F^p|_{\ext(\sigma_0)}=h$ and  $F$ is holomorphic everywhere except on
$\ext(\sigma_0)$, where it is quasiconformal. 
Now we define a conformal structure $\mu$ on $\clC$ setting
\[
\mu = \begin{cases}
(\Psi^{-1})^* \mu_0 & \text{on } \ext(\sigma_0)\\
\left(\left(f\mid_{\inter(\sigma_j)} \right)^{-1} \circ \cdots \circ 
\left(f\mid_{\inter(\sigma_{p-1})} \right)^{-1}\right)^* \mu & \text{on }
\inter(\sigma_j) \text{ for } j = 1, \ldots, p-1\\
\mu_0 & \text{elsewhere},
\end{cases}
\]
where $\mu_0$ is the standard structure. Then $\mu$ is bounded and
$F$-invariant, so by the Measurable Riemann Mapping Theorem, $F$ is
quasiconformally conjugate to a rational map $g$, under a quasiconformal
homeomorphism $\phi:\clC\to\clC$.

One can check that on some neighborhood of $\phi(\overline{\ext(\sigma_0)})$
the map $\Psi \circ \phi^{-1}$ is conformal and conjugates $g^p$ to $R_\alpha$.
Hence, $g$ has a $p$-cycle of Siegel discs containing
$\phi(\overline{\ext(\sigma_0)}) \cup \phi(\inter(\sigma_1)) \cup \cdots \cup
\phi(\inter(\sigma_{p-1}))$. Since $g$ is rational, it has a weakly repelling fixed
point, which cannot lie in the Siegel cycle. But $g$ is conformally
conjugate to $f$ everywhere else. Hence, $f$ has a weakly repelling fixed point.
This concludes the proof of Theorem~C.

%%%%%%%%%%%%%%%%%%%%%%%%%%%%%%%%%%%%%%%%%%
%%%%%%%%%%%%%%%%%%%%%%%%%%%%%%%%%%%%%%%%%%%%%
%%%%%%%%%%%%%%%%%%%%%%%%%%%%%%%%%%%%%%%%%%%%%

\bibliography{baker}

\end{document}